\documentclass{article}
\usepackage[utf8]{inputenc}
\usepackage{preamble}

\title{$2$-dimensional TFTs via modular $\infty$-operads}
\author{Jan Steinebrunner}
\date{
Lecture series as part of the workshop\\
``Cobordism categories: building bridges between algebra and geometry''\\
Copenhagen, 17-21 February 2025 \\ 
{\small version of \today}
}

\usepackage[style=alphabetic, 
maxbibnames=4, backend=biber, 
doi=false, isbn=false]{biblatex}
\begin{document}

\maketitle

This lecture series is based on joint work in progress with Shaul Barkan \cite{modular, cyclic},
as well as work in progress of the author \cite{genus}.
The five sections of these notes correspond to the five lectures, but more details have been added.

\begin{abstract}
    $2$-dimensional topological field theories ($2$D TFTs) valued in vector spaces are commutative Frobenius algebras.
    The goal of this lecture series is to generalize from the $1$-category of vector spaces to any symmetric monoidal $\infty$-category $\mathcal{C}$, i.e.~to study symmetric monoidal functors $\mathrm{Bord}_2 \to \mathcal{C}$.
    Choosing $\mathcal{C}$ to be the $(2,1)$-category of linear categories, this recovers a definition of modular functors, and choosing it to be the derived category of a ring yields a notion closely related to cohomological field theories.

    We will introduce a notion of modular $\infty$-operads and algebras over them, construct the modular $\infty$-operad of surfaces $\mathcal{M}$, and show that algebras over $\mathcal{M}$ in $\mathcal{C}$ are exactly $2$D TFTs valued in $\mathcal{C}$.
    Along the way we will encounter variants modular $\infty$-operads (such as cyclic $\infty$-operads and $\infty$-properads) as well as a proof of the $1$D cobordism hypothesis with singularities.
    This uses some (mild) $\infty$-category theory, but no familiarity with ($\infty$-)operads will be assumed.
    
    The main goal will be to filter $\mathcal{M}$ by genus to obtain an obstruction-theoretic description of $2$D TFTs with general target.
    Applying this to invertible TFTs one can construct a new spectral sequence exhibiting relations between the cohomology groups of moduli spaces of curves.
\end{abstract}

\subsubsection*{Acknowledgements}
As these notes are based on several projects, there is a long list of people who ought to receive thanks for their support and contributions, first and foremost Shaul Barkan, who co-authored \cite{properads,modular,cyclic} and who helped to build the foundations of the story told here.
More precise acknowledgements will follow in the respective papers, but
I would like to already thank the organizers and the participants of the masterclass in Copenhagen and my test-audience in Cambridge, for their curiosity and all their the insightful questions and ideas!
I am grateful to GeoTop (DNRF151) for funding the masterclass.

\newpage
\setcounter{tocdepth}{2}
\tableofcontents
\newpage

\section{2D TFTs -- 1-categorical and \captioninfty-categorical}

This lecture is a brief introduction to ($(\infty,1)$-categorical) topological field theories (TFTs) with a bias towards the kind of examples that we can later study using modular operads.
Readers who are new to the subject are encouraged to read Atiyah's (short!) original paper \cite{atiyah1988topological}, which explains how TFTs arise from physics, and Freed's excellent survey \cite{Freed-CH}, which explains both the physics background and why we might want to get higher categories involved.

I have also made an attempt to sketch what $(\infty,1)$-categories are, in the hope that this is sufficient for those readers, who are unfamiliar with them, but are familiar with basic homotopy theoretic concepts, to follow the rest of the lectures.
However, if you truly want to learn about $(\infty,1)$-categories this will not suffice.

\subsection{The 1-categorical story}

\subsubsection{The definition of TFTs.}
According to Atiyah \cite{atiyah1988topological} and Segal \cite{Segalconformal} a \hldef{$d$-dimensional topological quantum field theory}%
    \footnote{In fact, according to Atiyah's convention this would be a TQFT in dimension $d+1$. In these lecture notes we'll stick to the more common convention of letting $d$ denote the dimension of the top-dimensional manifolds that appear.}
should assign to a closed oriented $(d-1)$-manifold $M$ a vector space $\Zcal(M)$ -- the space of states associated to $M$.
Given a $d$-dimensional bordism $W\colon M \to N$, i.e.~a compact oriented $d$-manifold $W$ and an identification of its boundary as $\partial W \cong M \sqcup N^-$, the TQFT should assign a linear map $\Zcal(W)\colon \Zcal(M) \to \Zcal(N)$, which we can think of the time-evolution of states along $W$.
These assignments are subject to various axioms and in fact, also require additional coherence data.
All of this data, coherence, and axioms, can be assembled in to symmetric monoidal functor
\[
    \Zcal\colon \Bord_d \too \Vect_k.
\]
Here \hldef{$\Bord_d$} is the $d$-dimensional, oriented bordism category.
Its objects are closed oriented $(d-1)$-manifolds, and the morphisms $M \to N$ are equivalence classes of pairs $(W,i)$ of a compact oriented $d$-manifold $W$ and an orientation-preserving diffeomorphism $i\colon \partial W \cong M \sqcup N^-$.
(Two such pairs $(W,i)$ and $(W',i')$ are equivalent if there is an orientation preserving diffeomorphism $\varphi\colon W \cong W'$ such that $i' \circ \varphi = i$.)
$\Bord_d$ has a symmetric monoidal structure with $\sqcup\colon \Bord_d \times \Bord_d \to \Bord_d$ given by disjoint union of both objects and morphisms.
On the other side $\Vect_k$ is the category of vector spaces over some field $k$, and it has the symmetric monoidal structure given by the tensor product over $k$.
For $\Zcal$ to be a symmetric monoidal functor it thus has to come with coherence isomorphisms
\[
    \Zcal(M \sqcup N) \cong \Zcal(M) \otimes \Zcal(N).
\]
This is motivated by the quantum mechanical principle that a quantum state of a disjoint union of systems is not simply a pair of states, one in each system, but rather an element of the tensor product of the state spaces.
(In particular, not every state is an elementary tensor of the form $|\phi\rangle \otimes |\psi \rangle$. This is the basis of quantum entanglement.)

\subsubsection{Commutative Frobenius algebras.}
In dimension $2$ we can explicitly classify TQFTs: there is an equivalence between the groupoid%
    \footnote{
        Note that if $\alpha \colon \Zcal \Rightarrow \Zcal'$ is a symmetric monoidal natural transformation between two TQFTs, then $\alpha$ is automatically invertible.
        Thus, the category of $d$-dimensional TQFTs is a groupoid.
    }
of $2$D TQFTs and the groupoid of commutative Frobenius algebras over $k$.
A commutative Frobenius algebra is an algebra $A$ together with a non-degenerate trace $\tau\colon A \to k$, i.e.~a linear functional such that for all $a \in A \setminus \{0\}$ there is a $b \in B$ with $\tau(a\cdot b) \neq 0$.
In other words, the trace $\tau$ is non-degenerate if and only if the resulting pairing
\[
    \langle a, b \rangle \coloneq \tau(a \cdot b)
\]
is non-degenerate in the sense that it induces an isomorphism $A^\vee \cong A$.
Historically, Frobenius algebras were first studied in the context of representation theory where they are usually not commutative \cite{Nakayama1939}.
(Here non-degeneracy of $\tau$ is defined by requiring that $\tau^{-1}(0)$ contains no non-trivial left (or right) ideal.)
An example of such a Frobenius algebra is the group ring $k[G]$ of a finite group $G$ with the trace given by picking out the coefficient of the unit: $\tau(\sum_{g \in G} a_g g) \coloneq a_e$.
A natural commutative example arises from certain manifolds.
\begin{example}\label{ex:H*M-Frobenius}
    If $M$ is a closed oriented $2n$-manifold whose rational cohomology is concentrated in even degrees, then the cohomology ring $H^*(M;\Qbb)$ is a commutative Frobenius algebra over $k=\Qbb$ with coevaluation $\tau(\alpha) \coloneq \alpha([M]) \in \Qbb$ for $\alpha \in H^{2n}(M;\Qbb)$. 
    Poincar\'e duality implies that this is indeed a non-degenerate trace.
    For example $A = H^*(\mathbb{CP}^n;\Qbb) = \Qbb[x]/x^{n+1}$ is a commutative Frobenius algebra with $\tau(\sum_i a_i x^i) = a_n$.
\end{example}

\subsubsection{Classifying 2D TQFTs.}
Given a $2$D TQFT $\Zcal \colon \Bord_2 \to \Vect_k$ let $A \coloneq \Zcal(S^1)$ denote its value on the circle.
As every closed oriented $1$-manifold is a disjoint union of circles and $\Zcal$ is symmetric monoidal, we can describe the value of $\Zcal$ on any $1$-manifold as
\[
    \Zcal(M) = \Zcal\left(\bigsqcup_{i=1}^n S^1\right) \cong \bigotimes_{i=1}^n \Zcal(S^1) = A^{\otimes n}.
\]
Every $2$-dimensional bordism can be built under composition and gluing from a few elementary pieces, so to fully encode the TQFT it will suffice to record its value on these pieces.
As illustrated in \cref{fig:Frobenius-generators} the pair of pants and disk bordism give us an algebra structure on $A$ with multiplication and unit
\[
    \mu\colon A \otimes A = \Zcal(S^1 \sqcup S^1) \too \Zcal(S^1) = A
    \qqand
    \nu\colon k = \Zcal(\emptyset) \too \Zcal(S^1) = A
\]
and reading the disk as a bordism in the other direction we obtain a trace map
\[
    \tau\colon A = \Zcal(S^1) \too \Zcal(\emptyset) = k.
\]
By composing and comparing bordisms, one checks that the algebra structure is unital, associative, and commutative, and that the trace is non-degenerate.
Thus, $(A,\tau)$ is a commutative Frobenius algebra.
We can define a groupoid $\mrm{CFrob}_k$ of commutative Frobenius algebras over $k$ by defining an isomorphism of Frobenius to be an isomorphism of $k$ vector spaces that commutes with all the structure.

\begin{figure}[ht]
    \centering
    \def\svgwidth{.6\linewidth}
    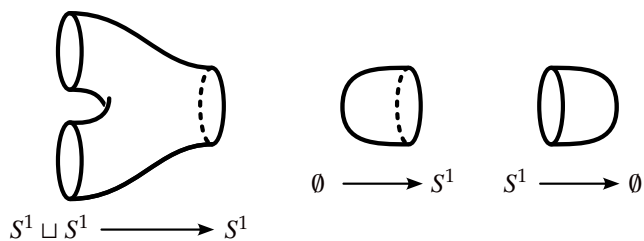
    \caption{The three bordisms that we use to describe the commutative Frobenius structure on $S^1$.}
    \label{fig:Frobenius-generators}
\end{figure}

In this form, the above (roughly) shows that there is a functor
\begin{align*}
    \Fun^\otimes(\Bord_2, \Vect_k) & \too \mrm{CFrob}_k\\
    \Zcal & \mapsto (\Zcal(S^1), \dots)
\end{align*}
from the category (in fact groupoid) of $2$-dimensional TQFTs to the groupoid of commutative Frobenius algebras.
As every $2$-bordism can be built from pairs of pants and disks we can recover the TQFT $\Zcal$ from the Frobenius algebra $A \coloneq \Zcal(S^1)$.
A more subtle statement is that indeed every commutative Frobenius algebra comes from a TQFT. To see this, one needs to show that the relations we imposed as part of the definition of a commutative Frobenius algebra indeed suffice to deduce all relations between composite $2$-bordisms.
This is often referred to as a folk theorem, but was first proven by \cite{Abrams1996}, see also the detailed exposition in \cite{Kock2003}.
\begin{thm}\label{thm:folk-2D-TQFT}
    The above functor is an equivalence between the groupoid of $2$D TQFTs and the groupoid of commutative Frobenius algebras.
\end{thm}

We make a few observations:
\begin{enumerate}
    \item 
    One might question to what extent \cref{thm:folk-2D-TQFT} really ``classifies'' $2$D TQFTs, as it just compares them to another notion.
    For example, it is not clear from the theorem whether there are any non-trivial TQFTs, or e.g.~how many TQFTs there are if we say fix $\dim_k \Zcal(S^1) = 4$.
    However, the theorem reduces a complicated infinite set of conditions (such as $\Zcal(W\cup V) = \Zcal(V) \circ \Zcal(W)$ for any pair of composable bordisms) to a just a handful of axioms.
    Moreover, the algebraic structure we reduce to has been studied before, albeit in a different form, in the context of representation theory. 
    \item 
    Given a commutative Frobenius algebra $A$ one can construct a comultiplication $\Delta\colon A \to A \otimes A$ on $A$ by using the pairing to dualize the multiplication.
    In the TQFT picture $\Delta$ is the image of the opposite pair of pants morphism.
    The trace $\tau$ is a counit for this comultiplication.
    One can thus also define a \hldef{commutative Frobenius algebra} as a $5$-tuple $(A, \mu, \nu, \Delta, \tau)$ where $(A,\mu, \nu)$ is a commutative unital algebra, $(A, \Delta, \tau)$ is a cocommutative counital coalgebra, and these structures are compatible in the sense that they satisfy the Frobenius relations
    \[
        (\mu \otimes \id_A) \circ (\id_A \otimes \Delta) 
        = \Delta \circ \mu
        = (\id_A \otimes \mu) \circ (\Delta \otimes \id_A) .
    \]
    \item 
    Its worth noting that all the generators and relations that we have considered so far are connected genus $0$ surfaces, i.e.~they are punctured $2$-spheres thought of bordisms in various ways.
    Thus, in some sense $\Bord_2$ is generated by genus $0$ surfaces and the relations between genus $0$ surfaces are also enough to at least classify all symmetric monoidal functors from $\Bord_2$ to $\Vect_k$.
\end{enumerate}

\subsubsection{Exploring other targets.}
So far all the TQFTs we considered were valued in the symmetric monoidal category $\Vect$ of vector spaces over $k$ for some field $k$.
The definition of TQFTs as symmetric monoidal functors suggests an immediate generalization for more general targets as follows.
For $\Ccal$ a symmetric monoidal category, we define a \hldef{$\Ccal$-valued $d$-dimensional TFT}%
    \footnote{
        We drop the ``Q'' from TQFT because what we do here is not necessarily related to quantum physics.
    }
to be a symmetric monoidal functor
\[
    \Zcal\colon \Bord_d \too \Ccal.
\]
As we shall see there are many sensible choices for what $\Ccal$ might be.

    Choosing $\Ccal = \Vect_k$ recovers the previous definition: a $d$-dimension TQFT is a $\Vect_k$-valued $d$-dimensional TFT.
    We also see that $k$ need not really be a field; any commutative ring will do.
    
    We could let $\Ccal = \mrm{sVect}_k$ be the symmetric monoidal category of super vector spaces (i.e.~$\Zbb/2$-graded vector spaces where the symmetric monoidal structure introduces a sign whenever we braid to odd dimensional vector spaces).
    An analogue of \cref{thm:folk-2D-TQFT} says that $\mrm{sVect}_k$-valued $2$D TFTs are classified by ``super commutative Frobenius algebras''.
    This is the natural setting for \cref{ex:H*M-Frobenius}:
    if $M$ is any closed oriented $2n$-manifold, then $H^*(M; \Qbb)$ is always a super commutative Frobenius algebra.

Indeed, a more refined version of \cref{thm:folk-2D-TQFT} says that $\Ccal$-valued $2$D TFTs are always classified by ``commutative Frobenius algebra objects'' in $\Ccal$.
This holds because the proof of \cref{thm:folk-2D-TQFT} does not use much about $\Vect_k$, but is rather a statement about generators and relations for the symmetric monoidal category $\Bord_2$.

Generalizing further, we can let $\Ccal$ be a symmetric monoidal \category{}.
We will discuss the various interesting choices of such $\Ccal$ in section \ref{subsec:higher-targets}.
First, however, we will need to make sense of $\Bord_d$ as a symmetric monoidal \category{}.

\subsubsection{Linear categories.}
In Segal's approach to conformal field theories \cite{Segalconformal} one encounters the notion of modular functor, which is a sort of categorification of a $2$D TFT.
The idea is to assign to $S^1$ not a vector space, but a $k$-linear category $\Zcal(S^1) = \Ccal$.
Then to a bordism $W\colon M \to N$ the modular functor should not assign a linear homomorphism, but rather a $k$-linear functor $\Zcal(W)\colon \Zcal(M) \to \Zcal(N)$.
For example, the pair of pants bordism $S^1 \sqcup S^1 \to S^1$ will induce a monoidal product $\otimes \colon \Ccal \otimes_k \Ccal \to \Ccal$.
However, as we have increased the category-level the prudent thing is not to stop at $1$-morphisms: 
the modular functor $\Zcal$ should also assign to every diffeomorphism $\varphi\colon W \cong W'$ between bordisms (that both have source and target $M \to N$) a natural isomorphism $\Zcal(\varphi)\colon \Zcal(W) \cong \Zcal(W')$ of linear functors.
Summarizing this in category-theoretic language we can say that a \hldef{modular functor} is a symmetric monodial $2$-functor
\[
    \Bord_2 \too \LinCat_k.
\]
Here $\Bord_2$ should then be a suitable symmetric monoidal $2$-category where we do not identify diffeomorphic bordisms, but we instead introduce the diffeomorphisms as $2$-morphisms. (Though, for now, we should at least identify isotopic diffeomorphisms.)
Note that this is really just a $(2,1)$-category, meaning that all its $2$-morphisms are invertible.
One could also allow non-invertible $2$-morphisms in the form of $3$-bordisms with corners; this leads to the notion of a once-extended $3$D TFTs, a key source of modular functors.

\begin{rem}\label{rem:LinCat-caveat}
    To make the above definition precise, we have to specify what exactly the $2$-category of linear categories is and which tensor product we choose on it. 
    For example, we might want to restrict to left/right exact functors and take the Deligne tensor product.
    Even then, this does not quite give the right definition of modular functor yet, as it does not allow for projective representations or equivalently for central extensions of the bordism category.
    For a detailed discussion of these issues, and in fact a classification of certain modular functors using like-minded ideas, see \cite{BrochierWoike-modular}.

    For suitable such definitions, the unit of $\LinCat_k$ is the linear category of $k$-vector spaces $\Vect_k$ and the category of (left exact) functors $\Vect_k \to \Vect_k$ is equivalent to the category of vector spaces itself, with the comparison
    \[
        \Fun(\Vect_k, \Vect_k) \too \Vect_k
    \]
    given by evaluating on $k$.
    Then, since the category $\Bord_2(\emptyset,\emptyset)$ is the groupoid of closed surfaces where morphisms are isotopy classes of diffeomorphisms, a modular functor gives us representation of the mapping class groups $\pi_0\Diff(\Sigma_g)$ for all $g$.
    (If we considered a central extension of the surface category, it gives us projective representations.)
\end{rem}

\subsection{\texorpdfstring{$\Bord_2$ as an $(\infty,1)$-category}{Bord2 as an infinity category}}
For the purpose of this lecture series I'll be a bit vague about the exact definition of an \category{}, but I'll try to be precise about the interface we can use to produce \categories{}.

\subsubsection{What are spaces?}
Let's first discuss what we want \categories{} to be.
An ordinary category is equivalently a category enriched in sets: for any two objects $x, y \in \Ccal$ we have a set $\Map_\Ccal(x,y)$ of morphisms/maps from $x$ to $y$.
Similarly, an \category{} should be a category enriched in spaces meaning that for any two objects $x, y \in \Ccal$ we have a space $\Map_\Ccal(x,y)$ of morphisms/maps from $x$ to $y$.
Here a ``space'' should be understood as a somewhat nebulous object that captures a homotopy type.
There are a few different concrete ways of thinking about spaces:
\[
    \Top = \{ \text{topological spaces} \}
    \sim \mrm{CW} = \{ \text{CW cpxs.} \}
    \sim \mrm{Kan} = \{ \text{Kan cpxs.} \}
    \sim \mrm{sSet} = \{ \text{simplicial sets} \}
\]
More precisely, the \category{} of spaces $\An$ is the \category{} obtained by taking any of above $1$-categories and universally inverting the suitable notion of weak equivalence.
So while we still don't precisely know what an \category{} is, this tells us that there is a functor%
    \footnote{
        One could refer to functors between \categories{} as $\infty$-functors, but this is unnecessary notation as $\infty$-functors recover ordinary functors in the sense that an $\infty$-functor between to $1$-categories is the same as an ordinary functor.
    }
of \categories{}
\[
    \Top \too \An
\]
that turns a topological space into a ``space''.
Note that this functor cannot have an inverse as it sends all contractible topological spaces to isomorphic objects in $\An$.

\subsubsection{What are \captioninfty-categories?}
Since we can turn topological spaces into spaces, we should also be able to turn categories enriched in topological spaces into categories enriched in spaces, i.e.~\categories{}.
Indeed, there is a functor
\[
    \ac\colon \TopCat \too \Cat
\]
that takes a topologically enriched category, i.e.~an ordinary category $\Ccal$ together with a topology on each set $\Map_\Ccal(x,y)$ such that composition is continuous, and turns it into an \category{}.
In fact, every \category{} can be obtained this way, and we can think of $\TopCat$ as our ``model'' for $\Cat$ as long as we have the right notion of weak equivalence, namely Dwyer-Kan equivalences.
\begin{defn}
    For a topologically enriched category $\Ccal$ we define its homotopy category $h(\Ccal)$ to be the ordinary category obtained by identifying any two morphisms $f,g\colon x \to y$ in $\Ccal$ that can be connected by a path in $\Map_\Ccal(x,y)$.
    A functor $F \colon \Ccal \to \Dcal$ of topologically enriched categories is a \hldef{Dwyer--Kan equivalence} if it satisfies
    \begin{enumerate}
        \item (fully faithful) 
        for any two $x,y \in \Ccal$ the continuous map $F\colon \Map_\Ccal(x,y) \to \Map_\Ccal(F(x),F(y))$ is a weak equivalence, and
        \item (essentially surjective)
        for all $z \in \Dcal$ there is an $x \in \Ccal$ such that $F(x)$ is isomorphic to $z$ in the homotopy category $h(\Dcal)$.
    \end{enumerate}
\end{defn}
Therefore, we can define $\Cat$ as the \category{} obtained from the ordinary category $\TopCat$ by inverting Dwyer--Kan equivalences.

\begin{war}
    While it is true that every \category{} can be represented by a topologically enriched category, the corresponding statement for functors is a bit more subtle.
    If we fix two topological categories $\Ccal$ and $\Dcal$ then not every $\infty$-functor $\ac(\Ccal) \to \ac(\Dcal)$ can be represented by a topologically enriched functor $\Ccal \to \Dcal$.
    However, this is true if we first replace $\Ccal$ and $\Dcal$ by certain Dwyer--Kan equivalent topological categories.
    
    This is one of the reasons why the model of ``quasicategories'' is so popular: every \category{} can be represented by a quasicategory and if I fix two quasicategories then every $\infty$-functor between them can be represented by a map of quasicategories.
    Some other reasons for the popularity are that they are quite easy to define (a quasi-category is a simplicial set that satisfies the combinatorial ``inner horn filling condition''), admit an easy construction of functor categories, and that Joyal and Lurie have used this model to prove most of the basic facts you'll need for every-day $\infty$-category theory.
    The reason we won't really be talking about quasicategories in this lecture series is that they take a while to get used to, and they obscure the idea that \categories{} should be categories enriched in spaces.
\end{war}

\subsubsection{The homotopy hypothesis.}
An \hldef{$\infty$-groupoid} is an \category{} where every morphism admits and inverse (up to homotopy).
Grothendieck's homotopy hypothesis postulates that an \groupoid{} should be the same as a space.
This is indeed the case for any reasonable implementation of spaces and \categories{}.
To make sense of this, consider the path-category construction
\[
    \Pi_{\infty}\colon \Top \too \TopCat \too \Cat
\]
where $\Pi_\infty(X)$ is the topologically enriched category whose objects are points of $X$ and where morphisms $x \to y$ are (Moore) paths from $x$ to $y$.
Concretely, the topological space of morphisms is 
\[
    \Map_{\Pi_\infty(X)}(x,y) \coloneq \{(t,\gamma) \in [0,\infty) \times \Map(\Rbb, X) \;|\; \gamma(s) = x \text{ for } s \le 0 \text{ and } \gamma(s) = y \text{ for } s \ge t\}.
\]
Such morphisms can be composed by $(t,\gamma) \circ (s,\delta) = (t+s, \gamma*\delta)$ and under this composition every path admits an inverse morphism (up to homotopy) given by the reverse path.
Thus, $\Pi_\infty(X)$ lands in the full subcategory $\mrm{Gpd}_\infty \subset \Cat$ of \groupoids{}.
This functor sends weak equivalences of topological spaces to equivalences of \categories{} and thus descends to a functor $\An \too \mrm{Gpd}_\infty$, which is an equivalence.

All of the above means that we can think of spaces both as topological spaces and as \categories{} where all morphisms are invertible.
In particular, this means that if we construct a topologically enriched \emph{groupoid} $G$, then $\ac(G) \in \Cat$ is an \groupoid{} and thus a space.
We can alternatively obtain this space by first taking the geometric realization $|N_\bullet G|$ of the nerve of $G$, which gives us a topological space that represents the same space as $\ac(G)$.
    
\begin{example}
    This construction will be quite useful to construct certain moduli spaces.
    For example, consider the topologically enriched groupoid $\Mfd_d^\cong$ where objects are closed oriented $d$-manifolds and the mapping spaces $\Map(M,N)$ are the spaces of (smooth) diffeomorphisms with the Whitney $\Ccal^\infty$-topology.
    Then the space $|\Mfd_d^\cong| = \ac(\Mfd_d^\cong) \in \An$ is the \hldef{moduli space of closed $d$-manifolds}.
    If we're worried about set theory, we might want to restrict the objects to be submanifolds of some $\Rbb^k$.
    This also hints at a way of directly constructing this moduli space as a topological space, namely we can define a topology on the set of closed oriented $d$-dimensional submanifolds of $\Rbb^\infty = \bigcup_{k\ge0} \Rbb^k$ by identifying it with the quotient space
    \[
        \coprod_{[M]} \Emb(M, \Rbb^\infty) / \Diff(M)
    \]
    where the coproduct runs over representatives for diffeomorphism classes of manifolds and $\Diff(M)$ acts by precomposition on the space of embeddings.
\end{example}

\subsubsection{Spaces of bordisms.}
We would now like to construct an \category{} $\Bord_d$ where for two closed oriented $(d-1)$-manifolds $M$ and $N$ the space $\Map_{\Bord_d}(M, N)$ is a space of bordisms from $M$ to $N$.
The simplest%
\footnote{
    Of course what is or isn't simple is subjective, and the reader might find the approach of \cite{GMTW}, where the authors construct a topological space of bordisms, more straightforward.
}
way of describing such a space is as a topological groupoid.
As we just discussed a topological groupoid yields an \groupoid{} and thus a space, and if we wanted to go through topological spaces we can simply take the geometric realization of the groupoid we're about to describe.
The topologically enriched groupoid $\Bord_d(M,N)$ has as objects tuples $(t, W)$ of $t\in (0,\infty)$ and a compact oriented $d$-dimensional sub manifold%
\footnote{
    In fact, we could also replace $\Rbb^\infty$ here by $\Rbb^{2d+1}$ because every $d$ manifold can be embedded into $\Rbb^{2d+1}$ by Whitney's embedding theorem.
    If you've seen similar definitions before, this might be confusing because $\Emb(W, \Rbb^{2d+1})$ is not contractible.
    The reason that this doesn't matter is that our space of morphisms is a topologically enriched groupoid and not a topological space. We therefore don't need that the space of embeddings $W \hookrightarrow \Rbb^\infty$ is contractible, we just need that the set of embeddings $W \hookrightarrow \Rbb^{2d+1}$ is non-empty.
}
$W \subset [0,t] \times \Rbb^\infty$ such that $\partial W = \{0\} \times M \sqcup \{t\} \times N$ and $W$ is cylindrical near the boundary, meaning that
\[
    [0,\varepsilon) \times M \sqcup (t-\varepsilon,t] \times N \subset W.
\]
A morphism $(t,W) \to (s,W')$ in $\Bord_d(M,N)$ is an orientation preserving diffeomorphism $\varphi\colon W \cong W'$ that satisfies $i' = \varphi\circ i$ and is cylindrical near the boundary.%
\footnote{
    Here ``cylindrical'' means that near the boundaries $\varphi$ is of the form $\id_{[0,\varepsilon)} \times \varphi_0$ and $\id_{(t-\varepsilon,t]} \times \varphi_t$ for some $\varepsilon >0$.
}
We can topologise this space of morphisms using the Whitney $\mathcal{C}^\infty$-topology, i.e.~the topology of convergence in all derivatives.

If we geometrically realize this topological groupoid we obtain a topological space, whose homotopy type is 
\[
    |\Bord_d(M,N)| \simeq \bigsqcup_{[W\colon M \to N]} B\Diff_\partial(W).
\]
Here the disjoint union runs over diffeomorphism types of bordisms, i.e.~the set of path components of $|\Bord_d(M,N)|$ recovers the set of morphisms in the $1$-category of bordisms from the start of the lecture.
Each component corresponding to the diffeomorphism class of some $W \colon M \to N$ is the classifying space for the topological group $\Diff_\partial(W)$ of diffeomorphisms $\varphi\colon W \cong W$ that restrict to the identity on the boundary.%
\footnote{
    In fact, our definition of $\Bord_d(M,N)$ suggests that we should consider the smaller group $\Diff_{\varepsilon \partial}(W) \subset \Diff_\partial(W)$ where diffeomorphisms are required to be the identity on a specified collar of the boundary.
    It follows from the theorem of Cerf (about the contractibility of the space of collars) that this inclusion is a homotopy equivalence, so we can choose to work with $\Diff_\partial(W)$ instead.
}

As we'll mostly be talking about the case of $d=2$, let us recall some facts about the diffeomorphism groups of surfaces.
The group of path components $\pi_0\Diff_\partial(W)$ is the \hldef{mapping class group} of $W$; its elements are isotopy classes of diffeomorphisms of $W$.
\begin{thm}[Smale, Earle--Eells]
    The diffeomorphism group of the $2$-sphere and torus are equivalent to the following Lie groups
    \[  
        \SO(3) \simeq \Diff(S^2) 
        \qqand
        \mrm{SL}_2(\Zbb) \ltimes (S^1 \times S^1) \simeq \Diff(T^2) \,.
    \]
    For all other connected surfaces $\Sigma$ the identity component of the diffeomorphism group is contractible and thus $\Diff_\partial(\Sigma) \to \pi_0\Diff_\partial(\Sigma)$ is an equivalence.
\end{thm}

The mapping class groups $\Diff_\partial(\Sigma_{g,k})$ are usually infinite, always finitely presented and are well-studied in geometric group theory.

\begin{rem}
    Note that the above theorem means that $\Bord_2$ is \emph{almost} a $(2,1)$-category.
    The only thing that keeps its mapping spaces from being $1$-groupoids are $\Diff(S^2)$ and $\Diff(S^1 \times S^1)$.
    In particular, if we consider variants of $\Bord_2$ where closed manifolds are not allowed, then they'll be $(2,1)$-categories.
\end{rem}

\subsubsection{The bordism category as an \captioninfty-category.}
We now assemble the above into an \category{}.
For a more careful discussion, also in the extended case, see \cite{CS19}.
For a treatment similar to the one given here and a discussion of why we don't have to provide identity morphisms, see \cite{KK24}.

We can now define $\Bord_d$ as the \category{} associated to the following topological category.
Objects are closed oriented $(d-1)$-manifolds, the space of morphisms from $M$ to $N$ is given by $|\Bord_d(M, N)|$, and composition is defined as the realization of the functor
\[
    -\cup_N-\colon \Bord_d(M, N) \times \Bord_d(N, L) \too \Bord_d(M, L)
\]
that glues two bordisms along $N$.
Since we have embedded our bordisms in a cylindrical way we can glue them simply by shifting one of them and then taking a union.
\[
    W \cup_N V = (s,V) \circ (t,W) = (s+t, (V+s) \cup W)
\]
This works well with diffeomorphisms between bordisms and thus extends to a functor of topologically enriched groupoids.
The observant reader might complain that this $\Bord_d$ does not have identity morphisms: there is a canonical isomorphism $(M \times [0,1]) \cup_M W \cong W$, but these morphisms aren't equal.
This can be addressed by modifying our model, but only at the cost of making it more convoluted. 
(Though there are easier solutions if one works with topological spaces of submanifolds, this is for example done in \cite{GMTW}.)
Instead of doing that, we can use that the ``associated category'' functor $\ac\colon \TopCat \to \Cat$ can be uniquely extended to the category of quasi-unital topologically enriched categories.
These are topologically enriched categories $\Ccal$ that are not necessarily strictly unital, but which for each $x \in \Ccal$ there exists a quasi-unit $q_x\in \Ccal(x,x)$ such that the maps
\[
    q_x \circ -\colon  \Ccal(w,x) \too \Ccal(w,x)
    \qqand
    - \circ q_x\colon  \Ccal(x,y) \too \Ccal(x,y)
\]
are homotopic to the identity.
We'll briefly encounter quasi-units again later in \cref{thm:non-unital}.

\subsection{Symmetric monoidal structures}
In order to complete the definition of TFTs valued in a symmetric monoidal \category{} we still have to define what a symmetric monoidal structure on an \category{} is and how to construct such a structure on $\Bord_d$.
We would like to make the following definition:
\begin{defn}
    A symmetric monoidal \category{} is a commutative monoid in $\Cat$
    \[
        \SM \coloneq \CMon(\Cat).
    \]
\end{defn}
But this just leaves us with the question of what a commutative monoid in an \category{} is.

\subsubsection{An example: disjoint union of manifolds.}
Let's start with a simpler example. 
The \category{} $\Mfd_d^\cong = \Bord_d(\emptyset, \emptyset)$ obtained from the topological groupoid of $d$-manifolds and diffeomorphisms between them should be a commutative monoid under the multiplication operation
\[
    - \sqcup -\colon \Mfd_d^\cong \times \Mfd_d^\cong \too \Mfd_d^\cong
\]
that takes the disjoint union of two $d$-manifolds.
However, as for the gluing of bordisms in the last lecture ``taking the disjoint union'' of two $d$ manifolds is not literally associative or commutative.
(Say if they're each in $\Rbb^d \times [0,1]$, then we need to choose some order in which to stack them in the $[0,1]$-direction.)
Of course, if we know that two $d$-manifolds are already \emph{disjoint} submanifolds of $\Rbb^d \times [0,1]$ then we can just take their union, which is trivially associative, unital, and commutative whenever it is defined.
(Since the Boolean algebra of subsets of $\Rbb^d \times [0,1]$ is a commutative monoid under union.)
In other words, we have a zig-zag of topologically enriched groupoids
\[
    \Mfd_d^\cong \times \Mfd_d^\cong \xleftarrow[\quad]{\simeq}
    \{ (W_1,W_2) \;|\; W_1, W_2 \subset \Rbb^d \times [0,1], W_1 \cap W_2 = \emptyset \} \xtoo{\cup}
    \Mfd_d^\cong 
\]
where the leftward map is a Dwyer-Kan equivalence of topological groupoids.
As passing to the \category{} of spaces turns this equivalence into an isomorphism, and so we can still think of it as giving a binary operation on $\Mfd_d^\cong$.
To assert the associativity of this operation, we need to consider disjoint unions of more than two manifolds.
For a finite set $A$, let $A_+ \colon A \sqcup \{\infty\}$ be the finite set obtained by freely adjoining a base point.
Then let
\[
    \Mfd_d^\cong(A_+) \coloneq \{ W \subset \Rbb^d \times [0,1], \alpha\colon W \to A\}
\]
be the topologically enriched groupoid of closed submanifolds $W \subset \Rbb^d$ together with a continuous map $\alpha\colon W \to A$.
(Morphisms are diffeomorphisms that commute with the map to $A$.)
If we set $A = \ul{2} =\{1,2\}$ this recovers the previous definition as giving a map $W \to \ul{2}$ is the same as giving a disjoint decomposition $W = W_1 \sqcup W_2$.
This defines a functor from the category of pointed finite sets to the category of topologically enriched groupoids
\[
    \Mfd_d^\cong(-)\colon \Fin_* \too \TopGpd.
\]
Here to a morphism $\beta\colon A_+ \to B_+$ we assign the functor
\begin{align*}
    \Mfd_d^\cong(A_+) &\too \Mfd_d^\cong(B_+), \\
    (\alpha \colon W \to A) &\longmapsto (\beta \circ \alpha \colon W' \to B)
\end{align*}
where $W' \coloneq (\beta \circ \alpha)^{-1}(B) \subset W$ is the part of $W$ that is not sent to $\infty$ by $\beta \circ \alpha$.
The binary operation can then be recovered as the zig-zag
\[
    \Mfd_d^\cong(\ul{1}_+) \times \Mfd_d^\cong(\ul{1}_+) 
    \xleftarrow[\quad]{(\Mfd_d^\cong(\rho_1), \Mfd_d^\cong(\rho_2))}
    \Mfd_d^\cong(\ul{2}_+)
    \xtoo{\Mfd_d^\cong(\mu)}
    \Mfd_d^\cong(\ul{1}_+)
\]
where $\mu\colon \ul{2}_+ \to \ul{1}_+$ is the map that sends $1, 2 \mapsto 1$.
The maps $\rho_i$ are an instance of a more general map 
\begin{align*}
    \rho_a \colon A_+ &\too \ul{1}_+\\
    x &\mapsto \begin{cases}
        1 & \text{ if }x = a \\
        \infty & \text{ otherwise}
    \end{cases} 
\end{align*}
that we can define for all finite sets $A$ and elements $a \in A$.

\subsubsection{Commutative monoids in an \captioninfty-category.}
Motivated by the above example, let us now define what a commutative monoid in a \category{} $\Ccal$ is.
For simplicity, we will assume that finite products exist in $\Ccal$, though this is not really necessary.
This definition is modelled on Segal's definition of $\Gamma$-spaces in \cite{Segal1974}, see \cref{rem:Segal-Gamma}, for a modern approach see \cite[\S1]{gepner-universality}.

\begin{defn}\label{defn:CMon}
    A \hldef{commutative monoid} in $\Ccal$ is a functor $M\colon \Fin_* \to \Ccal$ such that for every finite set $A$ the map
    \begin{equation}\label{eqn:Segal-condition}
        (\rho_a)_{a \in A}\colon M(A_+) \too \prod_{a \in A} M(\ul{1}_+)
    \end{equation}
    is an equivalence.
\end{defn}

As for the manifold example we can obtain a multiplication on $M(\ul{1})_+$ via the zig-zag
\[
    M(\ul{1}_+) \times  M(\ul{1}_+) \xleftarrow[(\rho_1,\rho_2)]{\simeq}  M(\ul{2}_+) \xrightarrow{\mu} M(\ul{1}_+)
\]
where the left map is an equivalence by the Segal condition.
This multiplication map is unital, associative, and commutative, up to higher coherence, because it comes as part of the functor $M$.

\begin{defn}
    The \category{} of commutative monoids in $\Ccal$ is the full subcategory 
    \[
        \hldef{\CMon(\Ccal)} \subset \Fun(\Fin_*, \Ccal)
    \]
    on those functors that satisfy the Segal condition from \cref{eqn:Segal-condition}.
    In particular, the \category{} of symmetric monoidal \categories{} is the full subcategory
    \[
        \hldef{\SM} \coloneq \CMon(\Cat) \subset \Fun(\Fin_*, \Cat).
    \]
\end{defn}

We can define the symmetric monoidal structure on $\Bord_d$ analogously to how we defined the disjoint union on $\Mfd_d^\cong$.

\begin{rem}\label{rem:Segal-Gamma}
    Commutative monoids in the \category{} of spaces $\An$ are also known as $\Ebb_\infty$-algebras in spaces, or special $\Gamma$-spaces, see \cite{Segal1974,gepner-universality}.
    If we additionally require that the induced multiplication on $\pi_0 M(\ul{1}_+)$ has inverses, then we get very special $\Gamma$ spaces, or equivalently infinite loop spaces.
\end{rem}

\begin{rem}
    The above definition of symmetric monoidal \categories{} is equivalent, via the straightening-unstraightening equivalence, to the one given in \cite{HA}.
    The upside of the definition given here is that it does not mention cocartesian fibrations and instead directly considers functors to $\Cat$.
    The downside is that often the best way to define a functor to $\Cat$ is exactly be writing down a cocartesian fibration and using straightening (thus factoring through Lurie's description), but in fact this is not needed for the examples we consider here.
    In the case of $\Bord_d$ we can directly write down a functor $\Fin_* \to \TopCat \to \Cat$ and trying to write it as a cocartesian fibration would only complicate matters.
\end{rem}

A more conceptual way of arriving at the Segal condition, which will be easier to generalize, is as follows.
For this we consider a generalization of the maps $\rho_a\colon A_+ \to \ul{1}_+$.
This is an instance of the general framework of ``algebraic patterns'' developed by Chu--Haugseng \cite{patterns1}, which we will implicitly be using throughout the lecture series.
The reader interested in homotopy coherent structures is strongly encouraged to also look at \cite[Sections 1-8]{patterns1}.
\begin{defn}
    A map $f\colon A_+ \to B_+$ in $\Fin_*$ is called \hldef{inert} if $f^{-1}(b)$ has exactly one element for all $b \in B$. (Though $f^{-1}(\infty)$ is allowed to have as many elements as it wants.)
    Let $\Fin_*^\xint \subset \Fin_*$ denote the wide subcategory where the morphisms are inert maps.
\end{defn}

Then we have the following equivalent characterization of commutative monoids.
\begin{lem}[{\cite[Lemma 2.9]{patterns1}}]\label{lem:Segal-RKE}
    A functor $M\colon \Fin_* \to \Scal$ satisfied the Segal condition if and only if the restricted functor $M_{|\Fin_*^\xint}$ is right Kan extended from the full subcategory $\{\ul{1}_+\} \subset \Fin_*^\xint$.
\end{lem}

\subsection{Higher categorical targets}\label{subsec:higher-targets}

\subsubsection{Chain complexes and cohomological field theories.}
Let $\Ch_\Qbb$ denote the derived category of $\Qbb$, i.e.~the \category{} obtained by inverting quasi-isomorphisms in the $1$-category of chain complexes over $\Qbb$.
Then a $\Ch_\Qbb$-valued TFT is a symmetric monoidal functor
\[
    \Zcal \colon \Bord_2 \too \Ch_\Qbb.
\]
We will see later that $\Zcal(S^1)$ is in particular an $\EtwoSO$-algebra.
This functor in particular induces maps
\[
    \Zcal_{g,k}\colon B\Diff_\partial(\Sigma_{g,k}) \subset \Map_{\Bord_2}(\emptyset, \sqcup_k S^1)
    \too \Map_{\Ch_\Qbb}(\Qbb, \Zcal(S^1)^{\otimes k})
\]
The mapping spaces $\Map_{\Ch_\Qbb}(\Qbb, C)$ in $\Ch_\Qbb$ are Eilenberg-Mac Lane spaces on (the connective part of) the homology of $C$, so we can interpret these maps as cohomology classes:
\[
    \Zcal_{g,k} \in H^*(B\Diff_\partial(\Sigma_{g,k}); \Zcal(S^1)^{\otimes k}).
\]
This is closely related to the notion of \emph{cohomological field theory} introduced by Witten \cite{Witten-cohFT}.
In the mathematical formulation by Kontsevich--Manin \cite[Definition 2.2]{KontsevichManin1994} such a cohomological field theory roughly consists of the cohomology ring of a variety $A=H^*(V;\Qbb)$ and maps
\[
    I_{g,n}\colon A^{\otimes n} \too H^*(\Mbar_{g,n};\Qbb)
\]
for all $g,n$ with $n+2g\ge 3$, subject to various conditions.
The key difference here is that the Deligne--Mumford compactification $\Mbar_{g,n}$ appears, rather than the moduli space $\Mcal_{g,n} \simeq B\Diff(\Sigma_{g,n})$.
This difference can be made up for by establishing a pushout square of modular \operads{}
\[\begin{tikzcd}
    B\SO(2) \ar[r] \ar[d] \ar[dr, phantom, very near end, "\ulcorner"] & \Mcal \ar[d] \\
    * \ar[r] & \Mbar .
\end{tikzcd}\]
A version of this for (a specific model of) \properads{} was established by Deshmukh \cite{Des22}, and this ought to be equivalent to a similar pushout of modular \operads{} in the framework outlined in these notes.
Given such a pushout we can apply the symmetric monoidal envelope $\Env(-)$ discussed in \cref{thm:envelope} to obtain a pushout square of symmetric monoidal \categories{}
\[\begin{tikzcd}
    \Free(B\SO(2)) \ar[r] \ar[d] \ar[dr, phantom, very near end, "\ulcorner"] & \Env(\Mcal) = \Bord_2 \ar[d] \\
    \Free(*) \ar[r] & \Env(\Mbar) .
\end{tikzcd}\]
Extra care has to be taken to ensure the stability condition $n+2g\ge 3$.
Here $\Env(\Mbar)$ is a version of the bordism category $\Bord_2$, but where the mapping spaces are built out of the moduli spaces $\Mbar_{g,n}$ instead of $\Mcal_{g,n}$.
Then, a cohomological field theory would correspond to a symmetric monoidal functor
\[
    \Zcal\colon \Env(\Mbar) \too \Ch_\Qbb
\]
and the above pushout square would allow us to describe such functors are $\Ch_\Qbb$-valued TFTs together with a trivialization of the $S^1$-action on $\Zcal(S^1)$.

\subsubsection{Invertible TFTs and the GMTW-Theorem.}
A TFT is called \hldef{invertible} if $\Zcal(M) \in \Ccal$ is invertible under the tensor product and $\Zcal$ sends all morphisms to equivalences, see e.g.~\cite{SchommerPries2018}.
Invertible TFTs play an important role in theoretical physics, as they can both be used to describe low-energy states of gapped systems in condensed matter physics, and they appear as the anomaly of quantum field theories, see e.g.~the introduction of \cite{FreedHopkins}.

A TFT $\Zcal$ is invertible if and only if it lands in the Picard groupoid of $\Ccal$, which is defined as the full subgroupoid $\mrm{Pic}(\Ccal) \subset \Ccal^\simeq$ of the core on the $\otimes$-invertible objects.
Thus, to study invertible TFTs it suffices to study functors from $\Bord_d$ to Picard groupoids.
There is an adjunction
\[
    |-|^{\rm gp} \colon \SM \adj \Omega^\infty\text{-}\mrm{Spaces} 
\]
where the right adjoint identifies the category of infinite loop spaces with the full subcategory of $\SM$ on the Picard groupoids.
Thus, to study $d$-dimensional TFTs we only need to know the infinite loop space $|\Bord_d|$. (This is automatically group-like, so we don't need to apply the group completion $(-)^{\rm gp}$.)
This infinite loop space was described by Galatius--Madsen--Tillmann--Weiss \cite{GMTW} using a parametrized Pontryagin--Thom construction as
\[
    |\Bord_d| \simeq \Omega^{\infty-1} \mrm{MTSO}(d).
\]
Thus, invertible $d$-dimensional TFTs can be described in terms of maps out of the connective spectrum $\tau_{\ge0}\Sigma\mrm{MTSO}(d)$.
A more sophisticated version of this approach, allowing for tangential structures, reflection positivity, and fully extended TFTs, is used to great effect by Freed--Hopkins \cite{FreedHopkins} to make computations in condensed matter physics.

\subsubsection{Bordism categories.}
Another source of $2$-dimensional TFTs is by \hldef{dimensional reduction} from higher dimensional TFTs.
Suppose we have a $3$-dimensional TFT $\Zcal$ valued in $\Ccal$, then we obtain a $2$d TFT as the composite
\[
    \Bord_2 \xtoo{-\times S^1} \Bord_3 \xtoo{\Zcal} \Ccal
\]
which is also sometimes called the $S^1$-compactification of $\Zcal$.
This can be a physically meaningful operation, but it can also be a useful tool for studying the TFT $\Zcal$. (This strategy is for example used extensively in \cite{SchommerPries2018}.)
Note that this reduction gives us more than just a $2$d TFT:
since we can act on $S^1$ by rotations, we in fact get a $2$d TFT with $\SO(2)$-action.
More generally, this construction thus defines for every closed $n$-manifold $M$ a functor
\[
    \Fun^\otimes(\Bord_d, \Ccal) \too \Map(B\Diff(M), \Fun^\otimes(\Bord_{d-n}, \Ccal)).
\]

Motivated by this, we can ask if there are any other ways of doing a dimensional reduction, which do not come from crossing with a manifold.
\begin{question}
    Are there any symmetric monodial functors
    \[
        \Bord_{d-n} \too \Bord_d
    \]
    that are not of the form $-\times M$ for some closed $n-d$?
\end{question}
Presumably such functors exist, but the author has no clue how one would go about constructing them.
One category level up, one can also ask what the automorphisms of the functor $-\times M$ are.
Certainly $\Diff(M)$ acts on the functor, but there might be more automorphisms that don't come from diffeomorphisms.
We can also ask this for $n=0$, in which case the question concerns the automorphisms of $\Bord_d$.
This is an important group as it acts on the space of all $d$-dimensional TFTs, but as far as the author is aware, very little is known about it.
As a consequence of the main theorem of these lectures one can determine this automorphism group in dimension $2$:
\begin{cor}
    The identity and orientation reversal automorphism of $\Bord_2$ induce an equivalence
    \[  
        \Ctwo \xtoo{\simeq} \Aut_{\SM}(\Bord_2).
    \]
\end{cor}

\begin{exc}
    Consider the $1$-category $h(\Bord_2)$ where morphisms are diffeomorphism classes of bordisms.
    Show that the orientation reversal automorphism of $\Bord_2$ becomes isomorphic to the identity as a functor $h(\Bord_2) \to h(\Bord_2)$.
    Use the classification from \cref{thm:folk-2D-TQFT} (suitably adapted to general $1$-categorical targets) to show that $\Aut(h(\Bord_2)) \simeq *$.
\end{exc}

\section{Modular \captioninfty-operads}

The goal of this lecture is to introduce modular operads, which encode the gluing operations of connected surfaces.
The idea of modular operads was first conceived by Getzler--Kapranov \cite{GetzlerKapranov1998} who introduced what we would call ``modular operads enriched in chain complexes''%
\footnote{
    Also note that their definition comes with a genus-grading built in and that they require a ``stability condition'', which for example ensures that no disks appear in their modular operad.
    We'll encounter genus gradings later, but we'll not be imposing a stability condition.
}
in order to study the cohomology of the moduli spaces of curves $\Mcal_{g,k}$ and its compactification $\overline{\Mcal}_{g,k}$.
In particular, for them as for us the key example of a modular operad comes from surfaces.
(Note that we will not discuss enrichment here, but the theory presented easily generalizes to the enriched world thanks to \cite{patterns3}.)
The notion of modular $\infty$-operad, i.e.~modular operads enriched in ``spaces'', was first introduced by Hackney--Robertson--Yau \cite{HRY-graphical-ModOp} and the approach we present here is equivalent to theirs.
The only difference is that we give an alternative, but equivalent, definition of the graph category $\Gr$ (denoted $\mbf{U}^\op$ in their work).
Moreover, we in \cref{subsec:mfd-modular-operad} sketch a construction of the manifold modular operad $\Bcal_d$, which seems to be new in this form.%
\footnote{
    Though it of course is modelled on Getzler--Kapranov's modular operad.
    Also, Basualdo-Bonatto--Robertson already have a construction of a modular $\infty$-operad analogous to $\Bcal_2$.
}

The key idea behind modular operads is that they should model the algebraic structure that the moduli spaces of connected surfaces $B\Diff_\partial(\Sigma_{g,k})$ have under gluing boundary circles.
The elementary operations are the gluing maps
\[
    B\Diff_\partial(\Sigma_{g_1,k_1}) \times  B\Diff_\partial(\Sigma_{g_2,k_2}) \to
    B\Diff_\partial(\Sigma_{g_1+g_2,k_1+k_2-2}) 
    \quad\text{and}\quad
    B\Diff_\partial(\Sigma_{g,k}) \to B\Diff_\partial(\Sigma_{g+1,k-2}),
\]
but to keep track of compatibilities between these operations we will in fact consider more elaborate gluing maps where the combinatorics of the gluing are prescribed by a graph.
\begin{defn}\label{defn:cut-system}
    Let $M$ be a connected compact oriented $d$-manifold with boundary.
    A \hldef{cut-system} in $M$ is a codimension $1$ submanifold $S \subset M$ that is two-sided/coorientable and contains the boundary of $M$.
    To such cut-system we can assign a \hldef{dual graph $\Delta(M,S)$} whose vertices are the path components of $M \setminus S$ and whose edges are the components of $S$.
\end{defn}

\begin{figure}[ht]
    \centering
    \includegraphics[width = \linewidth]{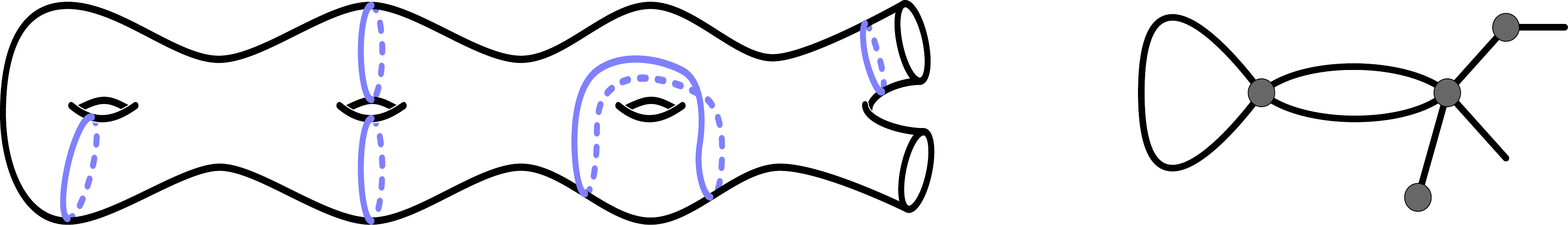}
    \caption{A connected surface $\Sigma$ with a cut system $S \subset \Sigma$ and its dual graph $\Delta(\Sigma,S)$.
    Note that dual graphs are connected, but can have loops, double-edges and univalent vertices.}
    \label{fig:dual-graph}
\end{figure}

\subsection{Graphs}

We now set up the precise combinatorial framework that we will use for describing the gluing of surfaces.
\begin{defn}
    A \hldef{graph} is a $4$-tuple $(V,A,\dagger,r)$ where $V$ and $A$ are finite sets, $\dagger \colon A \to A$ is a fixed-point free involution, and $t\colon A \to V_+$.
    We refer to $V$ as the set of vertices and $A$ as the set of arcs or half-edges.
    We will require all graphs to be \hldef{connected} in the sense that they are non-empty and cannot be written as a disjoint union.
\end{defn}

Equivalently, a graph is connected if the equivalence relation on $A \sqcup V$ generated by
$a \sim a^\dagger$ for $a \in A$ and $a \sim t(a)$ for $a \in t^{-1}(V)$
has only one equivalence class.

\begin{defn}
    We say that a graph is \hldef{elementary} if it is isomorphic to one of the following.
    \begin{itemize}
        \item The \hldef{edge} $\hldef{\fre} \coloneq (\emptyset, \{a, a^\dagger\})$, or
        \item The \hldef{$k$-corolla} $\hldef{\frc_k} \colon (\{v\}, \{a_1,a_1^\dagger, \dots, a_k, a_k^\dagger\})$ for some $k \in \Nbb$ where $r(a_i) = v$ and $t(a_i^\dagger) = \infty$, 
        including the \hldef{isolated vertex} $\frc_0 = (\{v\}, \emptyset)$.
    \end{itemize}
    Equivalently, a graph $(V,A)$ is elementary if and only if $V$ has at most one element and for every $a \in A$ we have $\infty \in \{t(a),t(a^\dagger)\}$.
\end{defn}

\begin{example}\label{ex:linear-graph}
    For all $n\ge 0$ the \hldef{linear graph} $\mfr{l}_n$ is defined as
    \[
        \mrm{l}_n \coloneq (\{v_1,\dots,v_n\}, \{a_0,a_0^\dagger, \dots, a_n,a_n^\dagger\}))
        \text{ where }
        t(a_i) = \begin{cases}
            v_i & \text{ if } i>0, \\
            \infty & \text{ if } i=0,
        \end{cases}
        \text{ and }
        t(a_i^\dagger) = \begin{cases}
            v_{i+1} & \text{ if } i<n, \\
            \infty & \text{ if } i=n.
        \end{cases}
    \]
\end{example}

\begin{defn}
    To each graph $\Gamma = (V,A)$ we assign a topological space, called its \hldef{compactification} defined by
    \[
        \Gamma^+ \coloneq \left( V_+ \sqcup A \times [0,1] \right) / ( (a,0) \sim t(a) \text{ and } (a,s) \sim (a^\dagger,1-s))
    \]
    This is a finite $1$-dimensional pointed CW complex with $0$ skeleton $V_+$ and base point $\infty$.
\end{defn}

\begin{figure}[ht]
    \centering
    \includegraphics[width = .6\linewidth]{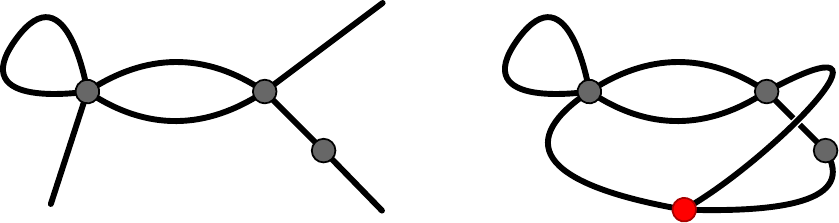}
    \caption{A graph $(X,V)$ and its one-point compactification $(X^+,V^+)$ with $\infty$ labelled in red.}
    \label{fig:compactification}
\end{figure}

\begin{rem}
    In fact, there's a bijection between isomorphism classes of graphs and isomorphism classes of finite $1$-dimensional pointed CW complexes that have the property that they are connected after removing the base point.
\end{rem}

\begin{war}\label{war:nodeless-loop}
    The description of graphs through their compactification suggests that we might want to add another graph $\frs$ to our category $\Gr$ such that $\frs^+ = S^1 \sqcup\{\infty\}$.
    This would be a graph with no vertices and a single edge that loops back onto itself; called the \hldef{nodeless loop}.
    Such a graph cannot be expressed using the combinatorial framework we have chosen, but it makes sense topologically, if we define a graph to be anything that ``locally looks like $\fre$ or $\frc_k$''.
    (A concrete definition would characterize $\Gamma^+$ as those pairs $(X,V^+,\infty)$ of pointed compact Hausdorff spaces where $V$ is finite and $X \setminus V^+$ is a $1$-manifold.)
    Adding this nodeless loop will yield a new category $\Gr^{\frs}$ that contains $\Gr$ (\cref{defn:graph-map}) as a full subcategory.
    This new category has better categorical properties, in particular it is an \emph{extendable algebraic pattern} in the sense of \cite[Definition 8.5]{patterns1}, which is useful for computing free modular operads.
    The downside is that $\Gr^{\frs}$ is no longer a $1$-category, but rather an \category{}, as for example $\Aut_{\Gr^\frs}(\frs) = \mrm{O}(2)$.
    Hackney--Robertson--Yau introduce such a variant of their graph category $\widetilde{U}$ \cite[Section 4.1]{HRY-graphical-ModOp}, which is a $1$-category and thus will be equivalent to $h(\Gr^{\frs})^\op$.
    However, this does not make $\widetilde{U}^\op$ into an extendable pattern.
    For the purpose of these lectures we'll simply ignore the nodeless loop and work with the smaller category $\Gr$, as it turns out that $\Gr$ and $\Gr^\frs$ result in equivalent notions of modular \operads.
\end{war}

\subsubsection{Graph maps -- topological.}

\begin{defn}\label{defn:graph-map}
    A \hldef{graph map} $f \colon \Gamma \to \Lambda$ is a continuous map $f\colon \Gamma^+ \to \Lambda^+$ between one-point compactifications such that
    \begin{enumerate}[(1)]
        \item $f(V_\Gamma^+) \subset V_\Lambda^+$ and $f(\infty) = \infty$, 
        \item $f^{-1}(\Lambda^+ \setminus V_\Lambda^+) \to \Lambda^+ \setminus V_\Lambda^+$ is a homeomorphism, and 
        \item $f^{-1}(v)$ is connected for all $v \in V_\Lambda$.
    \end{enumerate}
    We let $\hldef{\Gr}$ denote the category whose objects are connected graphs and morphisms are isotopy classes of graph maps.
    (An isotopy of graph maps is a continuous $[0,1]$-family of graph maps.)
\end{defn}

\begin{rem}
    If we drop the third condition we obtain a more general category $\GR$.
    This $\GR$ seems to be the appropriate generality in which to develop the general theory discussed here, but its corresponding algebraic structure are ``circuit algebras'' / ``non-connected modular operads'', which will not turn up in this lecture series.
    So we will restrict to $\Gr$ throughout.
\end{rem}

Note that a graph map $f\colon \Gamma \to \Lambda$ induces a map of vertices $f_{\rm vert} \colon V_\Gamma^+ \to V_\Lambda^+$ and a map of arcs (by taking preimages) $A_\Gamma \leftarrow A_\Lambda\cocolon f_{\rm arc}$.
These induce functors
\[
    V^+\colon \Gr \too \Fin_* 
    \qqand
    A\colon \Gr^\op \too \Fin^{\BCtwo}.
\]

The idea that graph maps are isotopy classes of continuous maps might be a little unsettling at first, but it turns out that (at least for connected graphs) the data of a graph map is determined entirely combinatorially.
\begin{lem}\label{lem:isotopy-class}
    The combined functor
    \[
        (V^+, A)\colon \Gr \too \Fin_* \times (\Fin^{\BCtwo})^\op
    \]
    is faithful.
    In other words, two graph maps are isotopic if and only if they induce the same map on vertices and on arcs.
\end{lem}

\begin{figure}[ht]
    \centering
    \includegraphics[width = .8\linewidth]{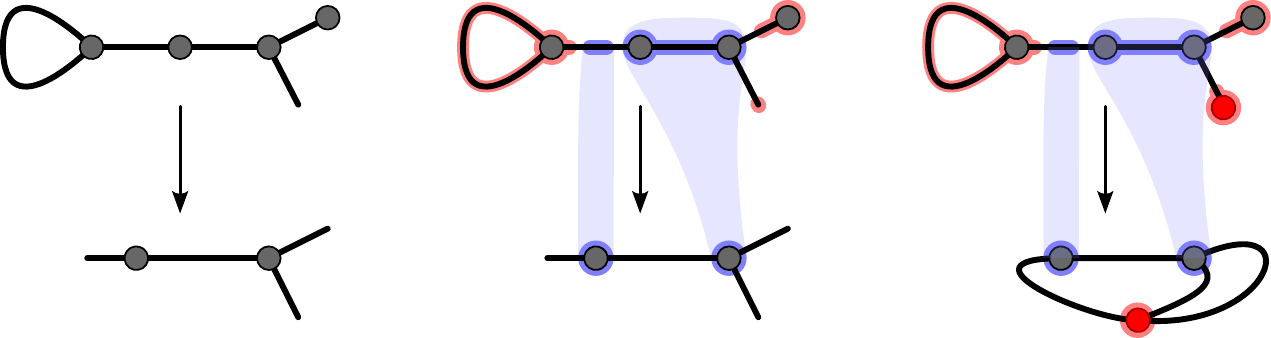}
    \caption{
    Three pictures of the same graph map $(X,V) \to (Y,W)$.
    First unlabelled, second with the preimages of vertices in blue and the preimages of infinity in red, and third the continuous map $X^+ \to Y^+$ with the preimages of infinity in red.}
    \label{fig:three-maps}
\end{figure}

\subsubsection{Graph maps -- combinatorial.}
By \cref{lem:isotopy-class} graph maps are entirely determined by combinatorial data. 
Consequently, it must be possible to define them without ever making reference to the topological space $\Gamma^+$.
We'll discuss this combinatorial definition here and leave it up to the reader to chooses which of the equivalent definitions they work with.

In order to define graph maps we will need the notion of a path in a graph.
Fix a graph $\Gamma = (V,A)$ and two vertices $v_a, v_b \in \Gamma$.
A \hldef{path} in $\Gamma$ from $v_{\rm start}$ to $v_{\rm end}$ is a finite sequence of arcs $(a_1,\dots, a_n)$ such that
    \[
        v_{\rm start} = t(a_1^\dagger), \,
        t(a_1) = t(a_2^\dagger), \,
        \dots \,
        t(a_{n-1}) = t(a_n^\dagger), \,
        t(a_n) = v_{\rm end}.
    \]
If $v_{\rm start} = v_{\rm end}$ we allow the empty path.
We say that a path is \hldef{bivalent} if each of the vertices $t(a_i)$ is bivalent or $\infty$ for all $1\le i <n$.

\begin{defn}
    A \hldef{graph map} $f\colon (V,A) \to (W,B)$ consists of two maps
    \[
        f_{\rm v}\colon V_+ \to W_+
        \qqand 
        A \longleftarrow B \cocolon f_{\rm a}
    \]
    subject to the following conditions
    \begin{enumerate}
        \item $f_{\rm v}(\infty) = \infty$,
        \item $f_{\rm a}(b^\dagger) = f_{\rm a}(b)^\dagger$ for all $b \in B$,
        \item $f_{\rm a}^{-1}(a)$ admits a (unique) ordering making it a bivalent path from $f_{\rm v}(t(a))$ to $f_{\rm v}(t(a^\dagger))$ in $B$,
        \item\label{it:bivpath} if $f_{\rm v}(v) = f_{\rm v}(v') \neq \infty$ for some $v,v' \in V$, then there is a path from $v$ to $v'$ in $(V,A)$ that uses only half-edges in $A\setminus f_{\rm h}(B)$.
        \item each $w \in W_+ \setminus f_{\rm v}(V^+)$ is bivalent and appears in exactly one path as described in \ref{it:bivpath}.
    \end{enumerate}
\end{defn}

\begin{rem}
    The combinatorial definition here relies (just like \cref{lem:isotopy-class}) on the assumption that the graphs are connected, but this can be generalized by instead encoding $f_{\rm a}$ as a map in $\Assoc$, i.e.~by remembering total orders on its preimages.
\end{rem}

\begin{lem}
    Both the topological and the combinatorial definition yield well-defined categories, and these categories are equivalent to each other and to the category $\mathbf{U}^\op$ of \cite{HRY-graphical-ModOp}.
\end{lem}

\begin{figure}[ht]
    \centering
    \includegraphics[width = \linewidth]{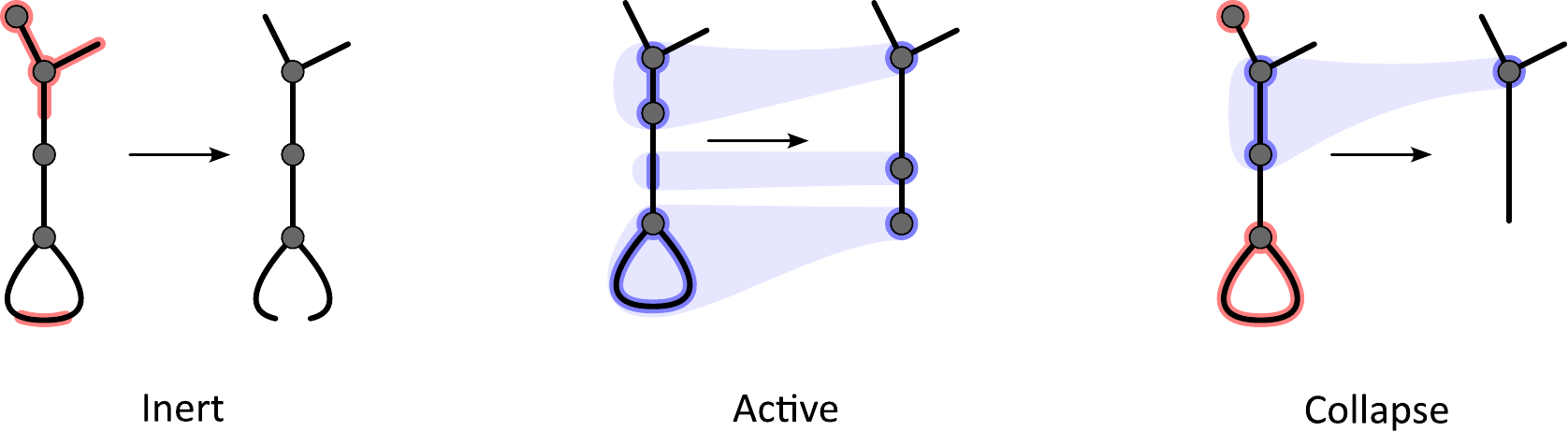}
    \caption{
        Three graph maps. The first two are inert and active, respectively.
        The third map is neither inert nor active, but it is a (quasi-)collapse in the sense of \cref{defn:quasi-collapse}
    }
    \label{fig:three-connected-maps}
\end{figure}

\subsection{Modular operads}
\subsubsection{The Segal condition -- one-coloured.}
We would like to define modular \operads{} as functors $\Gr \to \An$ satisfying a condition analogous to \cref{eqn:Segal-condition} in the definition of commutative monoids.
For this we need analogues of the maps $\rho_a\colon A_+ \to \ul{1}_+$ that project to individual elements.

For every vertex $v \in \Gamma$ let $A_v = r^{-1}(v) \subset A(\Gamma)$ be the set of half-edges incident at $v$.
Let $\frc_v$ be the corolla that has $A_v \times \{+,-\}$ as its set of arcs with $(a^+)^\dagger = a^-$, $r(a^+) = v$ and $r(a^-) = \infty$.
Then there is a canonical graph map
\[
    \rho_v\colon \Gamma \too \frc_v
\]
that sends $v$ to the unique vertex of $\frc_v$ and every other vertex to $\infty$, and on half-edges is the map $A_v \times \{+,-\} \to A$ that restricts to the inclusion of $A_v$ on $A_v \times \{+\}$.

\begin{defn}
    A functor $\Ocal\colon \Gr \too \An$ is a \hldef{one-coloured modular \operad} if for every graph $\Gamma$ the map
    \[
        (\rho_v)_{v \in V(\Gamma)} \colon \Ocal(\Gamma) \too \prod_{v \in V(\Gamma)} \Ocal(\frc_v)
    \]
    is an equivalence.
\end{defn}

\begin{example}
    Given a commutative monoid $M\colon \Fin_* \too \An$ as in \cref{defn:CMon} we can define a modular operad as the composite
    \[
        V_+^*M\colon \Gr \xtoo{V_+} \Fin_* \xtoo{M} \An
    \]
    where the first functor sends a graph to its set of vertices plus an additional base point.
\end{example}

\begin{example}\label{ex:grading-modular-operad}
    The \hldef{grading modular operad} $\gr\colon \Gr \to \Sets \subset \An$ is defined by assigning to each graph $\Gamma$ its set of possible labellings of vertices by natural numbers
    \[
        \gr(\Gamma) \coloneq \Map(V_\Gamma, \Nbb)
    \]
    and to a morphism $f\colon\Gamma \to \Lambda$ it assigns the map
    \begin{align*}
        \gr(f) \colon \Map(V_\Gamma, \Nbb) & \too \Map(V_\Lambda, \Nbb) \\
        \alpha & \longmapsto \left( v \in V_\Lambda \mapsto b_1(f^{-1}(v)) + \sum_{w \in W \cap f^{-1}(v)} \alpha(w) \right)
    \end{align*}
    that sums the weights over the fiber and adds the Betti number of the fiber.
    See \cref{fig:graded-graph-map} on page \pageref{fig:graded-graph-map} that describes the unstraightening of this functor.
\end{example}

\begin{exc}
    Show that $\gr$ is a well-defined functor and satisfies the Segal condition.
    Show that it is not isomorphic to $V^*M$ for any commutative monoid $M$.
    Consider modular suboperads of $\gr$, i.e.~subfunctors $F \subseteq \gr$ that also satisfy the Segal condition.
    Show there are infinitely many.
    Show that there are exactly two modular suboperads that contain the labelling $0 \in \gr(\frc_3)$
\end{exc}

\begin{example}\label{ex:endomorphism}
    For $\Ccal$ a symmetric monoidal $1$-category, $x \in \Ccal$ a dualisable object, and $e\colon x \otimes x \to \unit$ a non-degenerate symmetric pairing, we can define a one-coloured modular operad $\Ecal$, which generalizes the endomorphism operad of $x$.
    Its value on $\frc_k$ is $\Ecal(\frc_k) = \Map_\Ccal(1, x^{\otimes k})$.
    The functoriality with respect to graph morphisms is defined by using the pairing to cancel pairs of half-edges that are collapsed and by using the dual co-paring to label newly introduced bivalent vertices.
    We will discuss variants and generalizations of this in \cref{ex:underlying}, \cref{ex:underlying-sketch}, and \cref{ex:Udual}.
\end{example}

\subsubsection{The Segal condition -- multi-coloured.}
In a one-coloured modular operad $\Ocal$ the value on the edge $\Ocal(\fre)$ is always trivial.
In general, we need more flexibility than this: for example, when considering the modular operad $\Bcal_d$ that corresponds to $\Bord_d$, the space of colours $\Bcal_d(\fre)$ should be the space of connected closed $(d-1)$-manifolds.
Instead, there should be a space of colours and the product in the Segal condition should be replaced by a fiber product over this space.
In addition to the maps $\rho_v\colon \Gamma \to \frc_v$ for each vertex, we will also need maps
\[
    \rho_h \colon \Gamma \to \fre
\]
for every half-edge $h \in A(\Gamma)$.
This map (necessarily) sends all vertices to $\infty$ and on half-edges it is given by the inclusion $\{h,h^\dagger\} \subset A(\Gamma)$.

\begin{defn}
    A map of graphs $f\colon \Gamma^+ \to \Lambda^+$ is called \hldef{inert} if for all $w \in V_\Lambda$ the preimage $f^{-1}(v)\subset \Gamma^+$ is a single point and this point is a vertex.
    We let $\hldef{\Gr^\xint} \subset \Gr$ denote the wide subcategory where morphisms are inert maps.
    We let $\hldef{\Gr^\el} \subset \Gr^\xint$ denote the full subcategory on the elementary graphs $\fre$ and $\{\frc_k\}_{k\ge 0}$.
\end{defn}

Generalizing from the characterization of the Segal for commutative monoids in terms of inert maps in \cref{lem:Segal-RKE} we make the following definition.

\begin{defn}\label{defn:modular-operad}
    A \hldef{modular operad} is a functor $\Ocal\colon \Gr \to \An$ such that the restriction
    \[
        \Ocal_{|\Gr^\xint} \colon \Gr^\xint \too \An
    \]
    is right Kan extended from the full subcategory $\Gr^\el \subset \Gr^\xint$.
\end{defn}

This definition an instance of the concept of Segal spaces over algebraic patterns.
The theory of algebraic patterns, developed by Chu and Haugseng \cite{patterns1}, is extremely useful, and you should check it out!

We can also spell out the condition more concretely.
\begin{lem}\label{lem:Segal-condition-pullback}
    A functor $\Ocal \colon \Gr \to \An$ is a modular operad if for every graph $\Gamma$ the square
        \[\begin{tikzcd}
        	{X(\Gamma)} & {\displaystyle\prod_{v \in V(\Gamma)} X(\frc_v)} \\
        	{\displaystyle \prod_{e \in E(\Gamma)} X(\fre)} 
        	& {\displaystyle \prod_{a \in A(\Gamma)} X(\fre_a)} 
        	\arrow[from=1-1, to=2-1, "{(\rho_{e^+})_{e \in E(\Gamma)}}"']
        	\arrow[from=2-1, to=2-2, "{(+,-)}"]
        	\arrow[from=1-2, to=2-2]
        	\arrow[from=1-1, to=1-2, "{(\rho_v)_{v \in V(\Gamma)}}"]
        	\arrow["\lrcorner"{anchor=center, pos=0.125}, draw=none, from=1-1, to=2-2]
        \end{tikzcd}\]
    is cartesian, i.e.~it is a pullback square.
    To construct this square, we choose for each edge $e \in E(\Gamma)$ an ordering of the arcs as $e = \{e^+, e^-\}$; whether the square is cartesian will not depend on this choice.
    Then $\rho_{e^+}\colon \Gamma \intto \fre$ is the inert map we constructed before, 
    $(+,-)$ is the product of maps $X(\fre_e) \to X(\fre_{e^+}) \times X(\fre_{e^-})$ that are the identity on the first factor and apply the automorphism of $\fre$ on the second factor,
    and the right maps is the product of maps $X(\frc_v) \to \prod_{t(a) = v} X(\fre_a)$ coming from the inert maps $\frc_v \intto \fre_a$ that restrict to the arcs incident at the vertex.
\end{lem}

\begin{cor}
    A functor $\Ocal \colon \Gr \to \An$ is a one-coloured modular operad if and only if it is a modular operad and $\Ocal(\fre)$ is contractible.
\end{cor}

\begin{rem}
    If we think about graphs topologically, then we can describe the Segal condition as a sheaf condition with respect to covers by open subgraphs.
    Suppose for simplicity that we allow disconnected graphs in $\Gr$.
    Fix a graph $\Gamma$ and a decomposition $(\Gamma^+)\setminus \infty = U_1 \cup U_2$ such that each $U_i$ is open and has finitely many components.
    Then there are graphs $\Lambda_i$ with $\Lambda_i^+\cong U_i^+$ and $\Lambda_{01}^+ \cong (U_0\cap U_1)^+$ and the collapse maps $\Gamma^+ \to U_i^+$ induce inert maps $\Gamma \intto \Lambda_i$ as well as $\Lambda_i \intto \Lambda_{01}$.
    Then the Segal condition can be expressed as saying
    \[\Ocal(\Gamma) = \Ocal(\Lambda_0) \times_{\Ocal(\Lambda_{01})} \Ocal(\Lambda_1).\]
    The pullback square in \cref{lem:Segal-condition-pullback} is in fact an instance of this for the canonical open cover where $U_0 = \Gamma^+ \setminus V^+$ and $U_1$ is a disjoint union of small open neighbourhoods of the vertices.
    (Here we also need to use that if $\Ocal$ is defined on disconnected graphs, then the Segal condition in particular tells us that $\Ocal(\Gamma \sqcup \Gamma') = \Ocal(\Gamma) \times \Ocal(\Gamma')$.)
    
    Said yet another way, we can interpret $(\Gr^{\xint})^\op$ as a category of open topological graphs and open embeddings, with the comparison to the topological definition of $\Gr$ given by taking each open embedding $i\colon X \hookrightarrow Y$ to its collapse map $i^!\colon Y^+ \too X^+$.
    Then the Segal condition says exactly that $\Ocal_{|\Gr^\xint}$ is a sheaf with respect to such open embeddings.
\end{rem}

\begin{example}\label{ex:orientation}
    The orientation modular operad is the functor
    \[
        \mrm{or}\colon \Gr \xtoo{A} (\Fin^{\BCtwo})^\op \xtoo{\Map^{\Ctwo}(-,\Ctwo)} \Sets \subset \An,
    \]
    which sends a graph $\Gamma$ to the set $\Map^{\Ctwo}(A_\Gamma, \Ctwo)$ of orientations of the arcs of $\Gamma$.
    One checks that this indeed satisfies the condition of \cref{defn:modular-operad} as it is in fact right Kan extended from the full subcategory $\BCtwo=\{\fre\} \subset \Gr$.
    This modular operad has two colours $\mrm{or}(\fre) = \{\to,\leftarrow\}$ corresponding to the two orientations of the edge $\fre$.
    It has exactly one $k$-ary operation for choice of input colours, so $\Ocal(\frc_k)$ has $2^k$ elements.
\end{example}

\begin{example}\label{ex:underlying}
    We can extend the example from \cref{ex:endomorphism} to a multicoloured modular operad $\Ucal_{\rm vect}$ as follows.
    (See also \cite[(2.25)]{GetzlerKapranov1998}.)
    The colours of $\Ucal_{\rm vect}$ will be pairs of a vector space $V$ and a non-degenerate symmetric bilinear form on $V$.
    For a graph $\Gamma$, let $\Ucal_{\rm vect}(\Gamma)$ be the set of labellings $(V, \alpha)$ where for each edge $e \in E_\Gamma$, $V_e$ is a vector space with a non-degenerate symmetric bilinear form, and for each vertex $v \in V_\Gamma$, $\alpha_v \in \bigotimes_{a \in t^{-1}(v)} V_a$ is a vector in the tensor product of the vector spaces assigned to the adjacent edges.
    See \cref{fig:graph-labelled-by-vectors}.

    This can be made into a functor $\Gr \to \Sets$, where inert maps forget data, edge contractions are implemented by using the inner product to contract $\alpha_v \otimes \alpha_w$, and subdivisions are implemented using the copairing $\sum_{i} b_i \otimes b_i^\# \in V \otimes V$ induced by the inner product.
\end{example}

\begin{figure}[ht]
    \centering
    \def\svgwidth{.6\linewidth}
    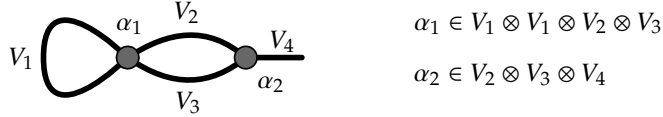
    \caption{A graph $\Gamma$ and a labelling by vector spaces that defines an element in $\Ucal_{\rm vect}(\Gamma)$.}
    \label{fig:graph-labelled-by-vectors}
\end{figure}

\subsubsection{A factorization system.}
To explain why our definition of modular \operads{} in \cref{defn:modular-operad} is indeed an instance of Segal spaces over an algebraic pattern \cite{patterns1}, we briefly describe a factorization system on $\Gr$.
In doing so, we'll give another description of inert maps, this time in terms of the combinatorial definition of graph maps.

\begin{defn}
    A graph map $f\colon \Gamma \to \Lambda$ is called
    \begin{enumerate}[(1)]
        \item \hldef{inert} if $f_{\rm v}^{-1}(v)$ has exactly one element for all $v \in V(\Lambda)$ and $f_{\rm a}^{-1}(a)$ is only empty when $f_{\rm v}(r(a)) = \infty$.
        (This means that the only edges collapsed are those sent to $\infty$.)
        \item \hldef{active} if $f_{\rm v}^{-1}(\infty) = \{\infty\}$ and the path $f_{\rm a}^{-1}(a)$ never passes through $\infty$
    \end{enumerate}
    We denote inert morphism by $\intto$ and active morphisms by $\actto$.
\end{defn}

\begin{lem}
    Every morphism $f\colon \Gamma \to \Lambda$ admits a factorization
    \[
        \Gamma \intto \Sigma \actto \Lambda
    \]
    and the category of such factorizations of $f$ is a contractible groupoid.
\end{lem}

\subsection{The manifold modular operad}\label{subsec:mfd-modular-operad}

We now define a modular \operad{} $\Bcal_d$ that is analogous to the bordism category $\Bord_d$.
This analogy will be made precise in \cref{thm:envelope} where we will see that one can obtain $\Bord_d$ from $\Bcal_d$ (and vice versa).

In principle, we would like to define $\Bcal_d$ as follows.
For a graph $\Gamma$, a $d$-manifold fibered over $\Gamma$ is a smooth manifold $W$ with $\partial W = \emptyset$ and a continuous map $\pi\colon W \to X$ such that
    \begin{enumerate}
        \item $\pi$ is proper (i.e.~preimages of compacts are compact),
        \item the fiber $\pi^{-1}(x)$ is connected for all $x \in X$,
        \item $\pi^{-1}(X \setminus V) \to X \setminus V$ is smooth and a submersion, and
        \item $\pi^{-1}(V) \subset W$ is a codimension $0$ submanifold with boundary.
    \end{enumerate}
However, its quite difficult and tedious to make this well-defined and functorial.

\subsubsection{Non-unital modular operads.}
Instead of defining an entire modular operad, we will only define its non-unital part, i.e.~we won't be specifying the unit map $\fre \actto \frc_2$ and all its associated coherence.
\begin{defn}\label{defn:quasi-collapse}
    A graph map $f\colon \Gamma \to \Lambda$ is a \hldef{quasi-collapse} if $f_{\rm v}\colon V_+(\Gamma) \to V_+(\Lambda)$ is surjective.
    We let $\Gr^\qc \subset \Gr$ denote the subcategory that contains all objects but only the quasi collapse morphisms.
\end{defn}
This category still contains all the inert morphisms, so we can make sense of the Segal condition.
\begin{defn}
    A \hldef{non-unital modular operad} is a functor $\Ocal\colon \Gr^\qc \too \An$ that satisfies the Segal condition.
\end{defn}

\begin{thm}[\cite{modular}]\label{thm:non-unital}
    Restriction along the inclusion $i\colon \Gr^\qc \to \Gr$ induces a functor
    \[
        i^*\colon \ModOp \too \mrm{nuModOp}.
    \]
    This functor is replete, i.e.~it induces a monomorphism on mapping spaces, and it is fully faithful on maximal subgroupoids.
\end{thm}

Thus, to construct a modular operad it suffices to construct a non-unital modular operad and show that it is in the essential image of $i^*$; for those non-unital modular operads the theorem guarantees that there is a unique extension to a modular operad.
There is a concrete description of the essential image of $i^*$ in terms of ``quasi-unital modular operads'', which makes this strategy feasible, but we won't go into the details here.

\subsubsection{The non-unital manifold modular operad.}
To define $\Bcal_d$ as a non-unital modular operad we need some preliminary constructions.
Recall that a cut-system in a compact oriented $d$-manifold $M$ is a codimension $1$ submanifold $S \subset M$ that is coorientable and contains the boundary of $M$.
We can now precisely define its dual graph.

\begin{defn}
    The \hldef{dual graph $\Delta(M,S)$} of a cut system is defined as follows:
    its vertices are the path components of $M \setminus S$ and
    its half-edges are pairs of a component of $S$ and a coorientation on it.
    (A coorientation of $S_0 \subset S$ can be described as an equivalence class of nowhere-zero normal vector fields up to positive rescaling.)
    The involution $\dagger$ reverses the coorientation and the root map $r\colon A \to V$ sends a component $S_0 \subset S$ to the component of $M\setminus S$ into which the coorientation points, assigning $\infty$ if $S_0\subset \partial M$ and the coorientation points out of $M$.
\end{defn}

Now to construct $\Bcal_d\colon \Gr^\qc \to \An$ we define for each graph $\Gamma \in \Gr^\qc$, a topologically enriched groupoid $B_d(\Gamma)$.
This groupoid has as objects triples $(M,S,\alpha)$ of a connected compact oriented $d$-manifold $M$ with a cut-system $S \subset M$ and an isomorphism of graphs $\alpha \colon \Delta(M,S) \cong \Gamma$.
The space of morphisms $(M, S, \alpha) \to (N, R, \beta)$ is the space of those diffeomorphisms $\varphi\colon M \to N$ that satisfy $\varphi(S) = R$ and for which the induced isomorphism of dual graphs satisfies $\beta\circ \Delta(\varphi) = \alpha$.
When the graph $\Gamma$ is the edge $\fre$ we simply set $B_d(\fre)$ to be the groupoid of closed oriented $(d-1)$-manifolds.
Every graph map $f\colon \Gamma \to \Lambda$ induces a functor
\begin{align*}
    B_d(f)\colon B_d(\Gamma) &\too B_d(\Lambda)  \\
    (M,S,\alpha) &\longmapsto (N, S', \alpha')
\end{align*}
where $N$ is obtained from $M$ by deleting all components of $M\setminus S$ that correspond to vertices being deleted (i.e.~sent to $\infty$) by $f$, cutting along all components of $S$ that correspond to edges being cut by $f$ (i.e.~those edges $\{a,a^\dagger\}$ where $f_{\rm a}^{-1}(a)$ has more than one element).
The cut system $R$ is obtained from $S$ by removing all components that correspond to edges $e = \{a,a^\dagger\}$ that are being collapsed by $f$ (i.e.~where $f_{\rm a}^{-1}(a) =\emptyset$).
We then obtain our modular operad by taking the realization of this groupoid
\[
    \Bcal_d\colon \Gr^\qc \xtoo{B_d} \TopGpd \xtoo{|-|} \An.
\]
This is quasi-unital and thus by \cref{thm:non-unital} uniquely extends to a modular operad.

In dimension $d=2$ we get that 
\[
    \Bcal_2(\fre) = B\SO(2)  \qqand
    \Bcal_2(\frc_k) = \bigsqcup_{g \ge 0} B\Diff(\Sigma_{g,k}).
\]

\begin{example}\label{ex:handlebody-modular-operad}
    We can also define a variant $\Bcal_d^\partial$ of $\Bcal_d$ where the $d$-manifolds are allowed to have boundary.
    For $d=3$ this has an interesting sub modular operad $\hldef{\Hcal}\subset \Bcal_3^\partial$, the \hldef{handlebody modular operad}.
    Its colours are $2$-disks and its operations are handlebodies with marked disks in their boundary:
    \[
        \Hcal(\fre) = B\Diff(D^2) \simeq B\SO(2)  \qqand
        \Hcal(\frc_k) = \bigsqcup_{g \ge 0} B\Diff_{\sqcup_k D^2}((S^1 \times D^3)^{\natural g})
    \]
    There is a map of modular operads
    \[
        \partial\colon \Hcal \too \Bcal_2,
    \]
    and it follows from the Smale conjecture ($\Diff(D^3) \simeq \SO(3)$, as proven by Hatcher \cite{Hatcher1981}) and Smale's theorem $\Diff(S^2) \simeq \SO(3)$ \cite{Smale59} that this map induces an equivalence in genus $0$.
\end{example}

\section{Envelopes and duality}

\subsection{Variants of modular \captioninfty-operads}
There are various notions closely related to modular operads, whose definition is entirely analogous to modular operads, except that the category of graphs is replaced by another category.
From here on we will often leave the ``$\infty$'' in modular \operads{} implicit.

\subsubsection{Properads.}
One of the most important variants for us will be properads, which are built on directed acyclic graphs.

\begin{defn}
    The category of directed graphs $\hldef{\dGr}$ is defined as the pullback
    \[\begin{tikzcd}
        \dGr \ar[r] \ar[d, "E"'] \ar[dr, phantom,very near start, "\lrcorner"] & \Gr \ar[d, "A"] \\
        \Fin^\op \ar[r, "-\times \Ctwo"] & (\Fin^{\BCtwo})^\op.
    \end{tikzcd}\]
    In other words, a directed graph is a graph $\Gamma$ together with a splitting of $A_\Gamma$ as $E_\Gamma \times \Ctwo$.
    The category of \hldef{directed acyclic graphs} is the full subcategory $\hldef{\daGr} \subset \Gr$ on those directed graphs that do not have directed cycles.
\end{defn}

\begin{example}
    \Cref{fig:Ucal-labelling} on \pageref{fig:Ucal-labelling} depicts a map of directed graphs and both graphs shown are in the full subcategory $\daGr \subset \Gr$ of directed acyclic graphs.
\end{example}

We can also make sense of the Segal condition for functors $\daGr \to \An$, as we shall make precise shortly (\cref{defn:graph-pattern}).
The elementary objects here are the directed edge $\fre$ and the directed corollas $\hldef{\frc_{k,l}}$ for all $k,l\ge 0$.
(To obtain $\frc_{k,l}$ the edges of $\frc_{k+l}$ are oriented in such a way that the vertex has $k$ incoming and $l$ outgoing edges.)
\Cref{fig:elementary-composition} on page \pageref{fig:elementary-composition} depicts several directed corollas.
Such functors are called \hldef{properads}, and we let
\[
    \hldef{\Prpd} \subset \Fun(\daGr, \An)
\]
denote the category of properads.

\begin{example}
    Restriction along the forgetful functor $\varphi\colon \daGr \to \Gr$ preserves the Segal condition and thus defines a functor
    \[
        \varphi^*\colon \ModOp \too \Prpd .
    \]
    For a modular operad $\Ocal$ the resulting properad $\varphi^*\Ocal$ has the same space of colours (though it has forgotten about the $\Ctwo$-action on this space) and its spaces of operations are
    \[  
        (\varphi^*\Ocal)(\frc_{k,l}) = \Ocal(\varphi(\frc_{k,l})) = \Ocal(\frc_{k+l}).
    \]
    We will often abuse notation and simply denote $\varphi^*\Ocal$ by $\Ocal$.
\end{example}

\begin{example}\label{ex:underlying-sketch}
    For every symmetric monoidal \category{} $\Ccal$ we will in \cref{thm:envelope} encounter a properad $\Ucal(\Ccal)$.
    While constructing this coherently is quite a bit of work (and essentially the main theorem of \cite{properads}) we can already describe it conceptually.
    The value of $\Ucal(\Ccal)(\Gamma)$ at a graph $\Gamma$ is the space of all labellings of $\Gamma$ by $\Ccal$ in the following sense.
    A labelling assigns to each edge $e \in E_\Gamma$ of $\Gamma$ an object $x_e \in \Ccal$ and to each vertex $v \in V_\Gamma$ a morphism
    \[
        \alpha_v \colon \bigotimes_{e \in E_\Gamma(v)^{\rm in}} x_e \too \bigotimes_{e \in E_\Gamma(v)^{\rm out}} x_e
    \]
    whose source is the tensor product of labels edges entering $v$ and whose target is the tensor product of labels of edges exiting $v$.
    The functoriality of this with respect to inert maps $f\colon \Gamma \intto \Lambda$ is given by forgetting the part of the labelling associated to $f^{-1}(\infty)$, the functoriality with respect to edge contractions is by composing morphisms, and the functoriality with respect to subdivisions is by introducing identity morphisms.
\end{example}

\begin{figure}[ht]
    \centering
    \def\svgwidth{.8\linewidth}
    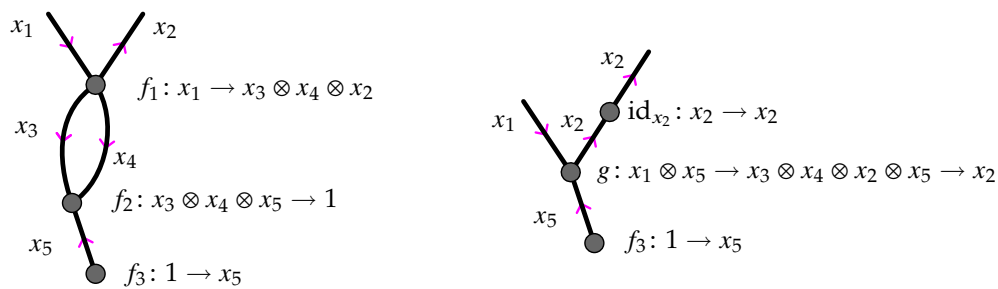
    \caption{A map of graphs $f\colon \Gamma \to \Lambda$.
    The labelling of $\Gamma$ by objects and morphisms in a symmetric monoidal \category{} $\Ccal$ defines an element in $\Ucal(\Ccal)(\Gamma)$.
    The labelling of $\Lambda$ is the image under the (active) graph map that contracts the double-edge and subdivides the top right edge.
    Note that in computing $g$ as a composite of (tensors of) $f_1$ and $f_2$ we have to use that $\otimes$ is symmetric.
    Also note that for this composition to make sense we need that the edges we contract do not form a directed cycle.}
    \label{fig:Ucal-labelling}
\end{figure}

\subsubsection{Graph patterns.}
Properads are only one of many variants of the definition of modular operads one might want to consider.
All of these variants are defined by first specifying a variant of the graph category and then considering functors to $\An$ that satisfy an analogue of the Segal condition.
To have a framework for these examples, we introduce the following definition.
\begin{defn}\label{defn:graph-pattern}
    A \hldef{graph pattern} is a \category{} $\G$ with a functor $\pi\colon\G \to \Gr$ that admits cocartesian lifts for inerts.
    We let $\G^\xint \subset \G$ denote the wide subcategory where morphisms are cocartesian lifts of inerts and let $\G^\el \subset \G^\xint$ denote the full subcategory on all objects that map to $\fre$ or $\frc_k$ in $\Gr$.
    Then we define the category of $\G$-Segal spaces as the full subcategory
    \[
        \Seg(\G) \coloneq \Seg(\G; \An) \subset \Fun(\G, \An) 
    \]
    on those functors $F\colon \G \to \An$ such that $F_{|\G^\xint}$ is right Kan extended from $\G^\el$.
\end{defn}

\begin{example}
    We list several ``combinatorial'' examples of graph patterns in \cref{tab:graph-patterns}.
\begin{table}[h]
    \centering
    \begin{tabular}{c|lc}
         Graph pattern & allowed graphs & (complete) Segal objects \\
         \hline
         $\Gr$& all (connected) graphs & modular operads \\
         $\Tree$& all contractible graphs & cyclic operads \\
         $\mbf{linGr} \simeq \Dop\sslash\Ctwo$& all graphs such that $\Gamma^+ \cong S^1$,\tablefootnote{
            These are exactly the linear graphs $\mfr{l}_n$ from \cref{ex:linear-graph}.
         }& \categories{} with involution \\
         \hline
         $\dGr = \Un(\mrm{or})$& (con.) directed graphs & wheeled properads\\
         $\daGr$& directed acyclic graphs & properads \\
         $\dTree$& directed trees & dioperads \\
         $\dTree^{\rm 1-out} \simeq \Omega^\op$& rooted trees\tablefootnote{
            Here we think of rooted trees as connected directed graphs with the property that edge vertex has exactly one outgoing edge.
         } & operads \\
         $\mbf{dlinGr}\simeq \Dop$& directed linear graphs\tablefootnote{
            A directed linear graph is a graph obtained by subdividing the directed edges. Its underlying graph is isomorphic to $\mfr{l}_n$ for some $n$ and all its directed edges point in the same direction. 
         } & \categories{} \\
         \hline
         $\gGr = \Un(\gr)$& genus graded graphs & graded modular operads \\
         $\gGr^{\le g}$& genus-restricted graphs & genus-restricted mod. op. \\
    \end{tabular}
    \caption{Some graph patterns and the name for the resulting notion of Segal spaces.
    See \cite[Section 6]{Hackney-graph-categories} for a more detailed overview, though note that Hackney works with the opposites of the categories considered here.
    See \cref{defn:gGr} for graded graphs.}
    \label{tab:graph-patterns}
\end{table}
    Another type of example can be obtained as follows.
    Suppose that $\Vcal\in \SM$ is a symmetric monoidal \category{}, i.e.~a functor $\Vcal \colon \Fin_* \to \Cat$ satisfying the Segal condition and let $\Vcal^\otimes \to \Fin_*$ denote its unstraightening. 
    (This is how symmetric monoidal \categories{} are encoded in \cite{HA}.)
    Then we can define a graph pattern $\Gr^{\Vcal} \coloneq \Gr \times_{\Fin_*} \Vcal^\otimes$.
    This is a category where objects are graphs whose vertices are decorated by objects of $\Vcal$.
    The resulting notion of Segal spaces gives a definition of \emph{modular operads enriched in $\PSh(\Vcal^\op)$} where we equip the \category{} of presheaves with the Day convolution symmetric monoidal structure.
    By adding the assumption that certain presheaves are representable we can thus define the notion of a modular operad enriched in $\Vcal$.
    This is an instance of the approach to enrichment outlined in \cite{patterns3} and the reader interested in enrichment is strongly encouraged to read about it in \cite{patterns3}.
\end{example}

\begin{rem}
    The functor $\pi\colon \dGr \to \Gr$ is a left fibration and in fact it is exactly the unstraightening of the orientation modular operad $\mrm{or}\colon \Gr \to \An$ from \cref{ex:orientation}.
    It thus follows from \cite[Proposition 3.2.5]{HK21} that left Kan extension along $\pi$ induces an equivalence
    \[
        \pi\colon \Seg(\dGr) \simeq \Seg(\Gr)_{/\mrm{or}} = \ModOp_{/\mrm{or}}.
    \]
    In other words, a wheeled properad is the same as a modular operad with a map to the orientation modular operad.
    See also \cite[Theorem 6.23 and Remark 6.25]{Hackney-graph-categories}.
\end{rem}

\subsection{The monoidal envelope}
In this section we will import the main result of \cite{properads} as a black-box. It concerns an adjunction that allows us to build the free symmetric monoidal \category{} on an \properad{}.
The right adjoint will give use an \category{} version of the construction from \cref{ex:underlying} and \cref{ex:underlying-sketch}.

\begin{thm}[\cite{properads}]\label{thm:envelope}
    There is an adjunction
    \[
        \Env\colon \Prpd \adj \SM \cocolon \Ucal
    \]
    that has the various useful properties discussed below.
\end{thm}

To describe the objects and morphisms of $\Env(\Ocal)$ we will need the free commutative monoid $\xF(X) = \bigsqcup_{n\ge 0} X^{\times n}_{h\Sigma_n}$. 
(One can show that this is indeed the forget-free monad of the adjunction $\CMon(\An) \adj \An$.)
The objects of $\Env(\Ocal)$ are freely generated by the colours of $\Ocal$ whereas the morphisms of $\Env(\Ocal)$ are freely generated by the operations of $\Ocal$, i.e.~
\[
    \Env(\Ocal)^\simeq = \xF(\Ocal(\fre)) 
    \qqand
    \Ar(\Env(\Ocal))^\simeq = \xF\left(\bigsqcup_{k,l \ge 0} \Ocal(\frc_{k,l})\right).
\]
\begin{example}\label{ex:envelope-operads}
    As \operads{} are just \properads{} where every operation has exactly one output we can also take monoidal envelopes of \operads{} and this recovers the concept of a ``free PROP on an operad'' and the monoidal envelope of \cite[\S2.2.4]{HA}.
    For example, the envelope of the terminal \operad{} $\Ebb_\infty$ is the symmetric monoidal category of finite sets under disjoint union $\Env(\Ebb_\infty) = \Fin$.
    More generally, if $\Ebb_n$ denotes the little $n$-disk operad, then its envelope is $\Env(\Ebb_n) = \Disk_n^{\rm fr}$, the symmetric monoidal \category{} of disjoint unions of framed $n$-disks and framing-preserving%
    \footnote{
        To be precise we take as objects disjoint unions of standard disks $D^n \subset \Rbb^n$ with their induced framing and embeddings $\sqcup_k D^n \hookrightarrow \sqcup_l D^n$ such that on each component of the source the restricted framing agrees with the standard framing up to rescaling.
        These embeddings are also known as rectilinear.
        (Equivalently, we could also consider embeddings that come with an isotopy between the restricted framing and the standard framing.)
    } 
    embeddings between them.
\end{example}

The envelope of the terminal properad $\Ocal = *$ is the \category{} of cospans of finite sets
\[
    \Env(*) \stackrel{!}{=} \Csp
\]
where objects are finite sets and a morphism $A \to B$ is a cospan $A \to X \leftarrow B$.
The general case of $\Env(\Ocal)$ can be thought of as a labelled version of $\Csp$: objects are finite sets labelled by the colours of $\Ocal$ and morphisms are cospans of finite sets labelled by operations of $\Ocal$.
For example, we recover the surface category as the envelope
\[
    \Env(\Bcal_2) \stackrel{!}{=} \Bord_2
\]
and more generally the envelope of $\Bcal_d$ is $\Bord_d$ and the envelope of the handlebody modular operad $\Hcal$ from \cref{ex:handlebody-modular-operad} is the handlebody bordism category, which is a subcategory $\Hbdy \subset \Bord_3^\partial$ whose objects are disjoint unions of $2$-disks and whose morphisms are disjoint unions of handlebodies.

\begin{rem}
    By slicing over $\Env(*) = \Csp$ the adjunction also gives an alternative description of \properads{}.
    Let $\Prpd^{\rm cpl} \subset \Prpd$ denote the full subcategory of complete \properads{} (see \cref{defn:complete}).
    Then the envelope induces a fully faithful functor
    \[
        (\Env(-) \to \Env(*))\colon \Prpd^{\rm cpl} \hookrightarrow \SMover{\Csp}
    \]
    whose essential image consists of those symmetric monoidal functors $\pi\colon \Pcal \to \Csp$ that are ``equifibered'' in the sense that the square
    \[\begin{tikzcd}
        \Pcal \times \Pcal \ar[r, "\otimes"] \ar[d, "\pi \times \pi"'] \ar[dr, phantom, very near start, "\lrcorner"]&
        \Pcal \ar[d, "\pi"] \\
        \Csp \times \Csp \ar[r, "\sqcup"] & \Csp
    \end{tikzcd}\]
    is cartesian.
    The functor $\pi\colon \Pcal \to \Csp$ factors through the subcategory $\Fin \subset \Csp$ of ``forward maps'' if and only if $\Pcal$ is an operad, i.e.~if and only if every operation in $\Pcal$ has exactly one output.
    In this case \cref{thm:envelope} recovers the main result of \cite{HK21}, see also \cite[\S4.3]{envelopes}.
\end{rem}

Using computations of free \properads{} one can determine the colours and operations of $\Ucal(\Ccal)$ for any $\Ccal \in \SM$.
It turns out to be an $\infty$-categorical version of the multi-colour modular operad from \cref{ex:underlying}.
Its colours are the objects of $\Ccal$ and its operations are maps between tensor products in $\Ccal$:
\[
    \Ucal(\Ccal)(\fre) = \Ccal^\simeq
    \qqand
    \Ucal(\Ccal)(\frc_{k,l}) = \colim_{\substack{x_1,\dots,x_k\\y_1,\dots,y_l} \in \Ccal^\simeq} \Map_\Ccal(x_1 \otimes \dots \otimes x_k, y_1 \otimes \dots \otimes y_l)
\]
The value of $\Ucal(\Ccal)$ on a graph $\Gamma$ is thus the space of labellings of edges by objects of $\Ccal$ and vertices by morphisms such that the source of the morphisms is the tensor product of the incoming edge labellings and the target is the tensor product of the outgoing edge labellings.
Thus, $\Ucal(\Ccal)$ coherently implements the sketched description from \cref{ex:underlying-sketch}.

\subsubsection{Algebras over \properads.}
Using the envelope adjunction we can define algebras over an \properad{} in a symmetric monoidal \category{}.

\begin{defn}
    Let $\Ccal \in \SM$ be a symmetric monoidal \category{} and $\Pcal \in \Prpd$ an \properad{}.
    Then we define the space of $\Pcal$-algebras in $\Ccal$ as the space of properad maps $\Pcal \to \Ucal(\Ccal)$
    \[
        \hldef{\Alg_\Pcal^{\rm prpd}(\Ccal)} \coloneq \Map_{\Prpd}(\Pcal, \Ucal(\Ccal)).
    \]
\end{defn}

Because of the adjunction $\Env\dashv \Ucal$ we can equivalently describe algebras as symmetric monoidal functors out of the envelope:
\[
    \Alg_\Pcal^{\rm prpd}(\Ccal) = \Map_{\Prpd}(\Pcal, \Ucal(\Ccal))
    \simeq \Map_{\SM}(\Env(\Pcal), \Ccal).
\]
In fact, this description might be preferable, as it defines an \emph{\category{}} of algebras over $\Pcal$ (where morphisms are symmetric monoidal natural transformations) rather than just an \groupoid{}.
If $\Pcal$ is an \properad{} where every operation has exactly one output (i.e.~it is an \operad{}) then this recovers the usual definition of algebras over an \operad{} as for example defined in \cite{HA}.

As we have that $\Env(\Bcal_d) = \Bord_d$ we get that $\Ccal$-valued TFTs are algebras over the properad $\Bcal_d$:
\[
    \Fun^\otimes(\Bord_d, \Ccal) \simeq \Alg_{\Bcal_d}^{\rm prpd}(\Ccal)
\]
(One can show that the left side is always a groupoid.)

\subsection{Completeness}
We briefly need to discuss the (slightly annoying) subject of ``completeness'' for modular operads.
(This was mostly skipped in the lecture.)
Let us first consider the case of $\Dop$-Segal spaces.
The \hldef{Rezk nerve} is the functor
\begin{align*}
    \xN\colon \Cat &\too \Fun(\Dop, \An) \\
    \Ccal &\longmapsto ([n] \mapsto \Map_{\Cat}([n], \Ccal))
\end{align*}
that turns a \category{} $\Ccal$ into the simplicial space $\xN_\bullet \Ccal$ which remembers the spaces of functors $[n] \to \Ccal$ for all $[n] = \{0\le \dots \le n\} \in \Dop$.
For example, $\xN_0\Ccal = \Ccal^\simeq$ is the space of objects of $\Ccal$ and $\xN_1\Ccal = (\Ar(\Ccal))^\simeq$ is the space of morphisms of $\Ccal$.
It's a theorem of Joyal--Tierney \cite{JT06, HS25-rezk-nerve} that this functor is fully faithful, and its essential image are those simplicial spaces that are ``complete Segal spaces'', i.e.~that
\[
    \xN\colon \Cat \xtoo{\simeq} \CSeg(\Dop) \subset \Fun(\Dop, \An).
\]
Here a simplicial space $X_\bullet$ is called Segal if the canonical maps
\[
    (\rho_1, \dots, \rho_n)\colon X_n \too X_1 \times_{X_0} \dots \times_{X_0} X_1 
\]
are equivalence for all $n$.
(To make sense of this just think of $\Dop$ as a subcategory of $\Gr$ whose objects are the linear graphs $[n] \hat{=} \mfr{l}_n$, though note that this is not a full subcategory.)

To define completeness, let $X_1^{\rm eq} \subset X_1$ denote the subspace of those $1$-simplices that are ``equivalences'' in the sense that they admit both-sided inverses under the composition defined via the Segal condition.
The degeneracy map $s_0$ always factors as
\[
    s_0\colon X_0 \too X_1^{\rm eq} \subseteq X_1
\]
and $X_\bullet$ is called \hldef{complete} if the map $X_0 \to X_1^{\rm eq}$ is an equivalence.

\begin{exc}
    Convince yourself that $\xN_\bullet \Ccal$ is always a complete Segal space.
\end{exc}

\begin{defn}\label{defn:complete}
    A modular \operad{} $\Ocal \colon \Gr \to \An$ is called \hldef{complete} if the restriction
    \[
        \Dop \xtoo{[n] \mapsto \mfr{l}_n} \Gr \xtoo{\Ocal} \An
    \]
    is complete. (It is always a Segal space.)
    We make the same definition for properads, cyclic operads, etc.
\end{defn}

Let us denote by $\ModOp^{\rm cpl} \subset \ModOp$ denote the full subcategory of complete modular operads. 
It then follows by abstract nonsense that there is a localization adjunction
\[
    L\colon \ModOp \adj \ModOp^{\rm cpl} \cocolon \mrm{include}
\]
where $L$ is the \hldef{completion functor}.
The non-formal part is the following.
\begin{thm}[\cite{modular}]\label{thm:completion}
    A morphism $f\colon \Ocal \to \Pcal$ is inverted by $L$ if and only if it is a \hldef{Dwyer--Kan equivalence}, i.e.~it satisfies the following two properties:
    \begin{enumerate}
        \item Fully faithful: for all $k \ge 0$ the square
        \[\begin{tikzcd}
            \Ocal(\frc_k) \ar[d] \ar[r, "f"] & \Pcal(\frc_k) \ar[d] \\
            \Ocal(\fre)^{\times k} \ar[r, "f"] & \Pcal(\fre)^{\times k}
        \end{tikzcd}\]
        is cartesian.
        \item Essentially surjective: 
        for every colour $x \in \Pcal(\fre)$ there exists a colour $y \in \Ocal$ such that $f(y)$ is ``isomorphic'' to $x$ in the sense that there are invertible $2$-ary operations relating them.
    \end{enumerate}
\end{thm}

\subsection{Rigid properads}

\subsubsection{Multi-mapping spaces and composition.}
For a properad $\Pcal$ and colours $x_1,\dots,x_k, y_1,\dots,y_l \in \Pcal(\fre)$ we define the \hldef{multi-mapping-space} as the pullback
\[\begin{tikzcd}
    \hldef{\Pcal(x_1,\dots,x_k; y_1,\dots,y_l)} \ar[r] \ar[d] \ar[dr, very near start, phantom, "\lrcorner"] & \Pcal(\frc_{k,l}) \ar[d] \\
    \{((x_1,\dots, x_k), (y_1,\dots, y_l))\} \ar[r] & \Pcal(\fre)^{\times k} \times \Pcal(\fre)^{\times l}
\end{tikzcd}\]
Note that with this notation a map of properads is fully faithful in the sense of \cref{thm:completion} if and only if it induces equivalences on all multi-mapping-spaces.

For every properad $\Pcal$ and every two-vertex graph $\Gamma \in \daGr$ that is built from two corollas $\frc_{k,l+r}$ and $\frc_{r+k',l'}$ by gluing $r$ of their edges (see \cref{fig:elementary-composition}), the Segal condition allows us to define a composition map as a zig-zag
\[
    \Pcal(\frc_{k,l+r}) \times_{\Pcal(\fre)^{\times r}} \Pcal(\frc_{r+k',l'}) \xleftarrow[\qquad]{\simeq} \Pcal(\Gamma) \too \Pcal(\frc_{k+k',l+l'}).
\]
On multi-mapping spaces this induces maps 
\begin{align*}
    \circ_{w_1,\dots,w_r} \colon &
    \Pcal(w_1,\dots,w_r,x_1',\dots,x_{k'}'; y_1',\dots,y_{l'}') \times \Pcal(x_1,\dots,x_k; y_1,\dots,y_l,w_1,\dots,w_r) \\
    &\too \Pcal(x_1,\dots,x_k, x_1',\dots,x_{k'}'; y_1,\dots,y_l,y_1',\dots,y_{l'}')
\end{align*}
which are subject to certain associativity constraints.
This combinatorial nightmare is neatly encoded in the category $\daGr$ of directed acyclic graphs.
See \cite{HRY15} for a more careful description of this structure and for the proof that it is faithfully encoded in $\Fun(\daGr, \Sets)$.

\begin{figure}[ht]
    \centering
    \def\svgwidth{.8\linewidth}
\begingroup%
  \makeatletter%
  \providecommand\color[2][]{%
    \errmessage{(Inkscape) Color is used for the text in Inkscape, but the package 'color.sty' is not loaded}%
    \renewcommand\color[2][]{}%
  }%
  \providecommand\transparent[1]{%
    \errmessage{(Inkscape) Transparency is used (non-zero) for the text in Inkscape, but the package 'transparent.sty' is not loaded}%
    \renewcommand\transparent[1]{}%
  }%
  \providecommand\rotatebox[2]{#2}%
  \newcommand*\fsize{\dimexpr\f@size pt\relax}%
  \newcommand*\lineheight[1]{\fontsize{\fsize}{#1\fsize}\selectfont}%
  \ifx\svgwidth\undefined%
    \setlength{\unitlength}{640.75942152bp}%
    \ifx\svgscale\undefined%
      \relax%
    \else%
      \setlength{\unitlength}{\unitlength * \real{\svgscale}}%
    \fi%
  \else%
    \setlength{\unitlength}{\svgwidth}%
  \fi%
  \global\let\svgwidth\undefined%
  \global\let\svgscale\undefined%
  \makeatother%
  \begin{picture}(1,0.18965883)%
    \lineheight{1}%
    \setlength\tabcolsep{0pt}%
    \put(0,0){\includegraphics[width=\unitlength,page=1]{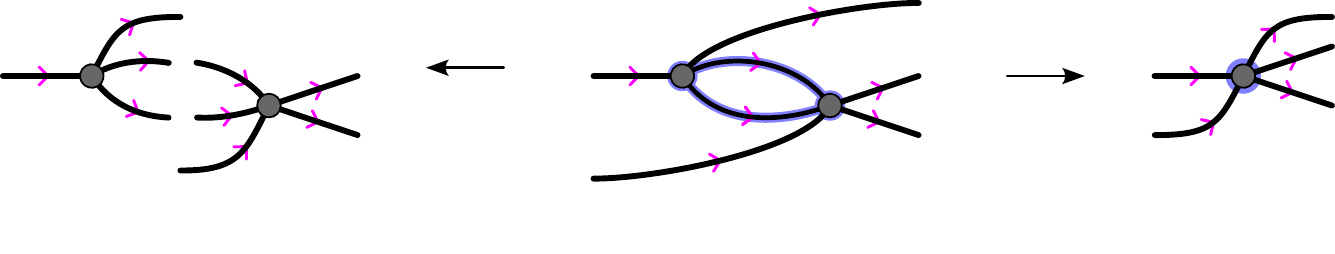}}%
    \put(0.46714656,0.00233243){\color[rgb]{0,0,0}\makebox(0,0)[lt]{\lineheight{1.25}\smash{\begin{tabular}[t]{l}$\Gamma = \frc_{2,3} \cup_{\fre \sqcup \fre} \frc_{3,2}$\end{tabular}}}}%
    \put(0.04726621,0.00233243){\color[rgb]{0,0,0}\makebox(0,0)[lt]{\lineheight{1.25}\smash{\begin{tabular}[t]{l}$\frc_{1,3}$\end{tabular}}}}%
    \put(0.90992369,0.00233243){\color[rgb]{0,0,0}\makebox(0,0)[lt]{\lineheight{1.25}\smash{\begin{tabular}[t]{l}$\frc_{2,3}$\end{tabular}}}}%
    \put(0.17998275,0.00233243){\color[rgb]{0,0,0}\makebox(0,0)[lt]{\lineheight{1.25}\smash{\begin{tabular}[t]{l}$\frc_{3,2}$\end{tabular}}}}%
    \put(0,0){\includegraphics[width=\unitlength,page=2]{elementary-composition.pdf}}%
  \end{picture}%
\endgroup%

    \caption{The graphs used in defining the composition maps. The right map collapses the internal edges (highlighted in blue) and the left maps are inert maps that restrict to the two corollas.}
    \label{fig:elementary-composition}
\end{figure}

\begin{example}
    For the underlying properad $\Ucal(\Ccal)$ of a symmetric monoidal \category{} $\Ccal \in \SM$ we get
    \[
        \Ucal(\Ccal)(x_1,\dots,x_k; y_1,\dots,y_l) = \Map_\Ccal(x_1 \otimes \dots \otimes x_k, y_1 \otimes \dots \otimes y_l).
    \]
    Here the composition morphisms are literally given by composing, at least after we suitably tensor with identity morphisms:
    \[
        \beta \circ_{w_1,\dots,w_r} \alpha = (\id_{y_1\otimes \dots\otimes y_l} \otimes \beta) \circ (\alpha \otimes \id_{x_1' \otimes \dots \otimes x_{k'}'}).
    \]
\end{example}

\subsubsection{Duality in properads.}
Using the notation of multi-mapping spaces we can translate the notion of a dualisable object in a (symmetric) monoidal category to the setting of properads.
\begin{defn}
    Two colours $x,y \in \Pcal(\fre)$ are called \hldef{dual} if there exist operations $e \in \Pcal(y,x; \emptyset)$ and $c \in \Pcal(\emptyset; x,y)$ such that the two composites
    \[
        e \circ_x c \simeq \id_y \in \Pcal(y;y)
        \qqand
        e \circ_y c \simeq \id_x \in \Pcal(x;x)
    \]
    are homotopic to the respective identity operations.
    In this case we say that $e$ is a \hldef{duality pairing} and $c$ a \hldef{duality copairing} between $x$ and $y$.
    A properad $\Pcal$ is called \hldef{rigid} if all its colours are dualizable, and we let
    \[
        \hldef{\Prpd^{\rig}} \subset \Prpd
    \]
    denote the full subcategory on the rigid properads.
\end{defn}

\begin{example}
    Two colours of the underlying modular operad $\Ucal(\Ccal)$ are dual if and only if they are dual (in the usual sense for monoidal categories) as objects of the symmetric monoidal category $\Ccal$.
\end{example}

\begin{example}
    If $\Ocal$ is a modular operad, then we can obtain a properad $\varphi^\ast\Ocal$ by restricting along $\varphi\colon \daGr \to \Gr$.
    This properad is always rigid: the automorphism of $\fre \in \Gr$ induces a map $(-)^\vee\colon \Ocal(\fre) \to \Ocal(\fre)$ and the identity operation $\id_x \in \Ocal(\frc_2)$ induces both $e \in (\varphi^*\Ocal)(x^\vee,x; \emptyset)$ and $c \in (\varphi^*\Ocal)(\emptyset;x,x^\vee)$, which compose to the identity.
\end{example}

\subsubsection{Rigid properads are modular operads.}
If a properad is rigid, the duals and duality (co)pairings turn out to be unique in a suitable sense.
Thus, duality induces an involution $(-)^\vee\colon \Pcal(\fre) \to \Pcal(\fre)$ on the space of colours, and we can use the (co)pairings to turn outputs into inputs and vice versa:
\[
    \Pcal(x_1,\dots,x_k; y, z_1,\dots, z_l) \simeq \Pcal(x_1, \dots, x_k, y^\vee; z_1,\dots,z_l)
\]
Thus we only really need to remember operations without outputs and a way of gluing them, in other words, this data should be equivalent to that of a modular operad.
This is precisely the content of the following theorem.
\begin{thm}[\cite{modular}]\label{thm:rigid properads}
    Restriction along the forgetful functor $\varphi\colon \daGr \to \Gr$ induces an equivalence
    \[
        \ModOp^{\cpl} \xtoo{\simeq} \Prpd^{\cpl, \rig} \subset \Prpd
    \]
    between the \category{} of complete modular operads and the \category{} of complete and rigid properads.
\end{thm}
\begin{proof}[Proof idea]
    On the most basic level, what we need to do is to show that given a complete rigid properad $\Pcal$, we can construct a functor $\Qcal\colon \Gr \to \An$ with $\varphi^*\Qcal \simeq \Pcal$.
    The actual theorem is just a more precise version of this statement.
    To define $\Qcal$ on a graph $\Gamma \in \Gr$ we proceed as follows.
    Pick $\Gamma \hookrightarrow \Gamma'$ a subdivision of $\Gamma$ that admits an acyclic orientation, and pick such an acyclic orientation $\fro$.
    ($\Gamma$ itself might not be in the essential image of $\daGr \to \Gr$ as it might have loops, i.e.~edges whose source and target agree.)
    Given such a directed acyclic subdivision, we could try to set
    \(
        \Qcal(\Gamma) \coloneq \Pcal(\Gamma',\fro),
    \)
    but this does not have a chance of yielding a well-defined functor on $\Gr$ as even the homotopy type of $\Qcal'(\Gamma)$ depends on the way we chose to subdivide $\Gamma \hookrightarrow \Gamma'$.
    However, we can define a subspace
    \[
        \Qcal(\Gamma) \subset \Pcal(\Gamma', \fro)
    \]
    where we only allow those labelling of $\Gamma'$ by $\Pcal$, which ``label each new vertex invertibly''.
    More precisely, this means that we take only those $x \in \Pcal(\Gamma', \fro)$ such that for every bivalent vertex $\alpha\colon \Gamma' \intto \frc_2$, which did not exist in $\Gamma$, the resulting element $\alpha_*(x) \in \Pcal(\frc_2,\fro_{|\dots})$ is an invertible morphism if it's in $\Pcal(\frc_{1,1})$, a duality pairing if it's in $\Pcal(\frc_{2,0})$, and a duality copairing if it's in $\Pcal(\frc_{0,2})$.
    Using completeness and rigidity one can then show that any two choices of directed acyclic subdivision of $\Gamma$ yield equivalent results for $\Qcal(\Gamma)$.
    Moreover, one can show that the category of directed acyclic subdivisions of $\Gamma$ is weakly contractible and that thus we have only made a contractible choice in picking one.
    Making all of the above precise essentially yields a proof of the theorem.
\end{proof}

\subsubsection{The modular operad of dualisable objects.}
The simplest way of obtaining rigid properads is to just discard all the colours that are not dualisable.
We let $\Pcal^\dual \subset \Pcal$ denote the full subproperad on the dualisable colours.
This defines a right adjoint
\[
    \begin{tikzcd}
        \Prpd^\rig \ar[r, hook, shift left = 1] & \ar[l, shift left = 1] \Prpd \cocolon (-)^\dual
    \end{tikzcd}
\]
(This full inclusion also has a left adjoint for formal reasons, but this left adjoint freely adds duals and is generally not as easy to compute.)
In particular, when $\Ccal$ is a symmetric monoidal \category{}, then $(\Ucal(\Ccal))^\dual \subset \Ucal(\Ccal)$ is a complete and rigid properad.
By \cref{thm:rigid properads} it thus acquires the structure of a modular operad.
Concatenating adjunctions we get a right adjoint
\[
    \Env\colon \ModOp^\cpl \simeq \Prpd^{\cpl,\rig} \adj \Prpd^\cpl \adj \SM \cocolon \Udual
\]

\begin{example}\label{ex:Udual}
    The value of $\Udual(\Ccal)$ on a graph $\Gamma$ is the space of ``labellings of $\Gamma$ by dualisable objects in $\Ccal$''.
    Such a labelling consists of a $\Ctwo$-equivariant map $x\colon A_\Gamma \to \Ccal^\simeq$ and a choice of $\alpha_v \in \Map(\bigotimes_{a \in s^{-1}(v)} x_a, \unit)$ for all $v \in V_\Gamma$, see \cref{fig:Ucal-dual-labelling}.
    This means that to each edge $\{a,a^\dagger\}$ we assign a pair of an object $x_a$ and its dual $x_{a^\dagger} = x_a^\vee$ and to each vertex a morphism $\alpha_v \colon x_{a_1} \otimes \dots \otimes x_{a_k} \to \unit$ where $a_1, \dots, a_k$ are the arcs that end at $v$.
    The functoriality of $\Udual(\Ccal)$ is such that when we contract an edge $\{a,a^\dagger\}$ the two incident labels $\alpha_{s(a)}$ and $\alpha_{s(a^\dagger)}$ are ``contracted'' using the pairing $x_a \otimes x_{a^\dagger} \to \unit$.
    That this sketched data indeed can be assembled into a modular operad is one of the main applications of \cref{thm:rigid properads}.
    We can think of $\Udual(\Ccal)$ and its functoriality as encoding the ``string calculus for dualisable objects in $\Ccal$''.
\end{example}

\begin{figure}[ht]
    \centering
    \def\svgwidth{.8\linewidth}
    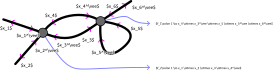
    \caption{A graph $\Gamma$ and a labelling that by dualizable objects and morphisms in a symmetric monoidal \category{} $\Ccal$ that defines a point in $\Udual(\Ccal)(\Gamma)$.
    Note that the arcs are labelled by dualizable objects such that the dual arc is labelled by the dual object -- without choosing an orientation of the graph this is not the same as a map $E(\Gamma) \to \Ccal^{\dbl}$.}
    \label{fig:Ucal-dual-labelling}
\end{figure}

\begin{rem}
    The modular operad $\Ucal_{\rm vect}$ from \cref{ex:underlying} is Dwyer--Kan equivalent to the modular operad $\Ucal(\Vect)$.
    The only difference between the two is that in $\Ucal_{\rm vect}$ we chose symmetric self-duality data for each vector space, but we can forget this extra data to get a map $\Ucal_{\rm vect} \to \Ucal(\Vect)$.
    (Precisely defining this map would require opening the black-box \cref{thm:envelope}, which we won't do here.)
    This map is fully faithful (assuming that the black-box does what it's supposed to) and it is also essentially surjective because every vector space admits a self-duality. 
    
    Note that it is sometimes convenient to build a version of $\Ucal(\Ccal)$ where the colours aren't just object of $\Ccal$, but rather self-dual objects of $\Ccal$ equipped with symmetric self-duality data.
    This equivalently amounts to replacing the space of objects of $\Ucal(\Ccal)$ with its $\Ctwo$-fixed points.
    The reason this can be useful is that in the resulting modular operad $\Ucal^{\rm self-dual}(\Ccal)$ the action of $\Ctwo$ on $\Ucal^{\rm self-dual}(\Ccal)(\fre)$ is trivial.
    There'll always be a fully faithful comparison map $\Ucal^{\rm self-dual}(\Ccal) \to \Udual(\Ccal)$, but in general it won't be essentially surjective as objects in $\Ccal$ need not be isomorphic to their dual.
    Moreover, it can be a little confusing to work with $\Ucal^{\rm self-dual}(\Ccal)$ as it is not complete: for example, there are several non-isomorphic self-dualities on $\Rbb^n$ (e.g.~the positive definite and the negative definite one), but they all define isomorphic colours in the modular operad $\Ucal^{\rm self-dual}(\Ccal)$ in the sense that there are invertible binary operations between them.
    (This is analogous to how we might define a category whose objects are finite dimensional Hilbert spaces, but where morphisms are just arbitrary linear maps -- such a category is just equivalent to the category of finite dimensional vector space.)
\end{rem}

\subsubsection{TFTs as modular algebras.}
We define modular algebras analogously to how we defined algebras over properads.
\begin{defn}
    For a modular operad $\Ocal$ and a symmetric monoidal \category{} $\Ccal$ we define the \category{} of modular $\Ocal$-algebras in $\Ccal$ as
    \[
        \hldef{\Algmod_\Ocal(\Ccal)} \coloneq \Map_{\ModOp}(\Ocal, \Udual(\Ccal)).
    \]
\end{defn}
By \cref{thm:rigid properads} this agrees with the \category{} $\Alg_{\varphi^*\Ocal}^{\rm prpd}(\Ccal)$.
Combining all the above we have
\[
    \Fun^\otimes(\Bord_d, \Ccal) \simeq \Map_{\Prpd}(\Bcal_d, \Ucal(\Ccal)) \simeq \Map_{\ModOp}(\Bcal_d, \Udual(\Ccal)) = \Algmod_{\Bcal_d}(\Ccal),
\]
so a $\Ccal$-valued $d$-dimensional TFT is equivalent to a modular $\Bcal_d$-algebra in $\Ccal$.

\subsubsection{The $1$D cobordism hypothesis.}
As another application of \cref{thm:rigid properads} we can classify TFTs in dimension $1$, in accordance with the Baez--Dolan cobordism hypothesis in dimension $1$ \cite{baez1997higher}.
To do this, we will need the adjunction
\[
    \Free\colon \An^{\BCtwo} \adj \ModOp \cocolon \col
\]
where the right adjoint $\col$ sends $\Ocal$ to its space of colours $\col(\Ocal) = \Ocal(\fre)$.
The left adjoint exists for formal reasons, but it also turns out to be fairly computable.
We'll see more examples of how to do this later -- it basically amounts to computing a left and a right Kan extension, as explained in a \cref{exc:B1-is-free}.
For now, we'll use the following fact as a black-box.
\begin{lem}\label{lem:B1-is-free}
    The $1$-dimensional manifold modular operad $\Bcal_1$ is equivalent to the free modular operad $\Free(\Ctwo)$.
\end{lem}

The following statement is the $1$-dimensional cobordism hypothesis, see \cite{baez1997higher,lurie2008classification,Harpaz-cobordism,GradyPavlov}.
\begin{cor}
    For every symmetric monoidal \category{} $\Ccal$, evaluating a $1$D TFT at the positive point $\pt_+$ yields an equivalence
    \[
        \ev_{\pt_+}\colon \Fun^\otimes(\Bord_1, \Ccal) \xtoo{\simeq} (\Ccal^\dbl)^\simeq.
    \]
\end{cor}
\begin{proof}
    By concatenating adjunctions we get
    \[
        \Env(\Free(-)) \colon \An^{\BCtwo} \adj \ModOp \adj \SM  \cocolon \col(\Udual(-)) = (-^{\dual})^\simeq
    \]
    so that the claim follows from $\Bord_1 = \Env(\Bcal_1) = \Env(\Free(C_2))$ and
    \[
        \Fun^\otimes(\Bord_1, \Ccal) \simeq \Map_{\An^{\BCtwo}}(\Ctwo, (\Ccal^\dbl)^\simeq) \simeq (\Ccal^\dbl)^\simeq. \qedhere
    \]
\end{proof}

\begin{exc}
    Describe the universal property of the unoriented $1$-dimensional bordism category $\Bord_1^{\rm unor}$.
    What about more general tangential structures?
\end{exc}

\begin{exc}\label{exc:B1-is-free}
    Prove \cref{lem:B1-is-free}. 
    To do this, consider the following full inclusions of graph categories.
    \[
        \BCtwo = \{\fre\} \xtoo{i} \mbf{linGr} \xtoo{j} \Gr^{\rm biv} \xtoo{k} \Gr
    \]
    Here $\mbf{linGr}$ is the full subcategory on the subdivisions of $\fre$ (i.e.~all graphs such that $\Gamma^+\setminus \{\infty\}$ is homeomorphic to $\Rbb$), and $\Gr^{\rm biv}$ is the full subcategory of bivalent graphs.
    These functors induce restrictions
    \[
        \Fun(\BCtwo,\An) \xleftarrow{i^*} \Fun(\mbf{linGr}, \An) \xleftarrow{j^*} \Fun(\Gr^{\rm biv}, \An) \xleftarrow{k^*} \Fun(\Gr, \An)
    \]
    all of which preserve the Segal condition.
    Show that the left Kan extension $i_!$, the right Kan extension $j_*$ and the left Kan extension $k_!$ all preserve Segal objects and that thus we have adjunctions
    \[\begin{tikzcd}
        \An^{\BCtwo} \rar["i_!", shift left = 1] &
        \Seg(\mbf{linGr}) \lar["i^*", shift left =1 ]\rar["j_*"', shift right = 1] &
        \Seg(\Gr^{\rm biv}) \rar["k_!", shift left = 1] \lar["j^*"', shift right = 1] &
        \Seg(\Gr). \lar["k^*", shift left = 1]
    \end{tikzcd}\]
    Moreover show that $j^*$ is conservative and that thus $j^* \dashv j_*$ is an adjoint equivalence.
    Conclude that we have an adjunction
    \[
        k_!j_*i_! \colon \An^{\BCtwo} \adj \Seg(\Gr) \cocolon i^*j^*k^*.
    \]
\end{exc}

\begin{exc}
    Describe a symmetric monoidal \category{} $\Dcal$ such that for every symmetric monoidal category $\Ccal$ there is an equivalence
    \[
        \Fun^\otimes(\Dcal, \Ccal) \simeq \Ccal \times_{\Ccal \times \Ccal} \Ar(\Ccal).
    \]
    What is $\Map_\Dcal(\unit, \unit)$? Show that if $f\colon x \to x$ is an endomorphism of a dualisable object in $x$, then $\mrm{Tr}(f^{\circ n}) = \ev_x \circ (f^{\circ n} \otimes \id_{x^\vee}) \circ \mrm{coev}_x$ has a $C_n$-action.
\end{exc}

\section{Cyclic operads}
The goal of this lecture is to understand $2$d TFTs ``restricted to genus $0$''.
Genus $0$ surfaces can be glued to one another, but they cannot be self-glued as that would increase genus.
Thus, they do not form a modular operad, but just a cyclic operad.
It will turn out that cyclic operads and algebras over them are easier to understand as they are closely related to ordinary ($\infty$-)operads, and we can quantify the difference between these two notions.
Using this theory we will also be able to classify ``handlebody TFTs'' as those turn out to be freely generated from genus $0$.

\subsection{Cyclic operads}
\subsubsection{Operads vs cyclic operads.}
Recall that we defined cyclic operads as functors $\Ocal \colon \Tree \to \An$ satisfying the Segal condition.
The \category{} of cyclic operads is then the full subcategory
\[
    \hldef{\CycOp} \subset \Fun(\Tree, \An)
\]
on these functors.
Every cyclic operad has an underlying operad, which we can obtain by precomposing with the forgetful functor
\[
    \Omega^\op = \dTree^{\rm 1-out} \xtoo{\psi} \Tree \xtoo{\Ocal} \An.
\]
We'll denote this restriction by $\underline{\Ocal} \coloneq \psi^*\Ocal$.
What was the space of $(k+1)$-ary operations in the cyclic operad $\Ocal(\frc_{k+1})$ is now the space of $k$-ary operations in the underlying operad $\ul{\Ocal}(\frc_{k,1}) = \Ocal(\frc_{k+1})$.
If we use the shorthand notation for arity $k$ operations this leads to the following confusing degree shift
\[
    \ul{\Ocal}(k) = \ul{\Ocal}(\frc_{k,1}) = \Ocal(\frc_{k+1}) = \Ocal(k+1).
\]
\begin{rem}
    Getzler and Kapranov \cite[(1.4)]{GetzlerKapranov1998} deal with this degree shift by introducing the notation $\Ocal((k+1)) \coloneq \Ocal(\frc_{k+1})$ and $\Ocal(k) \coloneq \ul{\Ocal}(k) = \ul{\Ocal}(\frc_{k,1})$.
    As we generally work with many-coloured (cyclic) operads we will rarely use the notation $\Ocal(k)$ anyway, so it will be more intuitive to distinguish the two sides by writing $\frc_{k,1}$ vs.~$\frc_{k+1}$.
\end{rem}

Classically, a cyclic operad can be defined as an operad $\Pcal$ together with an extension of the $\Sigma_{k}$-action on $\Pcal(\frc_{k,1})$ to a $\Sigma_{k+1}$-action, subject to a few axioms.%
\footnote{
    To be precise, this gives an ``un-augmented'' version of cyclic operads where the minimum arity is $\Ocal(1) = \Pcal(0)$ (if $\Pcal = \ul{\Ocal}$).
    The definition of cyclic operad we have here also has a space $\Ocal(0) = \Ocal(\frc_0)$, which is entirely forgotten when passing to the underlying operad.
    One can obtain operations of cyclic arity $0$ by composing two operations of cyclic arity $1$.
    This isn't possible in the underlying operad as the two corresponding operations of arity $0$ can't be composed.
}
In this perspective, a cyclic operad is an operad where one can (cyclicly?) permute the output and inputs.
Coherently, this approach won't work directly as there are a lot of compatibilities between these group-actions and the higher coherence of an $(\infty)$-operad, which it would be impracticable to spell out.
Instead, we'll later describe cyclic operads as operads equipped with a ``dualizing module''.

We'll think about operations in cyclic operas as not having outputs, but only inputs. One can then compose two such operations whenever they have dual inputs.
More precisely, let 
\[
    \Ocal(x_1,\dots,x_k) \coloneq \mrm{fib}_{(x_1, \dots, x_k)}\left( \Ocal(\frc_k) \too \Ocal(\fre)^{\times k} \right)
\]
denote the space of $k$-ary operations with fixed input colours $x_1$, \dots, $x_k$.
Then a cyclic operad has an involution $(-)^\vee\colon \Ocal(\fre) \to \Ocal(\fre)$ on its space of colours and composition operations
\[
    \Ocal(x_1,\dots, x_{k-1}, y) \times \Ocal(z_1, \dots, z_{l-1}, y^\vee) \too \Ocal(x_1, \dots, x_{k-1}, z_1, \dots, z_{l-1}).
\]

\subsubsection{Examples of cyclic operads.}
To get a better idea of cyclic operads, let's consider some examples. See table \ref{tab:cyclic operads} on page \pageref{tab:cyclic operads}.
\afterpage{
\begin{landscape}
\begin{table}[]
    \centering
    \begin{tabular}{c|cccccc}
         Operad & $\ul{\Ocal}(\frc_{k,1})$ & $\Env(\ul{\Ocal})$ & $\Alg_\Ocal(\Ccal)$ & $\Alg_\Ocal(\Cat)$ & $\Ocal(\frc_{k+1})$ \\
         \\
         \hline
         \\
         $\mrm{Com} = \Ebb_\infty$ & $*$ & $\Fin$ & commutative algebras & sym. mon. cat. &  $*$ \\
         \\
         $\mrm{Ass} = \Ebb_1$ & $\Sigma_k = \{\text{total orders on }\ul{k}\}$ & $\Assoc$ & associative algebras & monoidal cat. & \{\text{cyclic orders on }\ul{k+1}\}  \\
         \\
         $\Ebb_n$ & $\sqcup_k D^n \hookrightarrow D^n$ rectilinear & $\Disk_n^{\rm fr}$ & $\Alg_{\Ebb_1}(\dots\Alg_{\Ebb_1}(\Ccal)\dots)$ & braided mon. & generally not cyclic\\
         \\
         $\Ebb_n^{\rm SO} = \Ebb_n / \SO(n)$ & $\sqcup_k D^n \hookrightarrow D^n$ oriented & $\Disk_n^{\rm or}$ & $\Alg_{\Ebb_n}(\Ccal)^{h\SO(n)}$ & ribbon br. mon. & $\checkmark$ \\
         \\
         $\Bcal_d^{\rm sph}\subset \Bcal_d$ & $B\Diff_\partial(S^d \setminus \sqcup_{k+1} \Rbb^d)$ & $\subset \Bord_d$ & ? & ? & $B\Diff_\partial(S^d \setminus \sqcup_{k+1} \Rbb^d)$ \\
         \\
         $\Mcal \in \ModOp$ & $\Mcal(\frc_{k+1})$ & $\subset \Env(\Mcal)$ & ? & ? & $\Mcal(\frc_{k+1})$ 
    \end{tabular}
    \caption{Examples of (cyclic) operads}
    \label{tab:cyclic operads}
\end{table}
\end{landscape}
}
Here for any modular operad $\Mcal$ we can think of it as a cyclic operad by restricting it along the full inclusion $i\colon \Tree \hookrightarrow \Gr$.
Restriction along $i$ induces a functor
\[
    i^*\colon \ModOp = \Seg(\Gr) \too \Seg(\Tree) = \CycOp ,
\]
which we can think of as forgetting the self-gluing operations of a modular operads.
If we for example take $\Bcal_d$, then the cyclic operad obtained this way has an interesting cyclic suboperad $\Bcal_d^{\rm sph} \subset \Bcal_d$ that is defined by restricting to those colours that are $(d-1)$-spheres and those operations that are of the form $S^d \setminus \sqcup_k \Rbb^d$.
Note that this would not be a well-defined modular suboperad of $\Bcal_d$, as it is not closed under self-gluing. 

Some of the cyclic operads listed in the table actually coincide.
Recall the handlebody modular operad $\Hcal \subset \Bcal_3^\partial$ from \cref{ex:handlebody-modular-operad}.
Using Smale's theorem $\Diff(S^2) \simeq \SO(3)$ \cite{Smale59} and the Smale conjecture $\Diff(S^3) \simeq \SO(4)$ (as proved by Hatcher \cite{Hatcher1981}) one can show that there are equivalences of cyclic operads
\[
    \Ebb_2^{\rm SO} \simeq \Bcal_2^{\rm sph} \simeq \Hcal^{\rm sph}
\]
and the Smale conjecture also implies 
\[
    \Ebb_3^{\rm SO} \simeq \Bcal_3^{\rm sph}.
\]
Proving these equivalences mostly amounts to constructing an equivalence of operads and then using \cref{thm:cyclic}, which we'll see later in this section.

\subsubsection{From operads to cyclic operads.}
We will later be able to describe handlebody-TFTs as cyclic algebras over a cyclic operad and even for the usual $2$-TFTs the starting point of the classification will be cyclic algebras.
To make this kind of description actually useful we'll have to describe cyclic algebras in more concrete terms.
Our goal will thus be to write
\[
    \text{cyclic algebras over } \Ocal = \text{ algebras over } \ul{\Ocal} + \text{ extra data}
\]
Since we have a fairly good handle of algebras over $\ul{\Ocal}$ (for the $\Ocal$s involved) this will work as a ``concrete description'' as long as the extra data is manageable.
In fact, as cyclic algebras are just maps of cyclic operads, it will be more convenient to prove a version of this one category level up, i.e.~to prove a statement of the form
\[
    \text{cyclic operads } = \text{ operads } + \text{ extra data}.
\]
Here the relevant additional structure are certain right modules.

\subsection{Right modules over operads}
\begin{defn}
    A right%
    \footnote{
        The name comes from the point of view where one defines (one-coloured) operads as algebras for the composition product in symmetric sequences. (A version of this also works for multiple colours.)
        In this setting operads are associative algebras, so we can talk about left and right module over them. 
        Left modules turn out to be algebras (over the operad) in spaces and right modules turn out to be, well, right modules in the above sense.
    }
    module over an operad $\Ocal$ is a functor
    \[
        X\colon \Env(\Ocal)^\op \too \An,
    \]
    and we write $\hldef{\RMod_\Ocal} \coloneq \PSh(\Env(\Ocal))$ for the \category{} of right modules. 
    Note that here we are using the symmetric monoidal envelope construction for operads, as defined in \cite[\S2.2.4]{HA}, which is a special case of the envelope construction for properads, as discussed in \cref{ex:envelope-operads}.
\end{defn}

Since an object in the envelope is a formal tensor product of colours of the operad, we can think of the right module as contravariantly assigning to each such tensor product a space
\[
    x_1 \boxtimes \dots \boxtimes x_k \longmapsto X(x_1 \boxtimes \dots \boxtimes x_k)
\]
with the functoriality being such that if $\alpha \in \Ocal(y_1,\dots, y_l; x_1)$ is an operation with output colour $x_1$ (or equivalently $x_i$) there is a morphism
\[
    X(\alpha) \colon X(x_1 \boxtimes \dots \boxtimes x_k) \too X(y_1 \boxtimes \dots \boxtimes y_l \boxtimes x_2 \boxtimes \dots \boxtimes x_k).
\]

\subsubsection{Disk presheaves.}
A particular interesting example of right modules is the case of $\Ebb_n^{\rm SO}$ where the monoidal envelope is $\Disk_n$.
Recall that this category has objects $\sqcup_k \Rbb^n$ for $k \ge 0$ and the mapping spaces are the space of smooth embeddings.
Thus, a right module is a functor
\[
    X \colon \Disk_n^\op \too \An.
\]
Such functors are also called disk presheaves and the category of right modules is
\[
    \RMod_{\Ebb_n^{\rm SO}} = \PSh(\Disk_n).
\]

\begin{example}\label{ex:E_M}
    If $M$ is any $n$-manifold we can define an $\Ebb_n^{\rm SO}$-right module $E_M$ by
    \[
        \hldef{E_M}(\sqcup_k \Rbb^n) \coloneq \Emb(\sqcup_k \Rbb^n, M).
    \]
    This can be thought of as recording all the framed configuration spaces of $M$ as well as certain collapse maps between them.
    This is the basis of \emph{Goodwillie--Weiss embedding calculus}, where one for example defines
    \[
        T_\infty\Emb(M, N) \coloneq \Map_{\PSh(\Disk_n)}(E_M, E_N).
    \]
    While it is not necessary to think of these disk presheaves as right modules over an operad in order to do embedding calculus, this has turned out to be a very fruitful perspective, see \cite{KraKup-operadic-embedding-caluclus}.
\end{example}

\begin{example}\label{ex:factorisation-homology}
    Another concept that fits into this framework is \emph{factorization homology}, which we will encounter again later in \cref{ex:cylic-EnSO-algs}.
    Suppose for simplicity that $\Ccal$ is a symmetric monoidal category that has sufficient colimits and that for every $x \in \Ccal$ the functor $x \otimes -\colon \Ccal \to \Ccal$ preserves colimits.
    (E.g.~$\Ccal \in \CAlg(\PrL)$.)
    An $\Ebb_n^{\rm SO}$-algebra $A$ in $\Ccal$ can be written as a symmetric monoidal functor $\Disk_n \to \Ccal$ whose value at $\sqcup_k \Rbb^n$ is $A^{\otimes n}$.
    By composing with the functor $\Map_\Ccal(-,x)$ for some object $x \in \Ccal$ we obtain another disk presheaf $\Map_\Ccal(A^{\otimes -}, x)$.
    The mapping space from $E_M$ into this disk presheaf can be expressed in terms of factorization homology as
    \[
        \Map_{\PSh(\Disk)}(E_M, \Map_\Ccal(A^{\otimes -}, x)) \simeq \Map_\Ccal\big(\smallint_M A, x\big).
    \]
    Via the Yoneda lemma this in fact uniquely characterizes $\smallint_M A \in \Ccal$ as we vary $x$.
\end{example}

\subsubsection{A graph pattern for right modules.}
In order to put right modules on the same footing as the other algebraic structures we're considering it will be useful to write them as Segal spaces for a certain graph pattern.
This will in particular avoid the use of the monoidal envelope $\Env\colon \Op \to \SM$ whose construction we haven't really discussed.
What we will describe is not the category $\RMod_\Ocal$ of right modules over a specific operad $\Ocal$, but rather the more general category where objects are a pair of an operad and a right module over it.
Formally, this obtained as the cartesian unstraightening%
\footnote{
    Here and elsewhere we are liberally ignoring size issues.
}
\[
    \hldef{\RMod} \coloneq \Un^{\rm cart}\left(\Op^\op \xtoo{\Env} (\SM)^\op \too \Cat^\op \xtoo{\PSh(-)} \Cat\right).
\]

\begin{figure}[ht]
    \centering
    \def\svgwidth{.99\linewidth}
    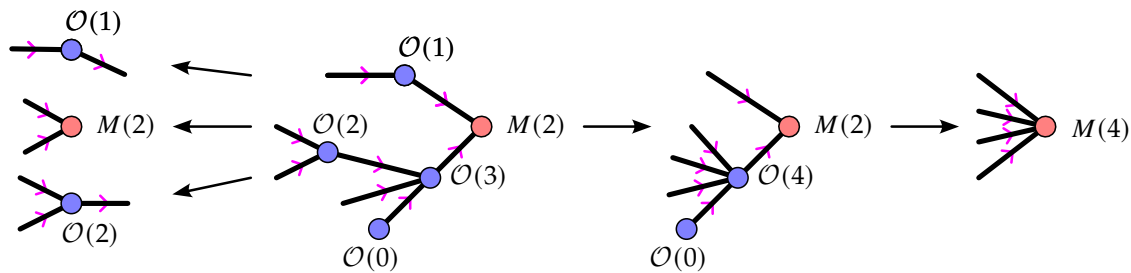
    \caption{Several objects and morphisms in $\dTree^{\le1\rm-out}$.
    \cref{prop:Seg-rMod} allows us to think of a $\dTree^{\le1\rm-out}$-Segal space as a pair of an operad $\Ocal$ and a right module $M$.}
    \label{fig:module-action}
\end{figure}

\begin{prop}\label{prop:Seg-rMod}
    There is an equivalence 
    \[
        \Seg(\dTree^{\rm \le1-out}) \simeq \RMod
    \]
    between the category of Segal spaces over $\dTree^{\rm \le1-out}$, the category of directed trees where each vertex has at most one outgoing edge, and the category of operads and right modules.
\end{prop}
\begin{proof}[Proof idea]
    A tree $T \in \dTree^{\rm\le1-out}$ has two types of vertices: those with an outgoing edge, for which the local corolla is $\frc_{k,1}$, and those with no outgoing edges for which the local corolla is $\frc_{k}$.
    The idea is that the former should be labelled by $k$-ary operations of an operad whereas the latter should be labelled by the values of a right module over said operad (when evaluated on $k$-fold tensor products of colours).
    See \cref{fig:module-action}.
    When contracting edges in this type of tree we can either identify two vertices with outgoing edge, in which case we use the operadic multiplication, or two vertices where one of them has an outgoing edges, in which case we use the action of the operad on the right module.
    In order to make this precise one has to look into the details of the construction of the symmetric monoidal envelope $\Env$ and relate it to trees.
\end{proof}

The morphism of graph patterns
\[
    \psi\colon \dTree^{\rm\le1-out} \too \Tree
\]
that forgets the direction of the edges induces a functor
\[
    \psi^*\colon \CycOp = \Seg(\Tree) \too \Seg(\dTree^{\rm\le1-out}) = \RMod.
\]
If we inspect this more closely, we see that it sends a cyclic operad $\Ocal$ to the tuple $(\ul{\Ocal}, \omega_\Ocal)$ where $\ul{\Ocal}$ is the underlying operad as before and $\omega_\Ocal$ is a right module over $\ul{\Ocal}$ given by the cyclic operad itself, in the sense that
\[
    \omega_\Ocal(x_1 \boxtimes \dots \boxtimes x_k) \coloneq \Ocal(x_1,\dots,x_k)
\]
We can think about this as saying that
\[
\text{``every cyclic operad is a right module over itself.''}
\]
Note that, confusingly, this is not the tautological right module structure of a (one-coloured) operad over itself.
Instead, it is an interesting module that involves a shift of the arities by one.
Writing this in a simplified notation that ignores colours we have
\[
    \ul{\Ocal}(k) \simeq \omega_\Ocal(k+1) .
\]

\subsection{Cyclic operads via right modules}
In this section we'll see a theorem that says that the data $(\ul{\Ocal}, \omega_\Ocal)$ of the underlying operad and its canonical right module suffices to reconstruct the cyclic operad $\Ocal$.
\subsubsection{Defining dualizing modules.}
The basic idea is that because we have $\ul{\Ocal}(k) \simeq \omega_\Ocal(k+1)$ and the symmetric group $\Sigma_{k+1}$ acts on the right we can transport the action along the equivalence to recover the symmetry of $\Ocal(k+1) = \ul{\Ocal}(k)$.
The higher coherences should follow by a more sophisticated version of this argument.
However, in order for this to be a useful description the equivalence between the operad and the right module should not be additional coherence data, but rather should emerge from the existing structure.
We can achieve this by imitating the notion of a duality pairing.

\begin{defn}
    Let $\omega\colon \Env(\Ocal)^\op \to \An$ be a right module over an operad $\Ocal$.
    We say that two colours $c, c^\vee \in \col(\Ocal)$ are \hldef{dual with respect to $\omega$} if there is an element $e \in \omega(c^\vee \boxtimes c)$ that is a \hldef{non-degenerate pairing} in the sense that  the composite maps (defined by acting on $e$)
    \begin{align*}
        \Ocal(x_1,\dots, x_k; c) \xtoo{\id \times e} 
        \Ocal(x_1,\dots, x_k; c) \times \omega(c^\vee \boxtimes c) \too \omega(c^\vee \boxtimes x_1 \boxtimes \dots \boxtimes x_k) \\
        \Ocal(x_1,\dots, x_k; c^\vee) \xtoo{\id \times e} 
        \Ocal(x_1,\dots, x_k; c^\vee) \times \omega(c^\vee \boxtimes c) \too \omega(x_1 \boxtimes \dots \boxtimes x_k \boxtimes c)
    \end{align*}
    both are equivalences for any choice of colours $x_1, \dots, x_k \in \col(\Ocal)$.
    We say that $\omega$ is a \hldef{dualizing module} for $\Ocal$ if every colour of $\Ocal$ has a dual (with respect to $\omega$).
\end{defn}

The motivating example for this definition is the following setting, where the above essentially recovers the notion of a duality pairing in a (symmetric) monoidal category.
\begin{example}\label{ex:U(C)-dualizing-module}
    If $\Ocal = \ul{\Ucal}(\Ccal) \in \Op$ is the underlying operad of a symmetric monoidal category $\Ccal$ then we can define a right module over it by 
    \[
        \omega_\Ccal(x_1 \boxtimes \dots \boxtimes x_k) \coloneq \Map_\Ccal(x_1 \otimes \dots \otimes x_k, \unit_\Ccal).
    \]
    Suppose $e\colon x \otimes y \to \unit_\Ccal$ is a duality pairing in $\Ccal$, in the sense that there exists a compatible copairing $c\colon \unit_\Ccal \to y \otimes x$ such that
    \[
        (e \otimes \id_x) \circ (\id_x \otimes c) = \id_x
        \qqand
        (\id_y \otimes e) \circ (c \otimes \id_y) = \id_y.
    \]
    Then we can think of $e$ as a point in $\omega_\Ccal(x \otimes y)$ and as such it will be a non-degenerate pairing because it induces equivalences
    \[
        \Map_\Ccal(z, y) \simeq \Map_\Ccal(z \otimes x, \unit)
        \qqand
        \Map_\Ccal(z, x) \simeq \Map_\Ccal(z \otimes y, \unit).
    \]
    Therefore, if $\Ccal$ is a symmetric monoidal category where every object is dualisable (i.e.~rigid) then $\omega_\Ccal$ will be a dualizing module for $\ul{\Ucal}(\Ccal)$.
    (Note, however, that the converse of this is not true:
    there interesting examples of non-rigid symmetric monoidal categories whose underlying operad admits a dualizing module, these turn out to be%
    \footnote{
        ``Turn out to be'' here means that with some work one can establish an equivalence between (symmetric monodial \categories{} $\Ccal$ plus a dualizing module for $\Ucal(\Ccal)$) and an $\infty$-categorical generalization of the definition in \cite{BoyarchenkoDrinfeld-GV}.
    }
    exactly the Grothendieck--Verdier categories of \cite{BoyarchenkoDrinfeld-GV}, which also appear in \cite{MuellerWoike-GV}.)
\end{example}

\subsubsection{The theorem.}
While the above emulates the notion of a duality pairing this analogy is not quite perfect:
as the relation between a duality pairing and its copairing can be expressed in terms of equations, duality pairings are automatically preserved by symmetric monoidal functors.
The same is not true for non-degenerate pairings in right modules: maps of right modules need not preserve them.
Thus, we need to impose this condition by hand when defining the category of dualizing modules.

\begin{defn}
    The \category{} of operads with dualizing modules is defined as the subcategory $\hldef{\RMod^{\rm dual}} \subset \RMod$ whose objects are those tuples $(\Ocal, \omega)$ where $\omega$ is a dualizing module for $\Ocal$ and whose morphisms are $(f,\alpha) \colon (\Ocal, \omega) \to (\Pcal, \omega')$ where $f\colon \Ocal \to \Pcal$ is any map of operads and $\alpha\colon \omega \to f^*\omega'$ is a map of $\Ocal$ right modules that preserves non-degenerate pairings:
    if $e \in \omega(c^\vee \boxtimes c)$ is non-degenerate, we require that $\alpha(e) \in \omega'(f(c^\vee) \boxtimes c)$ is also non-degenerate.
\end{defn}

The key theorem describing cyclic operads is the following. A version of this for the one-coloured (rational) case was proven by Willwacher \cite{Wil24}.%
\footnote{
    Note that there are some aesthetical differences between Willwacher's theorem and the theorem as it is stated here, which mostly come from the fact that Willwacher works with $1$-coloured cyclic operads while we allow multiple colours.
    This is, for instance, the reason that in our category $\RMod^{\rm dual}$ preserving the duality pairings is a condition that is homotopy invariant, whereas for Willwacher the right module is pointed (by a non-degenerate symmetric pairing) and maps have to strictly preserve this base point.
}
\begin{thm}[\cite{cyclic}]\label{thm:cyclic}
    The forgetful functor $\psi\colon \dTree^{\rm\le 1-out} \too \Tree$ induces an equivalence
    \[
        \psi^\ast \colon \CycOp^{\cpl} \xtoo{\simeq} \RMod^{\dual,\cpl} \subset \RMod.
    \]
\end{thm}
\begin{proof}[Proof idea]
    The proof of this is similar to theorem \cref{thm:rigid properads} by considering for each tree $T$ the category of directed subdivisions of $T$ such that each vertex has at most one outgoing edge.
    This category can be identified with the poset of simplices of $T$ and is thus weakly contractible.
    In assembling the proof one needs to be a little more careful as $\RMod^\dual$ is not a full subcategory of $\RMod$.
\end{proof}

\subsubsection{Examples of dualizing modules.}
By \cref{thm:cyclic} we can promote an operad $\Ocal$ to a cyclic operad, simply by constructing a right module over it and checking that with respect to this right module all colours have a dual.
Often our operads will happen to have a connected space of colours, in which case we only need to find a single non-degenerate duality pairing.

\begin{example}\label{ex:dualizing-modules}
\begin{enumerate}[(1)]
    \item The operad $\Ebb_\infty = \mrm{Com}$ has a dualizing module given by the terminal right module $\omega_{\Ebb_\infty}\colon \Fin^\op \to \An$ with $\omega_{\Ebb_\infty}(A) = *$.
    \item For the operad $\Ebb_1 = \mrm{Ass}$ the envelope is the symmetric monoidal $1$-category $\Assoc$, where objects are finite sets and morphisms are maps of finite sets together with total orderings on each fiber.
    We can define a dualizing right module by
    \begin{align*}
        \omega_{\Ebb_1}\colon \Assoc^\op & \too \An \\
        A & \longmapsto \{\text{cyclic orders on }A\}.
    \end{align*}
    This is functorial in maps in $\Assoc^\op$ because if we have a cyclic order on $A$ and a map $B \to A$ with totally ordered preimages, then we can define a cyclic order on $B$ by replacing each point in $A$ with its totally ordered preimage.
    The duality pairing is given by $e \in \omega_{\Ebb_1}(\ul{2})$ the unique cyclic order on two elements.
    To see that it is a duality pairing we observe that the map
    \[
        \Ebb_1(n) = \{\text{total orders on }\ul{n}\} \too \{\text{cyclic orders on }\ul{n+1}\} = \omega_{\Ebb_1}(n+1)
    \]
    defined by acting on $e$ is indeed a bijection.
    \item 
    The envelope of $\Ebb_n^{\rm SO}$ is the category $\Disk_n$ of (disjoint unions of oriented) $n$-disks.
    We can define a right module over it as the quotient of the natural $\SO(n+1)$-action on the right module $E_{S^{n}}$:
    \begin{align*}
        \omega_{\Ebb_n^{\rm SO}}\colon \Disk_n & \too \An \\
        \sqcup_k \Rbb^n &\longrightarrow \Emb(\sqcup_k \Rbb^n, S^n)\sslash \SO(n+1)        
    \end{align*}
    One checks that this is a dualizing module.
    This equips $\Ebb_n^{\rm SO}$ with the structure of a cyclic operad.
    \item 
    The dualizing module  
    \(
        \omega_\Ccal(x_1\boxtimes \dots \boxtimes x_n) \coloneq \Map_\Ccal(x_1 \otimes \dots \otimes x_n, \unit)
    \)
    from \cref{ex:U(C)-dualizing-module} improves the underlying operad $\ul{\Ucal}(\Ccal)$ of a rigid symmetric monodial \category{} to a cyclic operad.
    (More generally, this works for symmetric monoidal categories that admit a dualizing object in the sense of \cite{BoyarchenkoDrinfeld-GV}.)
\end{enumerate}
\end{example}

\begin{rem}
    It is a theorem of Budney \cite{Budney2008-cyclic} that a certain model of the little framed $n$-disk operad can be promoted to a cyclic operad, and thus that $\EnSO$ ``is'' a cyclic operad. 
    Our approach seems simpler, but much more implicit, as we for instance cannot write a model in topological spaces.
    Note that a priori we now have two cyclic structures on the same operad, but in the case at hand, it should not be too hard to compare our dualizing module with the one Budney constructs and to thus get an equivalence of cyclic operads.
\end{rem}

\begin{exc}
    Check that $\omega_{\Ebb_n^{\rm SO}}$ is indeed a dualizing module.
    Construct a dualizing module for $\Ebb_n$ whenever $S^n$ is a Lie group ($n=1,3$), recovering for $n=1$ the ``cyclic orders'' dualizing module from \cref{ex:dualizing-modules}.
    Show that if $\Ebb_n$ admits a dualizing module, then $S^n$ is a loop space and thus $n=1,3,\infty$. (This last part is hard.)
\end{exc}

\begin{question}
    Is there a geometric model for the cyclic structure on $\Ebb_3$, analogous to the ``conformal'' model Budney \cite{Budney2008-cyclic} gives for the cyclic structure on $\EnSO$?
\end{question}

\subsection{Cyclic algebras}

We can define cyclic algebras analogously to modular algebras as follows.
\begin{defn}
    Given a symmetric monoidal \category{} $\Ccal$ and a cyclic operad $\Ocal$ we can define the \category{} of \hldef{cyclic $\Ocal$-algebras} in $\Ccal$, analogously to how we defined modular algebras, as
    \[
        \hldef{\Algcyc_\Ocal(\Ccal)} \coloneq \Map_{\CycOp}(\Ocal, \Udual(\Ccal))
    \]
    where, by abuse of notation, we let $\Udual(\Ccal)$ denote the cyclic operad obtained from the modular operad of the same name.
\end{defn}

\subsubsection{As algebras over the induced modular operads.}
To give an alternative description, we note that the forgetful functor that we defined as restriction along $i\colon \Tree \hookrightarrow \Gr$ admits a left adjoint
\[
    \Indcycmod \colon \CycOp = \Seg(\Tree) \adj \Seg(\Gr) = \ModOp \cocolon i^\ast.
\]
Using the left adjoint we can rewrite
\[
    \Algcyc_\Ocal(\Ccal) = \Map_{\CycOp}(\Ocal, \Udual(\Ccal))
    \simeq \Map_{\ModOp}(\Indcycmod(\Ocal), \Udual(\Ccal))
    = \Algmod_{\Indcycmod(\Ocal)}(\Ccal) 
\]
\begin{rem}
    That this left-adjoint exists follows from the adjoint functor theorem, and we can write it as the composite of left Kan extension $i_!$ and the ``Segalification'', i.e.~the left adjoint to the full inclusion $\Seg(\Gr) \hookrightarrow \Fun(\Gr, \An)$.
    This Segalification is generally hard to compute, but it turns out that one can change some of the graph patterns involved so that instead of $i_!$ we need to compute the left Kan extension along another functor $j\colon \gGr^{=0} \hookrightarrow \gGr$, which has the property that $j_!$ preserves Segal spaces, so that no Segalification is needed.
    (See \cref{exc:B1-is-free} for an analogous story.)
    We'll see more of this in the final lecture.
\end{rem}

For now, we'll use the following fact as a black-box, but we'll see a bit more about its proof in the next lecture.
\begin{prop}\label{prop:Hbdy-free-on-E2SO}
    There are equivalences
    \[
        \Indcycmod(\Ebb_2^{\rm SO}) \simeq \Indcycmod(\Hcal^{\rm sph}) \simeq \Hcal \subset \Bcal_3^\partial
    \]
    and similarly, $\Indcycmod(\Ebb_3^{\rm SO})$ can be identified as a sub modular operad of $\Bcal_3$ with operations
    \[
        \Indcycmod(\Ebb_3^{\rm SO})(\frc_k) = \coprod_{g \ge 0} B\Diff((S^1\times S^2)^{\#g} \setminus \sqcup_k D^3).
    \]
\end{prop}
\begin{proof}[Proof idea]
    The first statement is due to Giansiracusa \cite{Giansiracusa2011-handlebody}, who proves this using the contractibility of the disk complex. 
    (He uses a slightly different model of the handlebody modular operad, but a similar argument applies here.)
    The second claim can be deduced from \cite[Proposition 3.12]{BBS-finiteness}.
\end{proof}

Combining this we above, we get that symmetric monoidal functors out of the handlebody \category{} are equivalent to $\Ebb_2^{\rm SO}$-algebras.
Recall that the handlebody category is the subcategory $\Hbdy \subset \Bord_3^\partial$ whose objects are disjoint unions of $2$-disks and whose morphisms are handlebodies.
It is the monoidal envelope of the handlebody modular operad.
\begin{cor}
    There are equivalences
    \[
        \Fun^\otimes(\Hbdy, \Ccal) 
        \simeq \Algmod_\Hcal(\Ccal)
        \simeq \Algcyc_{\Ebb_2^{\rm SO}}(\Ccal).
    \]
\end{cor}
In analogy with how TFTs $\Bord_2 \to \LinCat_k$ are modular functors, symmetric monoidal functors $\Hbdy \to \LinCat_k$ are called ``ansular functors'' \cite{MW23}.

\begin{rem}\label{rem:graph-bordism}
    Analogously to \cref{prop:Hbdy-free-on-E2SO} one can also describe the free modular operad on the cyclic operad $\Ebb_\infty$.
    Using Culler--Vogtmann's contractibility of ``outer space'' one can show that it has operations
    \[
        \Indcycmod(\Ebb_\infty)(\frc_k) \simeq \coprod_{n \ge 0} B\mrm{hAut}_{\ul{k}}(\vee_n S^1)
    \]
    and in particular for $k=0$ we get the classifying spaces for the groups $\mrm{Out}(F_n)$ of outer automorphisms of the free groups.
    The monoidal envelope of this modular operad is the subcategory
    \[
        \Env(\Indcycmod(\Ebb_\infty)) \subset \mrm{Cospan}(\An)
    \]
    of the \category{} of cospans of spaces, where objects are finite sets and morphisms are cospans $A \to X \leftarrow B$ such that $X$ is equivalent to a finite $1$-dimensional CW complex.
    This monoidal envelope is also almost equivalent to the graph bordism category used in Galatius' seminal paper on the stable homology of automorphism groups of free groups \cite{Galatius2011}, the only difference is that we allow univalent vertices.
\end{rem}

\subsubsection{Cyclic algebras in terms of right modules.}
Since cyclic algebras are just defined as maps of cyclic operads, we can use \cref{thm:cyclic} to express them as algebras over the underlying operad and a non-degenerate map of right modules.
Concretely, a cyclic $\Ocal$-algebra in $\Ccal$ consists of two pieces of data
\[
    A \in \Alg_{\ul{\Ocal}}(\Ccal) 
    \qqand
    \alpha\colon \omega_\Ocal \to A^*\omega_\Ccal.
\]
We can think of $A$ as a symmetric monoidal functor $\Env(\ul{\Ocal}) \too \Ccal$ that sends $x_1 \boxtimes \dots \boxtimes x_k$ to $A(x_1) \otimes \dots \otimes A(x_k)$, or simply to $A^{\otimes k}$ if $\ul{\Ocal}$ is one-coloured.
Thus, the pulled back right module $A^*\omega_\Ccal$ can be described as
\begin{align*}
    A^*\omega_\Ccal \colon \Env(\ul{\Ocal})^\op & \too \An \\
    x_1 \boxtimes \dots \boxtimes x_k & \longmapsto \Map_\Ccal(A(x_1) \otimes \dots \otimes A(x_k), \unit),
\end{align*}
which we'll usually write simply as $\Map_\Ccal(A^{\otimes-}, \unit)$.
The map $\alpha$ is a \emph{non-degenerate} map of right modules, so it is a map 
\[
    \alpha\colon \omega_\Ocal(-) \too \Map_\Ccal(A^{\otimes -}, \unit)
\]
in $\PSh(\Env(\ul{\Ocal}))$ such that it preserves non-degenerate pairings.
In other words, whenever $e \in \omega_\Ocal(x\boxtimes y)$ is a non-degenerate pairing between two colours $x,y \in \col(\ul{\Ocal})$, the resulting 
\[
    \alpha(e) \in \Map_\Ccal(A(x) \otimes A(y), \unit)
\]
is required to be a duality pairing between $A(x)$ and $A(y)$.

\begin{example}[Cyclic $\Ebb_\infty$-algebras]
    A cyclic algebra over $\Ebb_\infty$ consists of an $\Ebb_\infty$-algebra $A \in \Alg_{\Ebb_\infty}(\Ccal^\dbl)$ and a non-degenerate map $\alpha\colon \omega_{\Ebb_\infty} \to A^*\omega_\Ccal$.
    Since $\omega_{\Ebb_\infty}$ is the terminal right module we can easily compute the space of right module maps as 
    \[  
        \Map_{\PSh(\Fin)}(\omega_{\Ebb_\infty}, A^*\omega_\Ccal)
        \simeq \Map_{\PSh(\Fin)}(* , \Map_\Ccal(A^{\otimes -}, \unit))
        \simeq \lim_{\Fin^\op} \Map_\Ccal(A^{\otimes -}, \unit)
        \simeq \Map_\Ccal(A, \unit)
    \]
    where the last step uses that $\ul{1} \in \Fin^\op$ is initial.
    So the right module map $\alpha$ is uniquely determined by a map $\tau\colon A \to \unit$, which we think of as a ``trace'' on $A$.
    The value of $\alpha$ on $\ul{k}$ is then the composite
    \[
        \alpha_k = (A^{\otimes k} \xtoo{\mu} A \xtoo{\tau} \unit) \in \Map_\Ccal(A^{\otimes k}, \unit)
    \]
    where the first map is the $k$-fold multiplication (i.e.~the unit if $k=0$).
    The non-degeneracy condition thus means exactly that 
    \[
        \alpha_2 = \tau \circ \mu \colon A \otimes A \to \unit
    \]
    is a non-degenerate duality pairing between $A$ and itself.
\end{example}

We can study cyclic $\Ebb_n^{\rm SO}$-algebras similarly.
For simplicity, assume that $\Vcal \in \CAlg(\PrL)$, so that $\Vcal$ is a \category{} with sufficient colimits and with a symmetric monoidal product that preserves them.
\begin{example}[Cyclic $\EnSO$-algebras]\label{ex:cylic-EnSO-algs}
    A cyclic $\EnSO$-algebra in $\Vcal$ consists of an $\EnSO$-algebra $A \in \Alg_{\EnSO}(\Vcal^\dbl)$ (which we may describe as an $\SO(n)$-fixed point in $\Alg_{\En}(\Vcal^\dbl)$) and a non-degenerate map $\alpha\colon \omega_{\EnSO} \to A^*\omega_\Vcal$.
    Recall that $\Env(\EnSO) = \Disk_n$, so that $\RMod_{\EnSO} = \PSh(\Disk_n)$.
    Using the characterization of factorization homology from \cref{ex:factorisation-homology} we can describe the space of right module maps as
    \begin{align*}
        \Map_{\PSh(\Disk_n)}(\omega_{\EnSO}, A^*\omega_\Vcal)
        & \simeq \Map_{\PSh(\Disk_n)}(\Emb(-,S^n)\sslash \SO(n+1) , \Map_\Vcal(A^{\otimes -}, \unit)) \\
        & \simeq \Map_{\PSh(\Disk_n)}(\Emb(-,S^n), \Map_\Vcal(A^{\otimes -}, \unit))^{\SO(n+1)} \\
        & \simeq \Map_{\Vcal}\left(\int_{S^n} A, \unit\right)^{\SO(n+1)}.
    \end{align*}
    Thus, the data of $\alpha$ is equivalent to an $\SO(n+1)$-invariant map
    \[
        \tau\colon \int_{S^n} A \too \unit,
    \]
    which we could also further write as a map $(\int_{S^n} A)\sslash \SO(n+1) \to \unit$.
    We can think of this as a ``symmetric trace'' on $A$.
    The non-degeneracy condition means that the composite map
    \[
        A \otimes A = \int_{D^n \sqcup D^n} A \too \int_{S^n} A \too \unit
    \]
    is a non-degenerate duality pairing between $A$ and itself.
\end{example}

\subsubsection{Calabi--Yau algebras.}
We can summarize the structure we found above as follows.

\begin{defn}
    An \hldef{$\Ebb_\infty$-Calabi--Yau algebra} in $\Ccal$ is a tuple $(A,\tau)$ of an $\Ebb_\infty$-algebra $A$ in $\Ccal$ and a non-degenerate trace $\tau$, i.e.~a map $\tau\colon A \to \unit$ such that it induces a self-duality on $A$ as above.

    Similarly, an \hldef{$\EnSO$-Calabi--Yau algebra} in $\Ccal$ is a tuple $(A,\tau)$ of an $\EnSO$-algebra $A$ in $\Ccal$ and an $\SO(n+1)$-invariant map $\tau\colon \int_{S^n} A \to \unit$ such that it induces a self-duality on $A$ as above.
\end{defn}

\begin{cor}\label{cor:CY-algebras}
    There is an equivalence between cyclic $\Ebb_\infty$/$\EnSO$-algebras in $\Ccal$ and $\Ebb_\infty$/$\EnSO$-Calabi--Yau algebras in $\Ccal$.
\end{cor}

\begin{example}
    For $n=1$ an $\Ebb_1$-Calabi--Yau algebra consists of an $\Ebb_1$-algebra $A$ together with an $\SO(2)$-invariant map $\tau\colon \int_{S^1} A \to \unit$ such that the induced pairing
    \[
        A \otimes A = \int_{D^1 \sqcup D^1} A \too \int_{S^1} A \xtoo{\tau} \unit
    \]
    is a duality pairing.
    In the case of the category of vector spaces $\int_{S^1} A = HH_0(A) = A/[A,A]$, and thus we can think of the trace as a function $\tau\colon A \to k$ that vanishes on the commutator; or equivalently that satisfies $\tau(ab) = \tau(ba)$.
    Thus, this recovers the notion of a symmetric Frobenius algebra.
\end{example}

\begin{rem}
    In \cite[\S4.6.5]{HA} Lurie defines a ``symmetric Frobenius algebra'' to be an $\Ebb_1$-algebra $A$ with a non-degenerate trace $\int_{S^1} A \to \unit$.
    He notes that one can also add $\SO(2)$-invariance data to the definition and that this condition plays an important role in the classification of (extended) $2$D TFTs \cite[Remark 4.6.5.9]{HA}.
\end{rem}

\begin{rem}
    The definition of $\Ebb_1$-Calabi--Yau algebras in the derived category $\Dcal(k)$ of a field $k$ agrees with the definition of a $1$-object case of a proper $A_\infty$-Calabi--Yau category as for example in \cite[Definition 1.5]{BravDykerhoff2019}.
\end{rem}

Combining the above consequences of \cref{thm:cyclic} with the fact that the handlebody modular operad $\Hcal$ is freely generated by the cyclic operad $\EtwoSO$ (\cref{prop:Hbdy-free-on-E2SO}) we get a classification of symmetric monoidal functors out of the handlebody category
\begin{cor}\label{cor:ansular-functors}
    There is an equivalence
    \[
        \Fun^\otimes(\Hbdy, \Vcal) 
        \simeq \{ (A, \tau) \;|\; A \in \Alg_{\EtwoSO}(\Vcal),\, \tau\colon \int_{S^2} A \too \unit \; \SO(3)\text{-invariant and non-deg.}\}
    \]
\end{cor}
In the case of $\Vcal = \LinCat_k$ (up to the caveats in \cref{rem:LinCat-caveat}) this gives an alternative description of the ansular functors from \cite{MW23}.
(This is not very surprising, as our reasoning here is analogous to the reasoning used in \cite{MuellerWoike-GV,MW23} where the authors also use cyclic operads and their envelopes.)

An analogous universal property holds for the category $\Env(\Indcycmod(\Ebb_\infty))$ from \cref{rem:graph-bordism}.

\section{The genus filtration}

In the previous lecture we discussed that the handlebody modular operad $\Hcal$ is freely generated by its genus $0$ part $\Hcal^{\rm sph}$, which in turn is equivalent to $\EtwoSO$.
Together with the description of cyclic operads as operads with dualizing modules, this allowed us to ``classify'' symmetric monoidal functors 
    $\Hbdy = \Env(\Hcal) \to \Ccal$
in terms of $\EtwoSO$-Calabi--Yau algebras in $\Ccal$.
Since the map $\partial \colon \Hcal \to \Mcal = \Bcal_2$ that sends a handlebody to its boundary is a bijection on $\pi_0$, this recovers the classification of $2$D TFTs in $\Vect_k$ as commutative Frobenius algebras. 
For more general targets $\Ccal \in \SM$ the restriction 
\[
    \Fun^\otimes(\Bord_2, \Ccal) \too \Fun^\otimes(\Hbdy, \Ccal)
\]
along $\partial$ won't be an equivalence, but in this lecture we will see that it is the starting point of a convergent tower that computes the space of $2$D TFTs valued in $\Ccal$.
Since the cyclic operads $\EtwoSO \simeq \Hcal^{\rm sph} \simeq \Mcal^{\rm sph}$ agree, we can think of $\Hcal \simeq \Indcycmod(\EtwoSO)$ as a ``genus $0$ approximation'' to $\Mcal$.
We'll now construct higher genus approximations and study their convergence.

\subsection{Filtering modular operads}

\subsubsection{Graded modular operads.}
To define graded modular operads we need to add ``genus gradings'' to our graphs.
This can be done using the grading modular operad from \cref{ex:grading-modular-operad}.
\begin{defn}\label{defn:gGr}
    The category of graded graphs \hldef{$\gGr$} is defined as the unstraightening
    \[
        \hldef{\gGr} \coloneq \Un(\Gr \xtoo{\gr} \Sets \subset \An)
    \]
    of the grading modular operad $\gr$.
\end{defn}

\begin{figure}[ht]
    \centering
    \def\svgwidth{.6\linewidth}
    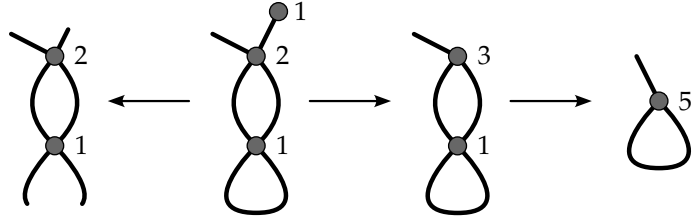
    \caption{Four objects and three maps in $\gGr$. The left-ward map is inert, and the others are active. The first two graphs are in $\gGr^{\le 2}$, but the other two are not.}
    \label{fig:graded-graph-map}
\end{figure}

Therefore, a graded graph is a graph $\Gamma$ together with a grading $l\colon V(\Gamma) \to \Nbb$ and the morphisms in the category are such that they add the grading of vertices that are identified plus one for every loop that is collapsed.
See \cref{fig:graded-graph-map} for some examples of graded graph maps.
Using this, we define the category of graded modular operads as Segal objects
\[
    \grModOp \coloneq \Seg(\gGr) \simeq \Seg(\Gr)_{/\gr} = (\ModOp)_{/\gr}.
\]
(The comparison functor $\grModOp \to \ModOp$ is defined by left Kan extending along $\gGr \to \Gr$. This preserves the Segal condition and sends the terminal graded modular operad to $\gr$. This is an instance of a general principle for left fibrations over algebraic patterns \cite[Proposition 3.2.5]{HK21})

Both the surface and the handlebody modular operad are graded with
\[
    \Mcal(\frc_k^{(g)}) = B\Diff(\Sigma_{g,k})
\]
where $\hldef{\frc_k^{(g)}}$ denotes the arity $k$ corolla $\frc_k$, labelled by $g \in \Nbb$.
Alternatively, we can describe the grading as a map $\Mcal \to \gr$ that records the genus of each of the components.
That this works is of course no coincidence: we chose $\gr$ so that $\gr = \pi_0(\Mcal)$.

\subsubsection{Bounding genus.}
We can define a filtration
\[
    \gGr^{\le 0} \subset \gGr^{\le 1} \subset \gGr^{\le 2} \subset \dots \subset \gGr
\]
where $\hldef{\gGr^{\le g}} \subset \gGr$ is the full subcategory on those graded graphs $(\Gamma,l)$ such that $l(v) \le g$ for all $v \in V_\Gamma$.
We refer to the Segal spaces over these graph patterns as \hldef{genus-restricted modular operads} and denote the category by
\[
    \hldef{\grModOp^{\le g}} \coloneq \Seg(\gGr^{\le g}).
\]
Genus $\le0$ restricted modular operads are in fact just cyclic operads:
\begin{lem}\label{lem:gGr0-slim}
    Restriction along the full inclusion $i\colon \Tree \hookrightarrow \gGr^{\le 0}$ induces an equivalence
    \[  
        i^*\colon \Seg(\gGr^{\le 0}) \xtoo{\simeq} \Seg(\Tree) = \CycOp
    \]
    whose inverse is given by right Kan extension.
    In particular, every Segal space on $\gGr^{\le 0}$ is right Kan extended from the full subcategory $\Tree$.
\end{lem}
\begin{proof}[Proof idea]
    It suffices to show that right Kan extension preserves the Segal condition.
    (Then it will define a fully faithful right adjoint, but $i^*$ is conservative because $\Tree$ contains all the elementary graphs of $\gGr^{\le 0}$, so the adjunction is an equivalence.)
    Because the Segal condition can be expressed as the functor on inerts being right Kan extended from elementaries, this follows by checking that the square of inclusions
    \[\begin{tikzcd}
        \Tree^{\xint} \ar[r, hook, "i^\xint"] \ar[d, "j"] & (\gGr^{\le 0})^{\xint} \ar[d, "k"] \\
        \Tree \ar[r, hook, "i"] & \gGr^{\le 0}
    \end{tikzcd}\]
    satisfies that the Beck--Chevalley transformation $i^\ast k_\ast^{\xint} \to j_\ast (i^\xint)^\ast$ is an equivalence.
    This amounts to showing that the functor
    \[
        \Tree^{\xint} \times_{(\gGr^{\le 0})^{\xint}} (\gGr^{\le 0})^{\xint}_{\Gamma/}
        \too
        \Tree \times_{\gGr^{\le 0}} \gGr^{\le 0}_{\Gamma/}
    \]
    is initial.
    This is indeed the case because this functor has a right adjoint, given by sending $f\colon \Gamma \to \Lambda$ to the inert part $f^\xint\colon \Gamma \intto \Lambda'$ of the unique inert-active-factorization of $f$.
    (This is well-defined, because if $\Lambda$ is a tree and $f^\act\colon \Lambda' \actto \Lambda$ in $\gGr^{\le 0}$ is active, then $\Lambda'$ is also a tree.)
\end{proof}

\begin{rem}
    \cref{lem:gGr0-slim} is an instance of a general principle for algebraic patterns \cite[Definition 14.7]{patterns1}, see also \cite[Corollary 2.64]{Shaul-arity}:
    for any algebraic pattern $\Pcal$ we can form a full subcategory $\Pcal^{\slim} \subset \Pcal$, which contains all those objects $x$ for which there exists an active map $x \actto e$ with $e$ elementary.
    This full subcategory inherits the pattern structure and the restriction along the full inclusion $i\colon \Pcal^\slim \hookrightarrow \Pcal$ induces an equivalence
    \[
        i^*\colon \Seg(\Pcal) \xtoo{\simeq} \Seg(\Pcal^\slim)
    \]
    whose inverse is given by right Kan extension.
    In the example of \cref{lem:gGr0-slim}, if we have an active morphism $\Gamma \actto \frc_k^{(0)}$, then $\Gamma$ must be a tree, as otherwise the corolla would have to be labelled by a positive number.
    Therefore, the slim subpattern of $\gGr^{\le 0}$ is exactly $\Tree$.
\end{rem}

Restricting along the inclusions $i_g^{g'}\colon \gGr^{\le g} \hookrightarrow \gGr^{\le g'}$ preserves the Segal condition, and thus we get a tower
\[
    \grModOp \too \dots \too
    \grModOp^{\le 2} \too \grModOp^{\le 1} \too \grModOp^{\le 0} \simeq \CycOp.
\]
Because the filtration of $\gGr$ was exhaustive (as every graded graph lies in $\gGr^{\le g}$ for some $g$) this tower converges in the sense that
\[
    \grModOp \simeq \lim_{g \to \infty} \grModOp^{\le g}.
\]

\subsubsection{Inducing up.}
So far we have described how to approximate the \category{} of modular operads, but really we would like to approximate objects of this category, such as $\Mcal$, rather than the category itself.
In other words, having found the right categorical setting, we should now ``go down one category level'' to reap the benefits.
(In fact, we'll want to go two category levels, because what we truly care about are modular algebras, which are maps of modular operads.)
To do this, we'll approximate $\Mcal \in \grModOp$ by the modular operad freely built from the restricted modular operads $(i_g^\infty)^\ast \Mcal \in \grModOp^{\le g}$.
This means that we'll consider the left adjoints, for $g\le g'$,
\[
    \hldef{\Ind_g^{g'}} \colon \grModOp^{\le g} \adj \grModOp^{\le g'} \cocolon (i_g^{g'})^\ast.
\]
These left adjoints exist for fairly formal reasons, but more importantly, they are actually computable in terms of left Kan extension.
\begin{lem}
    Left Kan extension along $i_g^{g'}$ preserves Segal spaces and thus defines a left-adjoint
    \[
        \Ind_g^{g'} = (i_g^{g'})_! \colon \grModOp^{\le g} = \Seg(\gGr^{\le g}) \adj \Seg(\gGr^{\le g'}) = \grModOp^{\le g'} \cocolon (i_g^{g'})^\ast.
    \]
    In particular, the left adjoint $\Ind_g^{g'}$ is fully faithful.
\end{lem}
\begin{proof}
    This can be deduced formally from the theory of algebraic patterns \cite{patterns1,patterns3} once one checks that $\Gr$ is a soundly extendable pattern.
    (For this argument it is better to include the nodeless loop $\frs = S^1$ from \cref{war:nodeless-loop} as it ensures that $\Gr$ is extended.)
\end{proof}

\begin{defn}
    For a graded modular operad $\Mcal$ and $0\le g \le \infty$ we define its \hldef{genus $g$ approximation} to be the graded modular operad
    \[
        \hldef{\Ocal^{(g)}} \coloneq \Ind_g^{g'}( (i_g^{g'})^* \Ocal ).
    \]
    These assemble into the \hldef{genus filtration of $\Ocal$}
    \[
        \Ocal^{(0)} \too \Ocal^{(1)} \too \Ocal^{(2)} \too \dots \too \Ocal.
    \]
\end{defn}

Because all these constructions are functorial the genus filtration defines a functor
\[
    \grModOp \too \Fun((\Nbb,\le), \grModOp).
\]
It follows formally that this filtration is convergent, that is, $\Ocal = \colim_{g \to \infty} \Ocal^{(g)}$.

\begin{example}
    By \cref{prop:Hbdy-free-on-E2SO} we have that
    \[
        \Mcal^{(0)} \simeq \Hcal^{(0)} \simeq \Hcal.
    \]
\end{example}

\subsection{The filtration on \texorpdfstring{$\Mcal$}{M}}

\subsubsection{Restricted cut systems.}
In order to compute $\Mcal^{(g)}$ we introduce the following variant of cut systems.
\begin{defn}
    Let $\Sigma$ be a surface and $0\le g\le \infty$.
    We say that $S \subset \Sigma$ is a \hldef{genus $\le g$ cut system} if
    \begin{enumerate}
        \item $S$ is a disjoint union of essential curves 
        \item each component of $\Sigma \setminus S$ has genus at most $g$.
    \end{enumerate}
    The set of all genus $\le g$ cut systems is topologised with the $C^\infty$-topology and forms a topological poset $\hldef{\Cut^g(\Sigma)}$ under $\subseteq$.
\end{defn}

\begin{figure}[ht]
    \centering
    \def\svgwidth{.6\linewidth}
\begingroup%
  \makeatletter%
  \providecommand\color[2][]{%
    \errmessage{(Inkscape) Color is used for the text in Inkscape, but the package 'color.sty' is not loaded}%
    \renewcommand\color[2][]{}%
  }%
  \providecommand\transparent[1]{%
    \errmessage{(Inkscape) Transparency is used (non-zero) for the text in Inkscape, but the package 'transparent.sty' is not loaded}%
    \renewcommand\transparent[1]{}%
  }%
  \providecommand\rotatebox[2]{#2}%
  \newcommand*\fsize{\dimexpr\f@size pt\relax}%
  \newcommand*\lineheight[1]{\fontsize{\fsize}{#1\fsize}\selectfont}%
  \ifx\svgwidth\undefined%
    \setlength{\unitlength}{1012.21971887bp}%
    \ifx\svgscale\undefined%
      \relax%
    \else%
      \setlength{\unitlength}{\unitlength * \real{\svgscale}}%
    \fi%
  \else%
    \setlength{\unitlength}{\svgwidth}%
  \fi%
  \global\let\svgwidth\undefined%
  \global\let\svgscale\undefined%
  \makeatother%
  \begin{picture}(1,0.23999765)%
    \lineheight{1}%
    \setlength\tabcolsep{0pt}%
    \put(0,0){\includegraphics[width=\unitlength,page=1]{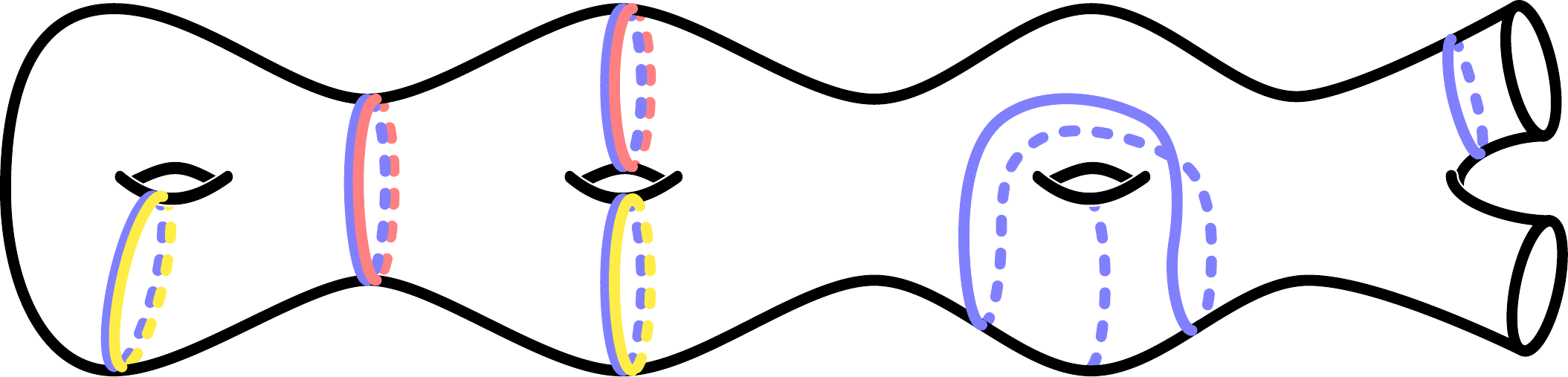}}%
    \put(0.11798036,0.06008475){\color[rgb]{0.93333333,0.83921569,0}\makebox(0,0)[lt]{\lineheight{1.25}\smash{\begin{tabular}[t]{l}$S_1$\end{tabular}}}}%
    \put(0.77579938,0.0113043){\color[rgb]{0.50196078,0.50196078,1}\makebox(0,0)[lt]{\lineheight{1.25}\smash{\begin{tabular}[t]{l}$S_2$\end{tabular}}}}%
    \put(0.41832605,0.16169025){\color[rgb]{1,0.50196078,0.50196078}\makebox(0,0)[lt]{\lineheight{1.25}\smash{\begin{tabular}[t]{l}$S_3$\end{tabular}}}}%
    \put(0,0){\includegraphics[width=\unitlength,page=2]{Cut_g.pdf}}%
  \end{picture}%
\endgroup%

    \caption{Three cut-systems in $\Cut^1(\Sigma_{3,2})$ such that $S_1 \subset S_2 \supset S_3$, but $S_1$ and $S_3$ are not related.}
    \label{fig:Cut_g}
\end{figure}

\begin{example}
    The cut systems in \cref{fig:Cut_g} are in $\Cut^1(\Sigma_{3,2})$ but only $S_2$ is in $\Cut^0(\Sigma_{3,2})$.
\end{example}

This topological poset appears naturally when computing the genus $g$ approximation of $\Mcal$.
\begin{lem}\label{lem:Cutg-as-fiber}
    For all $g, g', k \in \Nbb$ there are equivalences
    \[
        \Mcal^{(g)}(\frc_k^{(g')}) \simeq |\Cut^g(\Sigma_{g,k})| \sslash \Diff(\Sigma_{g,k})
    \]
    and thus there are fiber sequences
    \[
        |\Cut^g(\Sigma_{g,k})| \too \Mcal^{(g)}(\frc_k^{(g')}) \too \Mcal(\frc_k^{(g')}).
    \]
\end{lem}
\begin{proof}[Proof idea]
    We can compute $\Mcal^{(g)}$ as a left Kan extension along $i_g^{g'}$.
    The pointwise formula for left Kan extension gives
    \[
        \Mcal^{(g)}(\frc_k^{(g')}) 
        = \colim_{\Gamma \in \gGr^{\le g}_{/\frc_k^{(g')}}} \Mcal(\Gamma).
    \]
    Here the colimit runs over the category of maps $\Gamma \to \frc_k^{(g')}$ in $\gGr$ where $\Gamma \in \gGr^{\le g}$.
    By definition $\Mcal(\Gamma)$ is the groupoid of triples $(W,S,\alpha)$ where $W$ is a surface, $S \subset W$ is a cut system, and $\alpha\colon \Delta_{S\subset W} \cong \Gamma$ is a (grading preserving) graph isomorphism.
    We know from the dual graph that $W$ must be diffeomorphic to $\Sigma_{g',k}$, so we can write
    \[
        \Mcal(\Gamma) \simeq \{(S,\alpha) \;|\; S \subset \Sigma_{g,k},\, \alpha\colon \Delta_{S \subset W} \cong \Gamma\} \sslash \Diff(\Sigma_{g,k}).
    \]
    Inserting this in the colimit and doing some rewriting we get the desired result.
    (At some point we need to get rid of non-essential spheres, but those do not change the homotopy type of $\Cut^g(\Sigma_{g',k})$ by a cofinality argument.)
\end{proof}

For example, if $g' \le g$, then the empty cut system is allowed in $\Cut^g(\Sigma_{g',k})$, and it is an initial element in this poset.
Therefore, the realization of the poset is contractible, and we get that
    \[
        \Mcal^{(g)}(\frc_k^{(g')}) \simeq |\Cut^g(\Sigma_{g,k})| \sslash \Diff(\Sigma_{g,k})
        \simeq * \sslash \Diff(\Sigma_{g,k})
        \simeq \Mcal(\frc_k^{(g')})
    \]
as expected, since $\Ind_g^\infty$ is fully faithful and does not change the value on graphs that are bounded by $g$.

\subsubsection{The curve complex.}
The first interesting case is that of $g' = g+1$.
For simplicity, let us assume that $k=0$ and $g \ge 2$.
Then every \emph{non-empty} system of essential curves is allowed in $\Cut^g(\Sigma_{g+1})$, and thus we can identify this topological poset with the poset of simplices of the curve complex:
\[
    |\Cut^g(\Sigma_{g+1})| \simeq C(\Sigma)
\]
The \hldef{curve complex} $C(\Sigma)$ is the (discrete) simplicial complex where vertices are isotopy classes essential curves in $\Sigma$ and where $n+1$ vertices form an $n$-simplex if the essential curves can be made disjoint.

\begin{thm}[Harer, Ivanov, \cite{Harer1986,Ivanov1987}]\label{thm:curve complex}
    For $g \ge 2$, the curve complex $C(\Sigma_g)$ is equivalent to an infinite wedge of spheres of dimension $2g-2$.
    Moreover, $H_{2g-2}(C(\Sigma_g))$ is a virtual dualizing complex of dimension $4g-4$ for the mapping class group $\Gamma_g \coloneq \pi_0\Diff(\Sigma_g)$, meaning in particular that there are isomorphisms in rational homology
    \[
        H_k\left( C(\Sigma_g) \sslash \Gamma_g; \Qbb \right) 
        \cong H^{6g-6-k}(B\Gamma_g; \Qbb)
    \]
\end{thm}

\subsubsection{Convergence.}
Using \cref{thm:curve complex} one can show, with a bit more work, that the map from the genus $g$ approximation of $\Mcal$ to $\Mcal$ is highly connected
\begin{cor}
    The map $\Mcal^{(g)}(\Gamma) \to \Mcal(\Gamma)$ is $(2g-1)$-connected.
\end{cor}
\begin{proof}[Proof idea]
    It will suffice to show this claim for the map $\Mcal^{(g)}(\Gamma) \to \Mcal^{(g+1)}(\Gamma)$, as the map to $\Mcal(\Gamma)$ is the transfinite composite of such maps and the connectivity only increases.
    Moreover, it will suffice to show this when $\Gamma \in \gGr^{\le g+1}$, as both $\Mcal^{(g)}$ and $\Mcal^{(g+1)}$ are left Kan extended from this category and $(2g-1)$-connected maps are stable under colimits in the arrow category $\Ar(\Scal)$.
    As both modular operads have the same space of objects, the Segal condition reduces the claim to the value at $\frc_k^{(g')}$ for $0 \le g' \le g+1$.
    But the modular operads also agree in genus $<g+1$, so we really just need to show that the map 
    \[
        \Mcal^{(g)}(\frc_k^{(g+1)}) \too \Mcal^{(g+1)}(\frc_k^{(g+1)}) = \Mcal(\frc_k^{(g+1)}) = B\Diff(\Sigma_{g+1,k})
    \]
    has $(2g-1)$-connected fibers.
    When $k=0$ this follows from lemma \cref{lem:Cutg-as-fiber} since the curve complex $C(\Sigma_{g+1})$ is equivalent to $\bigvee_\infty S^{2g}$ by Harrer and Ivanov's theorem (\cref{thm:curve complex}), which is indeed $(2g-1)$-connected.
    In the case of $k>0$ we don't quite get the curve complex, but the fiber is still $(2g-1)$-connected, though we will not show this here.
\end{proof}

\subsubsection{Modular algebras in \texorpdfstring{$(n,1)$}{(n,1)}-categories.}
We can summarize this by saying that
\[
    \text{``the genus filtration converges with slope $2$.''}
\]
A concrete consequence of this is that when considering modular $\Mcal$-algebras in an $(n,1)$-category it suffices to consider operations of genus $\le n/2$.
Recall that a symmetric monoidal $(n,1)$-category is a symmetric monoidal \category{} $\Ccal \in \SM$ such that for any two objects $x, y \in \Ccal$ and the mapping space $\Map_\Ccal(x,y)$ is $(n-1)$-truncated, meaning that $\pi_k(\Map_\Ccal(x,y), f) = 0$ for all $k>n-1$ and $f\colon x \to y$.
\begin{cor}\label{cor:truncated-algebras}
    If $\Ccal$ is a symmetric monoidal $(n,1)$-category and $2g \ge n$, then the restriction map
    \[
        \Alg_\Mcal(\Ccal) \too \Alg_{\Mcal^{(g)}}(\Ccal)
    \]
    is an equivalence.
\end{cor}
\begin{proof}
    By assumption on $\Ccal$ the spaces of operations
    \[
        \Udual(\Ccal)(x_1,\dots, x_n) \simeq \Map_\Ccal(x_1 \otimes \dots \otimes x_n, \unit)
    \]
    are $(n-1)$-truncated.
    For the purpose of this proof we will pretend that this means that, via the Segal condition, $\Udual(\Ccal)(\Gamma)$ is $(n-1)$-truncated for all $\Gamma \in \Gr$.
    This is not quite true as $\Udual(\Ccal)(\fre) = (\Ccal^\dbl)^\simeq$ is an $n$-groupoid, but we can effectively ignore the space of objects (or change it using a DK-equivalence) as $\Mcal^{(g)}(\fre) \to \Mcal(\fre)$ is an equivalence.
    Up to this caveat the claim now follows because $\Mcal^{(g)} \to \Mcal$ is a pointwise $(2g-1)$-connected map in $\Fun(\Gr, \An)$ and $\Udual(\Ccal)$  is (basically) a $(n-1)$-truncated object in this category.
    Thus, as long as $2g-1 \ge n-1$ every map $\Mcal^{(g)} \to \Udual(\Ccal)$ uniquely extends to $\Mcal$.
\end{proof}

\begin{example}
    If $\Ccal = \LinCat_k$ is the $(2,1)$-category of $k$-linear categories (with the caveats from \cref{rem:LinCat-caveat}), then for symmetric monoidal functors $\Bord_2 \to \LinCat_k$ we get an equivalence
    \[
        \Fun^\otimes(\Bord_2, \LinCat_k) = \Alg_\Mcal(\LinCat_k) \xtoo{\simeq} \Alg_{\Mcal^{(1)}}(\LinCat_k).
    \]
    In other words, a modular functor is uniquely determined by what it does in genus $\le 1$,
    and to define a modular functor we only need to give a map of genus $\le 1$ restricted modular operads.
\end{example}

\begin{rem}
    For $n=1$ \cref{cor:truncated-algebras} only tells us that for every symmetric monoidal $1$-category $\Ccal$ the restriction
    \[
        \Fun^\otimes(\Bord_2, \Ccal) = \Alg_\Mcal(\Ccal) \xtoo{\simeq} \Alg_{\Mcal^{(1)}}(\Ccal)
    \]
    is an equivalence, but we know from \cref{thm:folk-2D-TQFT} that in fact even the restriction to genus $0$ uniquely determines a $2$D TFT.
    The issue here is that the map of modular operads
    \[
        \Hcal = \Mcal^{(0)} \too \Mcal
    \]
    is a bijection on $\pi_0$, but it is only $(-1)$-connected and not $0$-connected.
    For it to be $0$-connected it would have to induce a surjection on $\pi_1$, but in fact the mapping class groups of handlebodies \emph{inject} into the mapping class groups of surfaces.
    Presumably, this indicates that if in \cref{cor:truncated-algebras} we were to work with the notion of ``becomes an equivalence after $(n-1)$-truncation'' rather than the notion of ``$(n-1)$-connected map'', then we might obtain a slightly better bound.
\end{rem}

\subsection{Extending from genus \texorpdfstring{$g$ to $g+1$}{g to g+1}}

In the previous section we set up a convergent filtration of $\Mcal$ so that when studying $\Mcal$-algebras in $\Ccal$ we get a convergent tower
\[
    \Fun^\otimes(\Bord_2, \Ccal) = \Algmod_\Mcal(\Ccal) \too \dots 
    \too \Algmod_{\Mcal^{(2)}}(\Ccal) \too \Algmod_{\Mcal^{(1)}}(\Ccal) \too \Algmod_{\Mcal^{(0)}}(\Ccal) .
\]
Our goal now will be to understand how to step up this tower one genus at a time.
For example, we observed that (underived) modular functors are uniquely determined in genus $\le 1$, and as we already understand cyclic algebras it suffices to study how to extend from genus $0$ to genus $1$.
The key idea will be to think of the value of $\Mcal$ in genus $1$ as a right module over the operad that we obtain by restricting $\Mcal$ to genus $0$.
As before, it will be easier to first study this problem one (or in fact two) category levels up, that is, to study the restriction functor 
\[
    (i_{g-1}^g)^\ast\colon \grModOp^{\le g} \to \grModOp^{g-1}.
\]

\subsubsection{Genus \texorpdfstring{$g$}{g} is a right module over genus \texorpdfstring{$0$}{0}.}
For a (genus restricted) modular operad $\Ocal$, let $\Ocal_0 \coloneq \ul{\Ocal}_{|g\le 0}$ denote the genus $0$ part of $\Ocal$, thought of as an operad, rather than a cyclic operad.
Then there is a functor
\begin{align*}
    \grModOp^{\le g} &\too \RMod \\
    \Ocal & \longmapsto (\Ocal_0, \Ocal(\frc_\bullet^{(g)}))
\end{align*}
that records all the genus $g$ operations of the modular operad as a right module over the genus $0$ operations.
To construct this functor we can use that $\Ocal$ forgets to a cyclic operad, which then induces a right module over itself (via \cref{prop:Seg-rMod}) as $(\ul{\Ocal}, \omega_\Ocal) \in \RMod^\dual$.
Then we restrict to the genus $0$ part of $\ul{\Ocal}$ and the genus $g$ part of $\omega_\Ocal$.
Of course the resulting right module $\Ocal(\frc_\bullet^{(k)})$ is no longer a dualizing module.
\begin{example}
    In the case of $\Mcal$ we have $\Mcal_0 = \EtwoSO$, and we can describe the right module $\Mcal(\frc_\bullet^{(g)})$ as the disk presheaf
    \[
        \left(  \Mcal(\frc_\bullet^{(g)}) \colon \sqcup_k D^2 \longmapsto \Emb(\sqcup_k D^2, \Sigma_g) \sslash \Diff(\Sigma_g) \right) \in \PSh(\Disk_2)
    \]
    or more concisely as $E_{\Sigma_g} \sslash \Diff(\Sigma_g)$, using the notation from \cref{ex:E_M}.
    Note that in arity $k$ this indeed gives%
    \footnote{
        To see this, consider the Palais fiber sequence
        \[
            \Diff_{\partial}(\Sigma_{g,k}) \simeq \Diff_{\sqcup_k D^2}(\Sigma_k) 
            \too \Diff(\Sigma_k) \too \Emb(\sqcup_k D^2, \Sigma_k)
        \]
        and deloop it / apply the homotopical orbit-stabiliser lemma \cite[Lemma 2.11]{BBS-finiteness}.
    }
    \[
        \Emb(\sqcup_k D^2, \Sigma_g)\sslash \Diff(\Sigma_g) \simeq B\Diff_\partial(\Sigma_{g,k}).
    \]
\end{example}

Concatenating this with the induction functor $\Ind_{g-1}^g$ we obtain the following functor.
\begin{defn}
    For $g\ge 1$ the \hldef{$g$ latching module} of a genus $\le (g-1)$ restricted modular operad $\Ocal \in \grModOp^{\le g-1}$ is defined as
    \[
        \hldef{L_g\Ocal} \coloneq (\Ind_{g-1}^g\Ocal)(\frc_\bullet^{(g)}) \in \RMod_{\Ocal_0}
    \]
    which assembles into a functor
    \begin{align*}
        \grModOp^{\le g-1} &\too \RMod \\
        \Ocal & \longmapsto (\Ocal_0, L_g\Ocal).
    \end{align*}
\end{defn}

\begin{example}
    In the case of $\Mcal$ we compute the $g$ latching module $L_g\Mcal$ by using \cref{lem:Cutg-as-fiber} and manipulating colimits.
    (We only do this rewriting informally, a more formal argument would need to be more careful about the right module structure.)
    \begin{align*}
        L_g\Mcal(\sqcup_k D^2) 
        & \simeq |\Cut^{g-1}(\Sigma_{g,k})| \sslash \Diff_\partial(\Sigma_{g,k}) \\
        & \simeq \left( \colim_{i\colon \sqcup_k D^2 \hookrightarrow \Sigma_g} |\Cut^{g-1}(\Sigma_g\setminus i(\sqcup_k D^2))| \right) \sslash \Diff(\Sigma_{g}) \\
        & \simeq \left( \colim_{S \in \Cut^{g-1}(\Sigma_g)^\op} \Emb(\sqcup_k D^2, \Sigma_g\setminus S)\right) \sslash \Diff(\Sigma_{g}) 
    \end{align*}
    So we can write the latching module as 
    \[
        L_g\Mcal \simeq  \left( \colim_{S \in \Cut^{g-1}(\Sigma_g)^\op} E_{\Sigma_g \setminus S}\right) \sslash \Diff(\Sigma_{g}) .
    \]
\end{example}

\subsubsection{Main theorem.}
We now have all the tools at hand to describe the fiber of the restriction functor.
\begin{thm}[\cite{genus}]\label{thm:genus-extension}
    For every genus $\le g-1$ restricted modular operad $\Ocal$ there is an equivalence
    \[
        \fib_\Ocal\left( \grModOp^{\le g} \too \grModOp^{\le g-1} \right)
        \simeq (\RMod_{\Ocal_0})_{L_g/}
    \]
    induced by the functors described above.
\end{thm}

The restriction functor
    $\grModOp^{\le g} \too \grModOp^{\le g-1}$
is a cocartesian fibration and the above theorem describes its fibers one at a time.
A slightly more sophisticated version of the theorem describes the entire fibration by showing that it fits into a cartesian square
\[\begin{tikzcd}[column sep = 45]
    \grModOp^{\le g} \ar[d] \ar[r, "{L_g \to \Ocal(\frc_\bullet^{(g)})}"] \ar[dr, phantom, very near start, "\lrcorner"] &
    \Ar^{\rm vert}(\RMod) \ar[d, "s"] \\
    \grModOp^{\le g-1} \ar[r, "L_g"] & \RMod
\end{tikzcd}\]
where $\Ar^{\rm vert}(\RMod) = \Op \times_{\Ar(\Op)} \Ar(\RMod)$ is the vertical arrow category, i.e.~it is the full subcategory of $\Ar(\RMod)$ on those morphisms that map to equivalences in $\Op$.

Let's apply this to modular $\Mcal$ algebras in a symmetric monoidal \category{} $\Vcal \in \SM$.
Fix a modular $\Mcal^{(g-1)}$-algebra $A \in \Algmod_{\Mcal^{(g-1)}}(\Vcal)$ in $\Vcal$.
Then the space of lifts of $A$ to a modular $\Mcal^{(g)}$ algebra is the space of extensions
\[
    {\left\{ \begin{tikzcd}
        L_g \Mcal \ar[dr, "A"] \ar[d] & \\
        \Mcal(\frc_\bullet^{(g)}) \ar[r, dashed] & \Udual(\Ccal)
    \end{tikzcd} 
    \in \RMod_{\Mcal_0}
    \right\}}
    \simeq 
    {\left\{ \begin{tikzcd}
        \colim\limits_{S \in \Cut^{g-1}(\Sigma_g)^\op} E_{\Sigma_g \setminus S} \ar[dr, "A"] \ar[d] & \\
        E_{\Sigma_g} \ar[r, dashed] & \Udual(\Ccal)
    \end{tikzcd} 
    \in \PSh(\Disk_2)^{B\Diff(\Sigma_g)}
    \right\}} 
\]
If $\Vcal$ has sufficient colimits and the tensor product preserves them, then we can further rewrite this as an extension problem in the \category{} of $\Diff(\Sigma_k)$-representations in $\Ccal$
where we equip $\unit_\Ccal$ with the trivial action.
\begin{cor}\label{cor:space-of-lifts}
    For $\Vcal$ a symmetric monoidal \category{} with sifted colimits, which are preserved by the tensor product, the space of lifts of a modular $\Mcal^{(g-1)}$-algebra $A$ to a modular $\Mcal^{(g)}$-algebra is 
    \[ 
        \fib_A
        {\left( \begin{tikzcd}
        \Alg_{\Mcal^{(g)}}(\Vcal) \ar[d] \\
        \Alg_{\Mcal^{(g-1)}}(\Vcal)
        \end{tikzcd} 
        \right)}
        \simeq 
        {\left\{ \begin{tikzcd}
            \colim\limits_{S \in \Cut^{g-1}(\Sigma_g)^\op} \int_{\Sigma_g \setminus S} A \ar[dr] \ar[d] & \\
            \int_{\Sigma_g} A \ar[r, dashed] & \unit_\Ccal
        \end{tikzcd} 
        \in \Ccal^{B\Diff(\Sigma_g)}
        \right\}}
    \]
\end{cor}

\begin{rem}
    While this might seem like quite a complicated expression, it is useful because it ``linearizes'' the problem.
    We started out with the task of defining a (modular) algebra and now this task has been split up into giving an algebra over the ordinary operad $\EtwoSO$ and a sequence of maps of right modules over $\EtwoSO$.
    It would be interesting to see whether one can use this to, for example, study the deformation theory of $2$D TFTs.
\end{rem}

\subsection{Field theoretic interpretation}

\subsubsection{Correlation functions.}
We now give an interpretation of the inductive description of modular $\Mcal$-algebras in \cref{thm:genus-extension} in more field theoretic terms.
Suppose we are given a symmetric monoidal functor
\[
    \Zcal \colon \Bord_2 \too \Vcal
\]
where $\Vcal$ is a symmetric monoidal category with suitable colimits.
The value at $S^1$ is an $\EtwoSO$-algebra
\[
    A \coloneq \Zcal(S^1) \in \Alg_{\EtwoSO}(\Vcal),
\]
which is referred to as the algebra of \hldef{local operators} or point operators \cite[Remark 2.35]{FreedHopkins}.
This can be thought of as the space of possible fields inside a small disk, as observed through the boundary $S^1$ of that disk.
The algebra structure allows us to combine such operators.

Suppose we now fix a closed genus $g$ surface $\Sigma_g$.
Then for any configuration of $k$ framed points on $\Sigma_k$, given by an embedding $\sqcup_k D^2 \hookrightarrow \Sigma_k$, we can construct a bordism
$\Sigma_g \setminus \sqcup_k D^2 \colon \sqcup_k S^1 \too \emptyset$
and applying $\Zcal$ to this bordism yields a map
\[
    \Zcal(\Sigma_g \setminus \sqcup_k D^2) \colon A^{\otimes k} = \Zcal(\sqcup_k S^1) \too \Zcal(\emptyset) = \unit.
\]

\begin{figure}[ht]
    \centering
    \def\svgwidth{.99\linewidth}
\begingroup%
  \makeatletter%
  \providecommand\color[2][]{%
    \errmessage{(Inkscape) Color is used for the text in Inkscape, but the package 'color.sty' is not loaded}%
    \renewcommand\color[2][]{}%
  }%
  \providecommand\transparent[1]{%
    \errmessage{(Inkscape) Transparency is used (non-zero) for the text in Inkscape, but the package 'transparent.sty' is not loaded}%
    \renewcommand\transparent[1]{}%
  }%
  \providecommand\rotatebox[2]{#2}%
  \newcommand*\fsize{\dimexpr\f@size pt\relax}%
  \newcommand*\lineheight[1]{\fontsize{\fsize}{#1\fsize}\selectfont}%
  \ifx\svgwidth\undefined%
    \setlength{\unitlength}{2137.79619833bp}%
    \ifx\svgscale\undefined%
      \relax%
    \else%
      \setlength{\unitlength}{\unitlength * \real{\svgscale}}%
    \fi%
  \else%
    \setlength{\unitlength}{\svgwidth}%
  \fi%
  \global\let\svgwidth\undefined%
  \global\let\svgscale\undefined%
  \makeatother%
  \begin{picture}(1,0.17554589)%
    \lineheight{1}%
    \setlength\tabcolsep{0pt}%
    \put(0,0){\includegraphics[width=\unitlength,page=1]{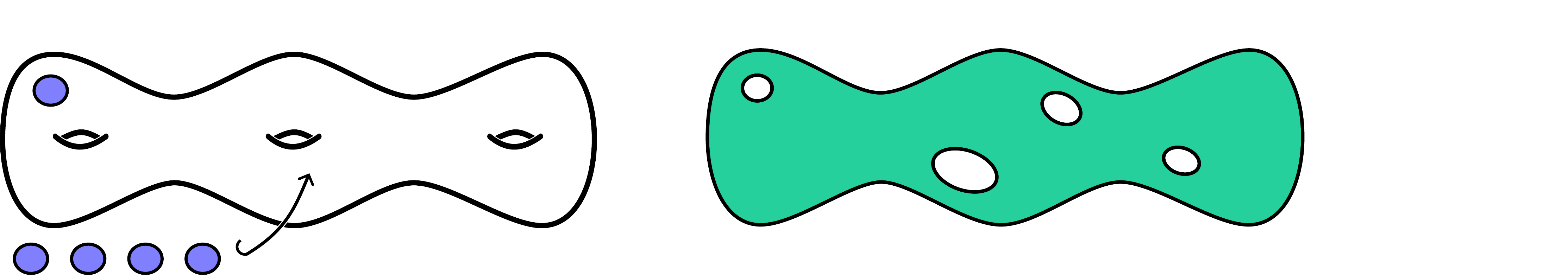}}%
    \put(0.22179808,0.00166735){\color[rgb]{0,0,0}\makebox(0,0)[lt]{\lineheight{1.25}\smash{\begin{tabular}[t]{l}$\Emb (\sqcup_k D^2, \Sigma_g)$\end{tabular}}}}%
    \put(0.58922362,0.00166735){\color[rgb]{0,0,0}\makebox(0,0)[lt]{\lineheight{1.25}\smash{\begin{tabular}[t]{l}$\Map_{\Bord}(\sqcup_k S^1, \emptyset)$\end{tabular}}}}%
    \put(0,0){\includegraphics[width=\unitlength,page=2]{correlation-complement.pdf}}%
    \put(0.90410768,0.00662985){\color[rgb]{0,0,0}\makebox(0,0)[lt]{\lineheight{1.25}\smash{\begin{tabular}[t]{l}$\emptyset$\end{tabular}}}}%
    \put(0.92659553,0.0716057){\color[rgb]{0,0,0}\makebox(0,0)[lt]{\lineheight{1.25}\smash{\begin{tabular}[t]{l}$\Sigma_g \setminus i(\sqcup_k D^2)$\end{tabular}}}}%
    \put(0.17359458,0.00774902){\color[rgb]{0,0,0}\makebox(0,0)[lt]{\lineheight{1.25}\smash{\begin{tabular}[t]{l}$i$\end{tabular}}}}%
  \end{picture}%
\endgroup%

    \caption{The complement of an embedding of $k$ disks into $\Sigma$ defines a bordism from $k$ circles to the empty manifold.}
    \label{fig:correlation-complement}
\end{figure}

When $\Vcal = \Vect$ this sends a tensor product of operators to a number, and this corresponds to ``inserting'' the operators at the points on $\Sigma_g$ prescribed by the disks, and then ``measuring'' the ``cross-section'' of the resulting interaction.
We can thus think of $\Zcal(\Sigma_g \setminus \sqcup_k D^2)$ as the \hldef{correlation function} of the operators \cite[Remark 2.36]{FreedHopkins}.
When two operators are inserted nearby we can instead merge them using the algebra multiplication on $A = \Zcal(S^1)$ and as we can do this for arbitrary configurations, the correlation functions at various points assemble into a map
\[
    \Zcal_g\colon \int_{\Sigma_g} A \too \unit,
\]
which by nature of its construction will also be $\Diff(\Sigma_g)$-invariant.
We call this the \hldef{universal genus $g$ correlation function}.

\subsubsection{An inductive description.}
Suppose we have a TFT that is only defined in genus $\le g-1$.
This can be made precise by saying that it is a modular algebra over $\Mcal^{(g-1)}$ or equivalently, a map of genus-restricted modular operads $\Mcal_{|\le g-1} \to \Udual(\Vcal)$, but we'll think of it as a kind of partially-defined symmetric monoidal functor
\(
    \Zcal \colon \Bord_2 \nrightarrow \Vcal
\)
that is only defined on surfaces whose path components are of genus $\le g-1$.%
\footnote{
    In fact, it's possible to make this notion precise, but comparing it to genus-restricted modular operads seems non-trivial.
}
Then, given such a genus restricted TFT, we can define a \hldef{correlation-constraint} in genus $g$
\[
    \Zcal_g^{\rm cut} \colon \int_{\Sigma_g}^{<g} A \too \unit
\]
where the restricted factorization homology is defined as a colimit over the poset of systems of essential curves
\[
    \hldef{\int_{\Sigma_g}^{<g} A} \coloneq \colim_{S \in \Cut^{<g}(\Sigma_g)^\op} \int_{\Sigma_g \setminus S} A.
\]
To make precise the construction of the correlation-constraint $\Zcal_g^{\rm cut}$ we essentially have to take the construction from \cref{cor:space-of-lifts}, but more conceptually speaking we can also describe it in terms of $\Zcal\colon \Bord_2 \nrightarrow \Vcal$.
For each system of essential curves $S \subset \Sigma_g$ 
and embedding $\sqcup_k D^2 \hookrightarrow \Sigma_g \setminus S$ the functor $\Zcal$ gives us a map 
\[
    \Zcal(\Sigma_g \setminus (S \sqcup \sqcup_k D^2)) \colon \Zcal(S^1)^{\otimes k} \too \Zcal(2S) 
\]
where $2S$ is the spherical normal bundle of $S \subset \Sigma_g$, which contains two circles for each curve in $S$.
Assembling this into a map out of factorization homology and composing with evaluation pairings $\Zcal(S^1 \sqcup S^1) \to \unit$ we get a map
\[
    \int_{\Sigma_g \setminus S} A \too A^{\otimes 2S} \xtoo{\mrm{ev}} \unit
\]
As we add or remove curves from $S$ this assembles into a map out of the colimit $\int_{\Sigma_g}^{<g} A$.

With this notation in hand, we can now give the following inductive description of $2$D TFTs.
\begin{cor}
    From a genus $(g-1)$-restricted TFT $\Zcal$ one can canonically construct a $\Diff(\Sigma_g)$-equivariant ``correlation constraint''
    \[
        \Zcal_g^{<g} \colon \int_{\Sigma_g}^{<g} A \too \unit.
    \]
    Extending $\Zcal$ to a genus $g$-restricted TFT is equivalent to constructing an extension
    \[\begin{tikzcd}[column sep = large]
        \int_{\Sigma_g}^{<g} A \ar[d] \ar[dr, "{\Zcal_g^{<g}}", bend left] & \\
        \int_{\Sigma_g} A \ar[r, dashed, "\Zcal_g"] & \unit
    \end{tikzcd}\]
    in the \category{} $\Fun(B\Diff(\Sigma_g), \Vcal)$.
\end{cor}
Thus, defining a $2$D TFT valued in $\Vcal$ is equivalent to giving a (cyclic) $\EtwoSO$-algebra in $\Vcal$ and a family of universal genus $g$ correlation functions, each compatible with the correlation constraint constructed from the lower genus data.

\begin{rem}
    While this might seem like a lot of data, the map $\int_{\Sigma_g}^{<g} A \to \int_{\Sigma_g} A$ might simply be an equivalence in situations we care about, thus trivializing the extension problem.
    By \cref{cor:truncated-algebras} this for example happens when $\Vcal$ is an $(n,1)$-category and $2g \ge n$.
\end{rem}

\begin{rem}
    Brochier--Woike give a classification of modular functors that uses factorization homology similarly to how it appears here \cite{BrochierWoike-modular}.
    They work in a $(2,1)$-categorical setting, which according to our heuristic should imply that the TFT is determined in genus $\le 1$, but they are in fact able to reduce it to the genus $0$ classification of \cite{MW23}, by allowing suitable central extensions of the mapping class groups involved.
\end{rem}

\subsection{A spectral sequence involving the homology of mapping class groups}

\subsubsection{Invertible $2$D TFTs and the GMTW theorem.}
So far we have mostly considered examples where the target category $\Ccal$ was truncated or admitted some colimits.
The classification result also turns out to be interesting when we consider the exact opposite case: invertible TFTs.
Let $X \in \Sp$ be any connective spectrum and $\Omega^\infty X$ its infinite loop space, which we can think of as a symmetric monoidal groupoid where every object is invertible.
Then invertible TFTs valued in $X$ are symmetric monoidal functors
\[
    \Bord_2 \too \Omega^\infty X.
\]
We can compute the space of such functors using 
the Galatius--Madsen--Tillmann--Weiss theorem $|\Bord_2| \simeq \Omega^\infty \Sigma \MTSO(2)$ \cite{GMTW} and the adjunction
\[
    |-|^{\rm gp} \colon \SM \adj \Omega^\infty\text{-spaces} \simeq \Sp_{\ge0} \cocolon \Omega^\infty
\]
where the left-adjoint inverts all morphisms in a symmetric monoidal \category{} and then group-completes the resulting symmetric monoidal groupoid. (This last step won't be necessary for us, as the symmetric monoidal envelope of a modular operad is always rigid and thus its groupoidification is group-like.)
Plugging in \cite{GMTW} we get
\[
    \Fun^\otimes(\Bord_2, \Omega^\infty X)
    \simeq \Map_{\Omega^\infty}(|\Bord_2|, \Omega^\infty X)
    \simeq \Map_{\Sp}(\tau_{\ge 0} \Sigma\MTSO(2), X).
\]
The spectrum $\MTSO(2)$ is a Thom spectrum, which differs from $\Sigma^\infty_+ \CP^\infty$ only by a cell in degree $-2$.
In fact, as we'll mostly be working rationally in this section, we might as well think of $\tau_{\ge 0} \Sigma \MTSO(2)$ as $\Sigma^{\infty+1} \CP^\infty_+$.
More elaborate versions of this story have been used to great effect by Freed--Hopkins \cite{FreedHopkins} to make computations relevant to condensed matter physics.
In dimension $2$ the spectrum $|\Bord_2|$ is particularly simple, and in light of this fairly succinct classification of invertible $2$D TFTs it might seem odd to apply the inductive classification result \cref{thm:genus-extension} to it.
However, it turns out that in this case the obstruction in each genus step is quite interesting in its own right.

\subsubsection{The genus filtration on \texorpdfstring{$\MTSO(2)$}{MTSO(2)}.}
From the genus filtration on $\Mcal$ we get a convergent filtration of the classifying space of the surface category as
\[
    |\Env(\Mcal^{(0)})| \too |\Env(\Mcal^{(1)})| \too |\Env(\Mcal^{(2)})| \too \dots
    \too |\Env(\Mcal)| = |\Bord_2| = \Omega^{\infty-1}\MTSO(2)
\]
which can be delooped to give a convergent filtration of $\tau_{\ge 0} \Sigma \MTSO(2)$.
(Here filtration just means a functor from $(\Nbb, \le)$ and convergent means that the colimit is as indicated.)
By applying \cref{cor:CY-algebras} and \cref{thm:genus-extension} in the case of invertible TFTs we can compute the initial term and the associated graded of this filtration.
\begin{cor}
    The $0$th step of the filtration is 
    \[
        \Sigma^{\infty+1} \BSO(3)_+ 
    \]
    and the $g$th associated graded is 
    \[
        \Sigma^\infty S\left(\frac{|\Cut^{<g}(\Sigma_g)|}{\Diff(\Sigma_g})\right)
    \]
    where $S$ denotes unreduced suspension.
\end{cor}

For $g \ge 2$ we can identify the space of cut systems with the curve complex $C(\Sigma_g)$.
Let us write $\Gamma_g \coloneq \pi_0(\Diff(\Sigma_g))$ for the genus $g$ mapping class group.
As recalled in \cref{thm:curve complex}, it is a theorem of \cite{Harer1986,Ivanov1987} that $C(\Sigma_g)$ is a virtual dualizing complex for $\Gamma_g$ and that thus the $k$th rational homology of the homotopy quotient $C(\Sigma_g)\sslash \Gamma_g$ is isomorphic to the group cohomology of $\Gamma_g$ in degree $6g-6-k$.
Assembling all this (and shifting down by one) we get the following spectral sequence from the filtered spectrum $\Sigma \tau_{\ge 0} \MTSO(2)$.

\begin{figure}[p]
    \centering
    \input{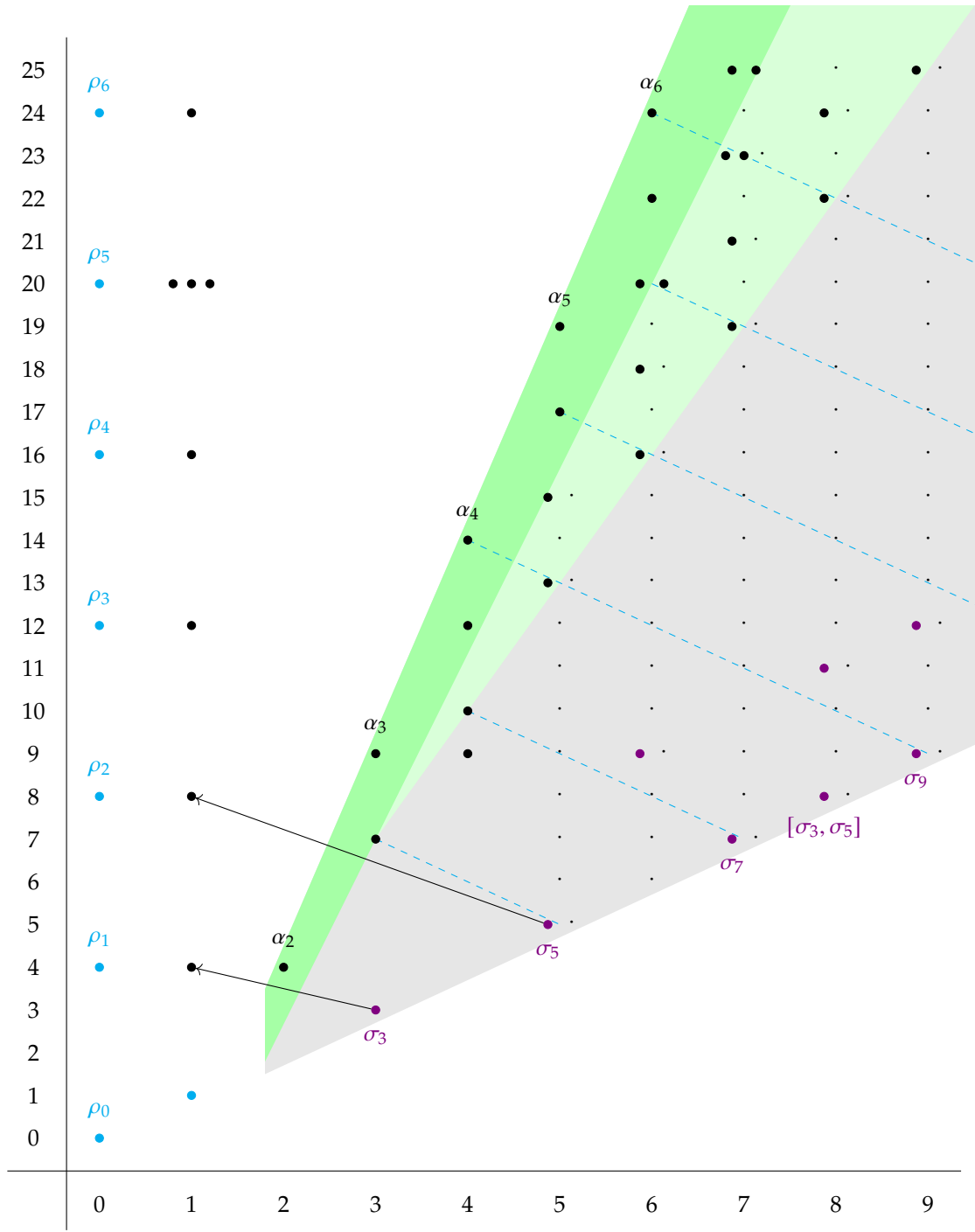}
    \printpage[ name = genus-spseq ]
    \caption{The genus spectral sequence. "$\bullet$" indicates a known non-zero class and "$\cdot$" indicates a group of unknown size. For example, $E_{5,5}^1$ is at least of dimension $1$. See \cref{ex:spectral-sequence}.}
    \label{fig:genus-spectral-sequence}
\end{figure}

\begin{cor}\label{cor:spectral-sequence}
    There is a convergent spectral sequence in rational vector spaces
    \[
        E^1_{g,k} = \begin{cases}
            H_k(\BSO(3); \Qbb) & \text{ for } g=0 \\
            H_{k-1}(B\Diff(S^1 \times S^1)/B\Diff(S^1 \times D^2); \Qbb) & \text{ for } g=1 \\
            H^{5g-6-k}(B\Gamma_g; \Qbb) & \text{ for } g\ge2 
        \end{cases}
        \Rightarrow
        H_{g+k}(\CP^\infty).
    \] 
\end{cor}
Note that the cohomology of $\CP^\infty$ has dimension $1$ in even degrees and is $0$ in odd degrees.
The classes in degrees $0$ mod $4$ are accounted for by the $0$th column, which survives to the $E^\infty$-page.
Thus, there must be exactly one class in each degree $4i+2$ that survives to the $E^\infty$-page and every other class must cancel in some way.
See \cref{fig:genus-spectral-sequence} for a picture of the $E^1$ page, to the extent that $H^*(B\Gamma_g; \Qbb)$ is known.

\begin{rem}\label{ex:spectral-sequence}
    The picture of the $E^1$-page of the genus spectral sequence in \cref{fig:genus-spectral-sequence} contains information from several sources.
    The computation in genus $g=3,4$ is due to Looijenga \cite{Looijenga-M3} and Tommasi \cite{Tommasi-M4}.
    In the dark green region we know homological stability and thus by \cite{MadsenWeiss} the cohomology is a polynomial ring on the Miller--Morita--Mumford classes $\kappa_i$.
    These classes generate the ``tautological ring'', which is known to have interesting structure, see e.g.~\cite{Faber1999}.
    The purple classes are the top weight classes found by \cite{CGP-tropical}.
    Cyan classes indicate classes that must survive to the $E^\infty$-page, the dashed cyan lines indicate that there must be exactly one class on them that survives to the $E^\infty$ page.
    The $d_2$ and $d_4$ differential can be deduced from the map of spectral sequences obtained by considering the same filtration for the terminal graded modular operad $\gr$, see \cref{rem:gr-spectral-sequence}.
\end{rem}

\begin{example}
    It follows from the spectral sequence \cref{fig:genus-spectral-sequence} that $H_{14}(B\Gamma_5)$ has dimension $1$ or $2$, as only one of the classes can support a ($d_4$) differential, and $E_{5,5}^\infty$ must be at most one-dimensional.
    This group contains a non-trivial class by \cite{CGP-tropical}.
\end{example}

\begin{rem}\label{rem:gr-spectral-sequence}
    If one considers the genus filtration for the grading modular operad $\gr$ instead of $\Mcal$, then the resulting spectral sequence has on its $E_1$-page the homologies of tropical moduli spaces $\Delta_g$, and it converges to the homology of $|\Env(\gr)| \simeq S^1$ (by \cite[Theorem B]{Jan-tropical}).
    The $E_1$ and $E_\infty$ page of this spectral sequence are isomorphic to the respective pages of the spectral sequence described in \cite{KWZ17-differentials-on-graph-complexes} (taking $d$ to be even and regrading appropriately).
    It seems reasonable to suspect that the spectral sequences are isomorphic.
    Note that the map $\Mcal \to \gr$ induces a map of spectral sequences, which can be used to detect non-trivial differentials in the spectral sequence from \cref{cor:spectral-sequence}.
\end{rem}

\begin{rem}
    This filtration and spectral sequence look conceptually similar to Rognes' stable rank filtration and the spectral sequences in the recent work of Brown--Chan--Galatius--Payne \cite{BCGP-Hopf}.
    In particular, they also encounter a spectral sequence that looks isomorphic to the one in \cite{KWZ17-differentials-on-graph-complexes}, see \cite[\S5.5.1]{BCGP-Hopf}.
    However, the precise connection between \cref{cor:spectral-sequence} and \cite{BCGP-Hopf} remains unclear.
\end{rem}

\printbibliography[heading=bibintoc]

\end{document}